%% file: journal.tex
\documentclass[11pt]{article}

\input{zh-preamble}

\begin{document}
    \ifarxivVersion
    \title{\textbf{Decision-Dependent Stochastic Optimization: \\ The Role of Distribution Dynamics}}
        \author{Zhiyu He\textsuperscript{*,$\dagger$}, Saverio Bolognani\textsuperscript{$\dagger$}, Florian D{\"o}rfler\textsuperscript{$\dagger$}, Michael Muehlebach\textsuperscript{*}\\
        \date{}
        }
    \else
    \title{Decision-Dependent Stochastic Optimization: \\ The Role of Distribution Dynamics}
    \author{\name Zhiyu He \email zhiyhe@ethz.ch \\
           \addr Automatic Control Laboratory, ETH Z{\"u}rich \\
           8092 Z{\"u}rich, Switzerland
           \AND
           \name Saverio Bolognani \email bsaverio@ethz.ch \\
           \addr Automatic Control Laboratory, ETH Z{\"u}rich \\
           8092 Z{\"u}rich, Switzerland
           \AND
           \name Florian D{\"o}rfler \email dorfler@ethz.ch \\
           \addr Automatic Control Laboratory, ETH Z{\"u}rich \\
           8092 Z{\"u}rich, Switzerland
           \AND
           \name Michael Muehlebach \email michael.muehlebach@tuebingen.mpg.de \\
           \addr Max Planck Institute for Intelligent Systems \\
           72076 T{\"u}bingen, Germany
           }

    \editor{Sebastian Stich}
    \fi

    \maketitle
    \ifarxivVersion
        
        \def\thefootnote{}\footnotetext{$^*$Max Planck Institute for Intelligent Systems, Germany (\{zhiyu.he, michael.muehlebach\}@tuebingen.mpg.de). $^\dagger$Automatic Control Laboratory, ETH Z{\"u}rich, Switzerland (\{zhiyhe, bsaverio, dorfler\}@ethz.ch).}
        \setcounter{footnote}{0}
        \def\thefootnote{\arabic{footnote}}
    \fi

    \begin{abstract}
        Distribution shifts have long been regarded as troublesome external forces that a decision-maker should either counteract or conform to. An intriguing feedback phenomenon termed \emph{decision dependence} arises when the deployed decision affects the environment and alters the data-generating distribution. In the realm of performative prediction, this is encoded by distribution maps parameterized by decisions due to strategic behaviors. In contrast, we formalize an endogenous distribution shift as a feedback process featuring nonlinear dynamics that couple the evolving distribution with the decision. Stochastic optimization in this dynamic regime provides fertile ground to examine the various roles played by dynamics in the composite problem structure. To this end, we develop an online decision-making algorithm that both adapts to and shapes the dynamic distribution. 
        Throughout the paper, we adopt a distributional perspective and demonstrate how this view facilitates characterizations of distribution dynamics and the optimality and generalization performance of the proposed algorithm. We showcase the theoretical results in an opinion dynamics context, where an opportunistic party maximizes the affinity of a population with evolving, polarized opinions, and in a recommender system scenario, featuring performance optimization with discrete distributions in the probability simplex. 
    \end{abstract}


    \input{section/introduction}
    \input{section/formulation}
    \input{section/design}
    \input{section/analysis}
    \input{section/experiment}
    \input{section/conclusion}


    \acks{We thank the Max Planck ETH Center for Learning Systems, the Swiss National Science Foundation via NCCR Automation (51NF40\_225155), and the German Research Foundation for their support. We acknowledge N.~Lanzetti for useful insights and comments, particularly on \cref{lem:wass_dist_bd_recur}, and thank G.~De Pasquale, P.~Grontas, L.~Aolaritei, S.~Li, X.~He, and S.~Hall for helpful discussions.}


    \newpage
    \appendix
    \input{section/appendix}

    \vskip 0.2in
    \newpage
    \bibliography{article}

\end{document}

%% file: zh-preamble.tex
\usepackage{amsmath,amscd,amsbsy,amssymb,latexsym,url,bm,amsthm}
\usepackage{array}
\usepackage{algorithm}
\usepackage{algorithmic}
\usepackage{booktabs}
\usepackage{etoolbox}
\preto\subequations{\ifhmode\unskip\fi} 
\usepackage[T1]{fontenc}
\usepackage{graphicx}
\usepackage{multirow,makecell}
\usepackage{mathtools}
\usepackage{subeqnarray}
\usepackage[caption=false,font=footnotesize,subrefformat=parens,labelformat=parens]{subfig}
\usepackage{tabularx}
\usepackage{threeparttable}
\usepackage{times}
\usepackage{tikz}
\usepackage{xcolor}
\allowdisplaybreaks
\newif\ifarxivVersion
\arxivVersionfalse  

\ifarxivVersion
    \usepackage{natbib}
    \usepackage[margin=1in]{geometry}
    \usepackage[colorlinks=false,allbordercolors={1 1 1}]{hyperref}
    \usepackage[capitalize]{cleveref}
    \newtheorem{example}{Example} 
    \newtheorem{theorem}{Theorem}
    \newtheorem{lemma}{Lemma} 
    \newtheorem{proposition}[theorem]{Proposition} 
    
    \newtheorem{corollary}[theorem]{Corollary}
    
    \makeatletter
    \let\c@lemma\c@theorem
    \makeatother
    \long\def\acks#1{\vskip 0.3in\noindent{\large\bf Acknowledgments}\vskip 0.2in\noindent #1}
    \newenvironment{keywords}{\bgroup\leftskip 20pt\rightskip 20pt \small\noindent{\bf Keywords:} }%
    {\par\egroup\vskip 0.25ex}
\else
    \usepackage{jmlr2e}
    \usepackage[capitalize]{cleveref}
\fi

\newtheorem{assumption}{Assumption}

\makeatletter
\newcommand{\zh@theoremsibling}[2]{%
    \expandafter\let\csname #1\endcsname\relax
    \expandafter\let\csname end#1\endcsname\relax
    \expandafter\let\csname the#1\endcsname\relax
    \expandafter\let\csname c@#1\endcsname\relax
    \newtheorem{#1}[theorem]{#2}%
    \expandafter\let\csname c@#1\endcsname\c@theorem
}
\zh@theoremsibling{lemma}{Lemma}
\zh@theoremsibling{proposition}{Proposition}
\zh@theoremsibling{corollary}{Corollary}
\zh@theoremsibling{remark}{Remark}
\zh@theoremsibling{definition}{Definition}
\makeatother

\crefname{assumption}{Assumption}{Assumptions}
\crefname{corollary}{Corollary}{Corollaries}
\crefname{theorem}{Theorem}{Theorems}
\crefname{lemma}{Lemma}{Lemmas}
\crefname{proposition}{Proposition}{Propositions}
\crefname{remark}{Remark}{Remarks}
\crefname{definition}{Definition}{Definitions}
\crefname{example}{Example}{Examples}
\ifdefined\RemoveFromHook
    \RemoveFromHook{label}[firstaid/cleveref]
\fi

\DeclareMathOperator{\E}{\mathop{\mathbb{E}}}
\DeclareMathOperator{\proj}{Proj}
\DeclareMathOperator{\sgn}{sgn}

\newcommand{\bigO}[1]{\ensuremath{\mathop{}\mathopen{}\mathcal{O}\mathopen{}\left(#1\right)}}
\renewcommand{\bar}[1]{\mkern 1.5mu\overline{\mkern-1.5mu#1\mkern-1.5mu}\mkern 1.5mu}
\newcommand{\barlambda}{\bar{\lambda}}
\newcommand{\ud}{\,\mathrm{d}}
\newcommand{\PR}[1]{\Pr\left(#1\right)}
\newcommand{\RNum}[1]{\uppercase\expandafter{\romannumeral #1\relax}}
\newcommand{\numcircled}[1]{\tikz[baseline=(char.base)]{
            \node[shape=circle,draw,inner sep=0.6pt] (char) {#1};}} 
\newcommand{\obj}{\tilde{\Phi}}
\newcommand{\Lobj}{L_{\obj}}
\newcommand{\nablabatch}[1]{\widehat{\nabla}_{\textup{mb}}^{#1}}
\newcommand{\widenabla}{\widehat{\nabla}^k}
\newcommand{\gammass}{\gamma_\textup{ss}}
\newcommand{\muss}{\mu_\textup{ss}}
\newcommand{\pss}{p_\textup{ss}}
\newcommand{\pfix}{p_\textup{fix}}
\newcommand{\hssmap}{h_\textup{ss}}
\newcommand{\var}{\operatorname{Var}}
\newcommand{\nmb}{n_{\textup{mb}}}
\newcommand{\sigmamb}{\sigma_{\textup{mb}}}
\newcommand{\sigmambbar}{\bar{\sigma}_{\textup{mb}}}
\newcommand{\EXPT}[1]{\E\left[#1\right]}
\newcommand{\calF}{\mathcal{F}}
\newcommand{\id}{\mathop{}\mathopen{}\mathrm{Id}}
\newcommand{\softmax}{\operatorname{softmax}}

\newcommand{\rev}[1]{{\color{black} #1}}
\newcommand{\revise}[1]{{\color{black} #1}}

\newcolumntype{Y}{>{\centering\arraybackslash}X}

\makeatletter
\def\raisedotfill{%
  \leavevmode
  \cleaders \hb@xt@ .44em{\hss\raise0.5ex\hbox{.}\hss}\hfill
  \kern\z@}
\makeatother

\ifarxivVersion
\else

\usepackage{lastpage}
\jmlrheading{27}{2026}{1-\pageref{LastPage}}{3/25; Revised
9/26}{10/26}{25-0514}{He, Bolognani, D{\"o}rfler, Muehlebach}

\ShortHeadings{Decision-Dependent Stochastic Optimization}{He, Bolognani, D{\"o}rfler, Muehlebach}
\firstpageno{1}
\fi

\makeatletter
\pretocmd{\bibitem}{%
  \color{black}%
  \@ifundefined{keycolor#2}{}{\csname keycolor#2\endcsname}%
}{}{}
\pretocmd{\@lbibitem}{%
  \color{black}%
  \@ifundefined{keycolor#2}{}{\csname keycolor#2\endcsname}%
}{}{}
\newcommand{\citecolor}[1]{%
  \@namedef{keycolor#1}{\color{black}}%
}
\makeatother

\definecolor{lightgreen}{RGB}{225, 245, 230}

%% file: section/introduction.tex
\section{Introduction}\label{sec:introduction}
Modern decision-making problems in machine learning, operations research, and control often feature intrinsic randomness, streaming data, and large scales. At the heart of such decision-making pipelines is stochastic optimization, which incorporates random objective functions, constraints, and algorithms with random iterative updates \citep{shapiro2021lectures}. Stochastic optimization often leverages knowledge, estimates, or samples of data distributions to disentangle the complex coupling between randomness and decisions, achieve fast processing and adaptation, and navigate vast search spaces to arrive at optimal solutions.

Classical stochastic optimization assumes that the random variables in a problem obey some fixed distributions. In practice, however, distribution shifts are inevitable and can be both \emph{exogenous} and \emph{endogenous}. Exogenous distribution shifts are largely due to changing environmental conditions, e.g., parameter shifts in online estimation or an arbitrary new distribution selected by an adversary. In this regard, online stochastic optimization emphasizes adaptation by sequentially drawing new samples and adjusting decisions \citep{jiang2020online,cao2021online}.

Endogenous distribution shifts acknowledge the influence of a decision-maker on the data-generating distribution. This influence, i.e., \emph{decision dependence}, contributes to a \emph{closed loop}, whereby the decision and the data distribution interrelate in a repeated decision-making scenario. Various issues cause endogenous shifts and lead to different problem formulations, such as a reinforcement learning agent interacting with its environment, a dominant decision-maker being a price maker in a market, or a content recommender shaping user preferences, among others. In two-stage stochastic programming, models of how first-stage decisions alter distributions of random quantities in the second stage are discussed in \citet{hellemo2018decision}. Performative prediction \citep{perdomo2020performative,hardt2023performative} tackles optimization involving distributions in the form of a general map parameterized by decisions. These parameterized distributions are inspired by strategic behaviors, where individuals intentionally modify features as a response to the deployed predictive model. A predominant assumption made a priori is that distribution shifts are Lipschitz, namely, a bounded change of decisions brings about a bounded change of the resulting distributions.

The aforementioned works largely capture decision dependence through (static) parameterized maps. In contrast, we address decision-making under endogenous distribution shifts represented by a broad class of nonlinear dynamics. This explicit formulation of distribution dynamics is motivated by the interactive feedback loop between a decision-maker and an evolving distribution, exemplified by practical problems in recommender systems \citep{dean2024accounting,lanzetti2023impact,chandrasekaran2024mitigating} and opinion dynamics \citep{proskurnikov2017tutorial}. Each individual random variable (representing a feature, a preference, or intrinsic uncertainty) follows latent dynamics coupling the historical value with the current decision. Through distribution dynamics, the decision affects the individual variable and, thus, the overall distribution. This type of decision dependence features a nonlinear mixture of sequential decisions and non-stationary distributions, which renders the associated stochastic optimization problem challenging and fundamentally different from performative prediction. Our formulation is also closely aligned with the mean-field setting \citep{lauriere2022learning}. However, major differences exist in terms of the performance measure, the structure of decisions, and the specifications of dynamics and distributions; see \cref{subsec:literature} for further discussions.

By exploiting the structure of distribution dynamics, we will provide a fine-grained analysis of distribution shifts, design iterative stochastic algorithms tailored to this dynamic setting, and establish guarantees of optimality and generalization for this decision-dependent stochastic problem. Our design and analysis benefit from a distributional perspective, building connections between stochastic optimization, nonlinear control, and metric probability spaces.

\subsection{Motivations}\label{subsec:motivation}
We present a motivating example of optimal decision-making under endogenous, dynamic distribution shifts in the domain of opinion dynamics.


\begin{figure}[!tb]
    \centering
    \includegraphics[width=0.75\columnwidth]{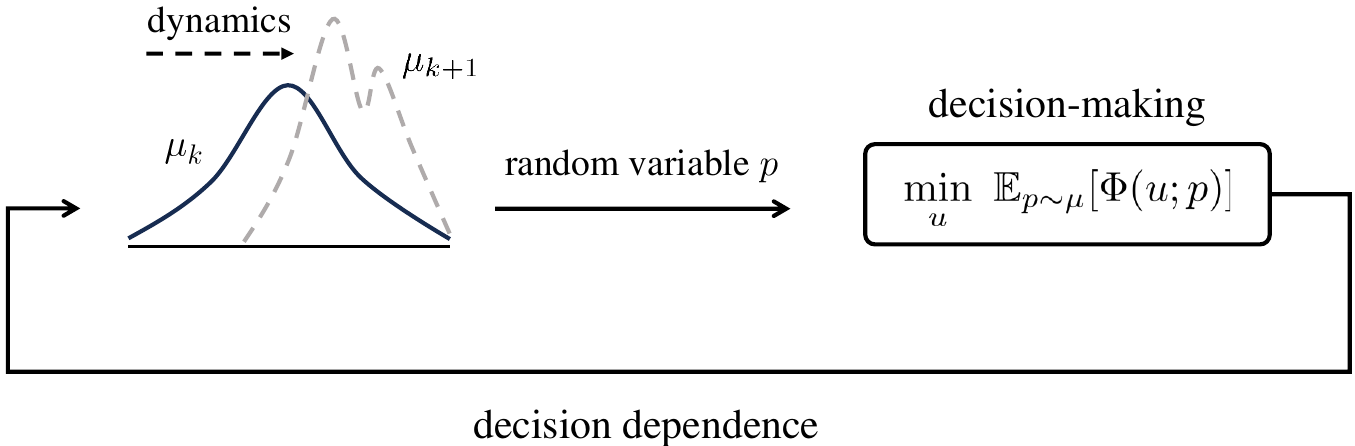}
    \caption{Stochastic optimization with decision dependence features a closed loop, involving endogenous distribution shifts from $\mu_k$ to $\mu_{k+1}$ due to the decision $u$ and the dynamics. Here $\Phi$ is an objective function of the decision $u$ and the random variable $p$; see problem~\eqref{eq:dd_opt_problem} in \cref{subsec:prob_formulation} for a formal account.}
    \label{fig:illustration-cl}
\end{figure}


\begin{example}[Ideology tailored to a polarized population]\label{exm:political_position}
	Consider the interaction between a political party (or a candidate) and a large population. The party is opportunistic, with the aim of gaining power by picking an ideology that aligns with the majority and grants it the most votes.

	Let the adopted ideology of the party and the position (or the preference state) of a random individual in the population at time $k$ be denoted by $q_k \in \mathbb{R}^m$ and $p_k \in \mathbb{R}^m$, respectively. Each coordinate of $q_k$ or $p_k$ indicates a liberal or conservative opinion on a political issue, e.g., taxes, health care, or immigration. The evolution of $p_k$ driven by $q_k$ is
	\begin{equation}\label{eq:opinion_dynamics_general}
		p_{k} = g(p_{k-1},q_k,p_0), \qquad p_0 \sim \mu_d, \quad k \in \mathbb{N}_{+},
	\end{equation}
	where $g: \mathbb{R}^m \times \mathbb{R}^m \times \mathbb{R}^m \to \mathbb{R}^m$ specifies the opinion dynamics, and $\mu_d$ is the distribution of the initial state. For instance, the classical Friedkin--Johnsen model \citep{proskurnikov2017tutorial} reads $p_{k} = \Lambda_1 W p_{k-1} + \Lambda_2 q_k + (I-\Lambda_1-\Lambda_2) p_0$, where $I \in \mathbb{R}^{m\times m}$ is the identity matrix, and $\Lambda_1,\Lambda_2,W \in \mathbb{R}^{m\times m}$ are weight matrices. The nonlinear polarized model \citep{hkazla2024geometric,gaitonde2021polarization} is $p_{k} \propto \lambda p_{k-1} + (1-\lambda) p_0 + \sigma (p_{k-1}^{\top} q_k)q_k$, and $p_k$ is always normalized, i.e., $\|p_k\|=1, \forall k$. At every time $k$, the party collects samples from the population and opportunistically adjusts the ideology $q_k$. To maximize the votes at an upcoming election, the steady-state population-wide affinity that the party intends to maximize is $\E_{\pss}[\pss^{\top} q]$, where $\pss$ satisfies the fixed-point equation $\pss = g(\pss,q,p_0)$ \citep{yang2020us,dean2022preference}. Since $\pss$ is rarely available, the party uses the current opinion distribution of $p_k$, estimated by polls (i.e., sampling), for decision-making.

	Due to the dynamics \eqref{eq:opinion_dynamics_general} and the chosen ideology $q_{k}$, the distribution followed by $p_k$ changes constantly. This distribution shift in turn affects the decision-making based on $p_k$ and causes a feedback loop; see \cref{fig:illustration-cl}. The classical paradigm of repeated sampling and retraining helps to adapt to this distribution shift. Nonetheless, this paradigm can suffer from sub-optimality because the dependence of the steady-state position $\pss$ on the decision $q$ is ignored. In \cref{sec:experiment}, we will review this motivating example and provide a detailed formulation, analysis, and numerical results.
\end{example}

Motivated by this example, we explore general stochastic optimization with endogenous distribution shifts arising from the interaction between a decision-maker and a dynamic distribution; see \cref{fig:illustration-cl} for an illustration. This setup involves several major challenges. First, the largely unknown distribution dynamics exclude an offline strategy based on the exact (re)formulation of the decision-making problem. Second, while the paradigm of repeated sampling and retraining in performative prediction facilitates adaptation (see \cref{subsec:literature}), achieving optimality beyond performative stability requires anticipating how the decision affects the distribution and applying proactive adjustments. Finally, the dynamic setup restricts us from sampling from the steady-state distribution corresponding to the decision, which is the \emph{very} distribution we care about while evaluating the overall performance. In this paper, we will address these challenges by developing and analyzing an online stochastic algorithm tailored to this dynamic setting.

\subsection{Contributions}
We are motivated by applications where a decision-maker drives a constantly evolving distribution and aims to optimize the distribution-level performance. We formulate a decision-dependent stochastic problem featuring endogenous distribution shifts with latent dynamics. For this general and dynamic setting, we adopt a distributional perspective at the intersection of stochastic optimization, nonlinear control, and metric probability spaces and present the following contributions.

\begin{itemize}
	\item We characterize the distribution shift via the contracting coefficient of distribution dynamics and the change of decisions. To this end, we use as the main metric the Wasserstein distance \citep{villani2009optimal} between the current distribution and its steady-state distribution induced by the decision. This metric plays a similar (albeit not necessarily the same) role as a Lyapunov function in control theory \citep{khalil2002nonlinear}. Although the exact value of such a Wasserstein distance can be elusive, the recursive inequality related to this metric sheds light on the dynamic evolution of the distribution.

	\item We propose an online stochastic algorithm that leverages samples from the current distribution and takes into account the composite structure of the problem due to dynamics. The iterative update direction consists of two terms. One term focuses on adaptation, and the most recent samples are exploited to adjust decisions. The other term actively shapes future distributions by anticipating the sensitivity of the distribution with respect to the decision, thereby informing optimal decision-making. \rev{Our algorithm leverages partial, structural knowledge of the world model governing distribution dynamics. This key information is the so-called sensitivity, an analytical component that can be learned from data under suitable conditions.}
	
	\item We establish optimality guarantees of the proposed algorithm in a nonconvex setting. In the face of dynamics and the composite structure, the convergence measure (i.e., the expected second moment of gradients) enjoys a favorable $\mathcal{O}(1/\sqrt{T})$ rate, where $T$ is the total number of iterations. This rate is as sharp as that of stochastic gradient descent for static nonconvex problems without decision dependence. Further, we provide high-probability convergence guarantees for a single run of the proposed algorithm. The key insight is to synthesize the coupled evolution of the aforementioned Wasserstein metric and the convergence measure, which correspond to the dynamic distribution and the iterative algorithm, respectively.

	\item We quantify the finite-sample generalization performance. We explore how the decisions obtained based on an empirical distribution with finite samples will generalize to the original distribution-level problem. We demonstrate that the generalization measure scales polynomially with the number of samples and the number of iterations, as well as poly-logarithmically with the inverse of the failure probability. These results are built on the measure concentration argument and the characterization of the distribution shift, illustrating the benefits of our distributional perspective.

	\item We illustrate the aforementioned results with practical examples, where an opportunistic party maximizes the affinity of a polarized dynamic population, and a recommender optimizes performance by interacting with a user. In the first setup, the population is modeled as a continuous distribution, whereas in the second example, the user is represented by a discrete distribution evolving in the probability simplex. We demonstrate that respecting the composite problem structure due to decision dependence is crucial for achieving fast convergence and improved optimality.
\end{itemize}

\subsection{Related work}\label{subsec:literature}
A multitude of works investigate the roles of distribution shifts and decision dependence in machine learning, optimization, and control. We provide a concise review of their setups and focus areas.


In applications, data-generating distributions can change because of various factors, e.g., non-stationary environments, unknown covariate shifts, and adversarial effects. There are two predominant strategies for addressing distribution shifts. One strategy pursues robustness against potential perturbations to distributions. This falls under the umbrella of distributionally robust optimization \citep{duchi2021learning,kuhn2024Distributionally} and is achieved by optimizing the worst-case cost over a so-called ambiguity set, i.e., the family of distributions that are close to the true distribution under certain metrics. Further, in two-stage stochastic programming, the first-stage decision may change the uncertain distribution in the second stage, thereby producing a decision-dependent ambiguity set \citep{hellemo2018decision}. Tractable reformulations are derived in \cite{luo2020distributionally,basciftci2021distributionally} to disentangle this dependence and obtain robust solutions.

In the face of constantly evolving distributions, a less conservative and more active strategy is to seek adaptation. Specifically, online stochastic optimization investigates an iterative loop whereby a decision-maker commits to a decision, receives samples from a dynamic distribution as feedback \citep{jiang2020online,cao2021online}, and then generates a new decision. Such an online framework is versatile and particularly well suited to tackling exogenous distribution shifts, which are outside the influence of a decision-maker. Nonetheless, in terms of optimization subject to endogenous distribution shifts caused by decisions (see also the beginning of \cref{sec:introduction}), unique phenomena (e.g., the existence of equilibria and stability issues) arise, requiring tailored methods for this closed-loop setting. We delineate some representative works as follows.

\emph{Performative prediction} studies optimization under decision-dependent distributions \citep{perdomo2020performative,hardt2023performative}. A canonical example is strategic classification, where an individual deliberately modifies her features in reaction to the deployed classifier, thereby gaining a more favorable classification result. Such endogenous distribution shifts are usually formalized as static distribution maps parameterized by decisions. This setup leads to a novel equilibrium notion termed performative stability, meaning the decision optimizes the objective given the specific distribution it induces. Through repeated sampling and retraining, various stochastic algorithms converge to performatively stable points \citep{mendler2020stochastic,drusvyatskiy2023stochastic}. A stronger solution concept is performative optimality, requiring that the decision and the induced distribution together lead to an optimal objective value. Convergence to performatively optimal points is achieved with additional structural assumptions on the distribution map, e.g., it belongs to a linear location-scale family \citep{miller2021outside,jagadeesan2022regret,narang2023multiplayer,yan2024zero} or an exponential family \citep{izzo2021learn}. As a general, elegant, and tractable framework, performative prediction admits numerous extensions, including network scenarios with cooperative \citep{li2022multi} or competing agents \citep{narang2023multiplayer,piliouras2023multi}, time-varying objectives \citep{cutler2023stochastic,wood2022online}, saddle point minimax problems \citep{wood2023stochastic}, and coupling constraints \citep{yan2024zero}.

Our work is closely related to \emph{stateful performative prediction}, which captures historical dependence and considers an evolving distribution that gradually settles at the stationary distribution map. Models of historical dependence include geometrically ergodic Markov chains \citep{li2022state} and geometric decay responses in the form of a linear mixture \citep{brown2022performative,ray2022decision}. With Lipschitz distribution shifts assumed in the first place, the algorithms therein converge to performatively stable (albeit not necessarily optimal) points. In contrast, we formalize a decision-dependent distribution shift featuring a broad class of nonlinear dynamics mixing continuous decisions and state distributions. Moreover, we explicitly characterize the distribution shift through the Wasserstein metric \citep{villani2009optimal} and link this shift with dynamics parameters and the change of decisions. We further establish convergence to first-order stationary solutions given a nonconvex objective function involving the decision-dependent structure.

Along the line of \emph{performative decision-making}, some recent works study the role of dynamics in the decision-dependent problem setup. Performative reinforcement learning \citep{mandal2023performative} handles transition probability and reward functions relying on the deployed policy and finds a stable policy given cumulative rewards. Performative control in \citet{cai2024performative} addresses linear dynamics with policy-dependent state transition matrices and seeks a performatively stable control solution as a linear combination of states and disturbances. The framework of \cite{conger2024strategic} represents the dynamics of a strategic population via a gradient flow in the Wasserstein space. The interconnection of this strategic population and a decision-maker results in coupled partial differential equations, which admit asymptotic convergence to optimal solutions for convex (or concave) energy functionals. In contrast, we characterize the distribution shift represented by general nonlinear dynamics. Furthermore, we offer insights into anticipating the sensitivity of the distribution shift, taking into account the composite structure due to dynamics in the algorithmic design, and achieving decision-making in the context of nonconvex objectives. The incorporation of anticipating sensitivity to actively shape distributions is the distinguishing feature of our algorithm compared to the aforementioned performative methods.

\emph{Mean-field formulation} abstracts the mutual influence in a vast homogeneous population by the interaction between a representative individual and an average density, i.e., the mean field \citep{caines2021mean}. This abstraction facilitates approximately solving an otherwise intractable multi-agent problem, wherein the joint states and actions may grow exponentially with the number of agents. The solutions to mean-field games or control are characterized by two coupled equations, namely, a forward equation for the evolving mean-field distribution and a backward equation associated with the individual value function. Classical approaches rely on the full knowledge of population dynamics and work under restrictive assumptions \citep{lauriere2022learning}. A recent trend is to apply reinforcement learning to learn models of dynamics \citep{huang2024model}, value functions, or policy functions \citep{cui2021approximately}, thereby obtaining Nash equilibrium policies or socially optimal policies. Different from the mean-field formulation, we examine a simplified case where individuals do not interact with one another. Nevertheless, we provide the following new insights. First, we require less information or learning effort related to distribution dynamics. Rather than constructing well-calibrated models, we only access steady-state sensitivity matrices corresponding to dynamics, which are learnable in appropriate settings; see \cref{subsec:dynamics}. Second, we explicitly characterize the distribution shift via dynamics parameters and the change of decisions. Finally, instead of analyzing a cumulative cost and a state-feedback policy, we seek optimal steady-state performance and a general decision vector. For the generic setting with nonconvex objectives and nonlinear dynamics, we quantify the first-order stationarity of the obtained solutions.

\emph{Control in probability spaces} addresses the formulation where the state of a system is a probability measure instead of a Euclidean vector \citep{chen2021optimal}. This formulation facilitates the characterization of the evolving uncertainties in a system or the collective behavior of a population, both of which are intrinsically modeled via probability distributions \citep{terpin2024dynamic,lanzetti2023impact}. A typical example is distribution steering, i.e., driving the state distribution from an initial density to a target density in finite time with minimum energy control \citep{chen2021optimal}. In this regard, tractable control strategies cross-fertilize insights from control theory (e.g., linear state-feedback structures) and optimal transport \citep[e.g., transport map calculations,][]{villani2009optimal}. In contrast, we search for an optimal decision vector rather than a state feedback control policy. Further, we do not aim for a specific final distribution. Instead, we hope the decision, together with the distribution induced through the nonlinear dynamics, will lead to an optimal steady-state behavior.

In a broader context, our work aligns with feedback optimization \citep{simonetto2020time,hauswirth2021optimization} that implements optimization iterations as a feedback controller, thereby regulating the steady-state behavior of a dynamical system. Nonetheless, in this paper we are concerned with distribution-level characterizations in a metric space of probability measures, which is drastically different from the system-theoretic analysis of feedback optimization in Euclidean space.

\vspace*{1em}

In summary, we capture the distribution shift arising from the dynamic evolution of a distribution driven by a decision-maker. Such dynamics are represented by a broad class of nonlinear equations encompassing continuous states and decisions. We propose and characterize an online stochastic algorithm that respects the composite problem structure due to dynamics, regulates the distribution flow in real time, and yields first-order stationary solutions to the overall nonconvex problem featuring decision dependence.

The remainder of this paper is structured as follows. \cref{sec:formulation} introduces preliminaries of metric spaces of probability measures and formulates the stochastic optimization problem involving dynamic decision-dependent distributions. \cref{sec:design} presents the intuitions behind and the design of our online stochastic algorithm. In \cref{sec:analysis}, we characterize the distribution dynamics and establish guarantees on the optimality and generalization performance of the proposed algorithm. \cref{sec:experiment} showcases an application in affinity maximization with a polarized, dynamic population following \cref{exm:political_position}, as well as another case study in a recommender system context involving discrete distributions. Finally, \cref{sec:conclusion} concludes this paper and discusses future directions. All proofs are provided in the appendix.

%% file: section/formulation.tex
\section{Preliminaries and Problem Formulation}\label{sec:formulation}
\subsection{Metric space of probability measures}
We review the background of metric probability spaces and refer readers to \cite{villani2009optimal} for more details. Let $\mathcal{P}(\mathbb{R}^m)$ be the space of Borel probability distributions on $\mathbb{R}^m$. Let $\mathcal{P}_1(\mathbb{R}^m) \triangleq \big\{\mu \in \mathcal{P}(\mathbb{R}^m): \int_{\mathbb{R}^m} \|x\| \ud \mu(x) < \infty \big\}$ be the space of distributions with finite first moments. The Dirac mass at point $x \in \mathbb{R}^m$ is denoted by $\delta_x$, i.e., for any Borel set $A \subseteq \mathbb{R}^m$, $\delta_x(A) = 1$ if $x \in A$ and $\delta_x(A) = 0$ otherwise. We use $X \sim \mu$ to indicate that a random variable $X$ is distributed according to $\mu$. 
The convolution of two distributions $\mu, \nu \in \mathcal{P}_1(\mathbb{R}^m)$ is denoted by $\mu * \nu$. Specifically, if two random variables $X$ and $Y$ are independent and distributed according to $\mu$ and $\nu$, respectively, then $X+Y \sim \mu * \nu$.
The pushforward of a distribution $\mu$ via a Borel map $f: \mathbb{R}^{m} \to \mathbb{R}^r$ is represented by $f_{\#}\mu$, where $(f_{\#} \mu)[A] \triangleq \mu[f^{-1}(A)]$ for every Borel set $A \subseteq \mathbb{R}^r$. In fact, if $X \sim \mu$, then $f(X) \sim f_{\#} \mu$. The identity map is $\id$. 

Let $\|z\|_P = \sqrt{z^\top P z}$ denote the weighted norm of a vector $z \in \mathbb{R}^m$, where $P \in \mathbb{R}^{m\times m}$ is positive definite. Consider a metric space $(\mathbb{R}^m, c)$ endowed with a norm-induced metric $c : \mathbb{R}^m \times \mathbb{R}^m \to \mathbb{R}_{\geq 0}$. Typical examples of $c$ include the Euclidean distance $c(x,y) = \|x-y\|$, the weighted distance $c(x,y) = \|x-y\|_P$, and other distances defined by composite norms, where $x,y \in \mathbb{R}^m$. The type-$1$ Wasserstein distance $W_1(\mu,\nu)$ between two distributions $\mu, \nu \in \mathcal{P}_1(\mathbb{R}^m) $ on $(\mathbb{R}^m, c)$ is
\begin{equation}\label{eq:W1_dist}
	W_1(\mu,\nu) = \inf_{\gamma \in \Gamma(\mu,\nu)} \int_{\mathbb{R}^m \times \mathbb{R}^m} c(x,y) \ud \gamma(x,y),
\end{equation}
where $\Gamma(\mu,\nu)$ is the set of all joint distributions (i.e., couplings) with marginals $\mu$ and $\nu$; see \citet[Definition~6.1]{villani2009optimal}. Intuitively, the Wasserstein distance is the minimum cost of transporting $\mu$ onto $\nu$, where the cost of moving a unit mass from $x$ to $y$ is $c(x,y)$, and the available transport plan is represented by $\gamma$. The Wasserstein distance is a flexible and quantitative measure of the discrepancy between distributions, particularly when they have disjoint supports, such as when one distribution is continuous and the other is discrete.

\subsection{Distribution dynamics}\label{subsec:dynamics}
We generalize the motivating case study in \cref{exm:political_position} and consider the following distribution dynamics with continuous states
\begin{equation}\label{eq:pop_dynamics}
	p_k = f(p_{k-1}, u_k, d),  \qquad p_0 \sim \mu_0, ~ d \sim \mu_d,~ (p_0,d) \sim \alpha, \quad k \in \mathbb{N}_{+}.
\end{equation}
In \eqref{eq:pop_dynamics}, $p_k \in \mathbb{R}^m$ is a random state at time $k$ distributed according to $\mu_k$, i.e., $p_k \sim \mu_k$. The initial state $p_0$ follows the distribution $\mu_0 \in \mathcal{P}_1(\mathbb{R}^m)$. Further, $u \in \mathbb{R}^n$ is a decision (or an input) that influences each random state, and $d \in \mathbb{R}^r$ following the distribution $\mu_d \in \mathcal{P}_1(\mathbb{R}^r)$ is an exogenous input that remains constant during iterations. For instance, $d$ can be a bias term (similar to the initial position in \cref{exm:political_position}), a disturbance, or a random parameter in a model of $f$. Each pair $(p_0,d)$ is independently drawn from the joint distribution $\alpha$, and the first and the second marginals of $\alpha$ are $\mu_0$ and $\mu_d$, respectively. For instance, if $p_0$ and $d$ as well as $\mu_0$ and $\mu_d$ are the same (cf.~\cref{exm:political_position}), then $\alpha = (\id, \id)_{\#} \mu_0$; if $p_0$ and $d$ are independent, then $\alpha$ is the product measure $\mu_0 \times \mu_d$.


The distribution dynamics \eqref{eq:pop_dynamics} feature decision dependence, in that the evolution of the distribution $\mu_k$ is driven by the decision $u_k$. This distribution will in turn determine the optimal decision for an optimal distribution-level behavior. Before we present the formal problem description, we specify some properties related to the dynamics \eqref{eq:pop_dynamics}. All these properties serve the purpose of characterization and analysis. \rev{Our online algorithm exploits limited information on the model $f$ of the distribution dynamics; it leverages samples and certain learnable \emph{sensitivities} related to \eqref{eq:pop_dynamics}.}

\begin{assumption}\label{assump:stable_system}
	The function $f(p,u,d)$ is continuously differentiable, $L_f^p$-Lipschitz continuous in $p$ with respect to the weighted norm $\|\cdot\|_P$, where $P \in \mathbb{R}^{m\times m}$ is positive definite, $L_f^u$-Lipschitz continuous in $u$ with respect to $\|\cdot\|$, and $L_f^d$-Lipschitz continuous in $d$ with respect to $\|\cdot\|$. Here, $L_f^p \in (0,1)$, $L_f^u > 0$, and $L_f^d > 0$. There exists a continuously differentiable steady-state map $h: \mathbb{R}^n \times \mathbb{R}^r \to \mathbb{R}^m$ such that $h(u,d) = f(h(u,d),u,d)$. Furthermore, $\nabla_u h(u,d)$ is $M_h^d$-Lipschitz in $d$ with respect to $\|\cdot\|$.
\end{assumption}

\cref{assump:stable_system} implies that the dynamics \eqref{eq:pop_dynamics} are contracting in $p$ with respect to the weighted norm $\|\cdot\|_P$. The existence and conditioning of $P$ (i.e., the so-called contraction metric) are known for incrementally exponentially stable nonlinear dynamics \citep{tsukamoto2021contraction}. 
\rev{Similar contraction and Lipschitz assumptions appear in \citet[Theorem~3]{jafarpour2021robust} on perturbed fixed-point equations and \citet[Lemma~1 and Theorem~2]{davydov2025time} on parameter-varying dynamics, in the contexts of implicit neural networks and time-varying optimization, respectively.} 
Based on the Lipschitz conditions of $f$, the Banach contraction theorem ensures that for a fixed input $u$ (i.e., $u_k=u,\forall k \in \mathbb{N}_{+}$) and a specific exogenous input $d$, the dynamics \eqref{eq:pop_dynamics} admit a unique steady state $\pss = h(u,d)$ satisfying $\pss = f(\pss,u,d)$. We can further establish that $h$ is $L_h^u$-Lipschitz in $u$ and $L_h^d$-Lipschitz in $d$ with respect to $\|\cdot\|$, where $L_h^u = L_f^u\sqrt{\lambda_{\max}(P)/\lambda_{\min}(P)}/(1-L_f^p)$ and $L_h^d = L_f^d\sqrt{\lambda_{\max}(P)/\lambda_{\min}(P)}/(1-L_f^p)$; see the parametric contraction mapping principle \citep[Theorem~1A.4]{dontchev2009implicit} and also \cref{lem:weighted_norm} in \cref{app:lemmas}. The Lipschitz continuity condition on $\nabla_u h(u,d)$ (i.e., the so-called sensitivity matrix) holds when $f$ has bounded Hessians. As we will see in \cref{exm:linear_contraction} below, stable linear dynamics naturally satisfy \cref{assump:stable_system}.
\rev{Although \cref{assump:stable_system} enforces contraction and rules out complex behaviors including bifurcation or multiple attractors, we will present extended results that account for non-contracting polarized dynamics in \cref{subsec:polarized_pop} (albeit retaining a unique steady-state distribution so that the optimization objective is well-defined).}

From a distribution-level perspective, all the steady-state samples $p_\textup{ss}$ follow the distribution $\muss(u)$ that depends on $u$ and $\mu_d$
\begin{equation}\label{eq:steady_state_dist}
  	\muss(u) = h(u,\cdot)_{\#} \mu_d,
\end{equation}
where $h(u,\cdot)_{\#} \mu_d$ denotes the pushforward of the distribution $\mu_d$ via a Borel map $h(u,\cdot): \mathbb{R}^r \to \mathbb{R}^m$ parameterized by the decision $u$.
\rev{The steady-state distribution $\muss(u)$ is uniquely determined by $u$ and $\mu_d$, a specification similar to those in \citet{brown2022performative,ray2022decision,li2022state}.}
Let $\nabla_u h(u,d) \in \mathbb{R}^{n\times m}$ be the steady-state sensitivity matrix of $p_\textup{ss} = h(u,d)$ with respect to the decision $u$. It follows from the implicit function theorem \citep[Theorem~1B.1]{dontchev2009implicit} that
\begin{equation}\label{eq:sens_formula}
	\nabla_u h(u,d) = -\nabla_u f(p_\textup{ss},u,d)\left[\nabla_p f(p_\textup{ss},u,d) - I\right]^{-1},
\end{equation}
which holds in an open neighborhood of $(u,d)$. The sensitivity $\nabla_u h(u,d)$ quantifies the rate of change of the steady-state sample $p_\textup{ss}$ with respect to the decision $u$. While \eqref{eq:sens_formula} may seem daunting at first glance, the sensitivity can be simplified in various scenarios. For instance, when the dynamics~\eqref{eq:pop_dynamics} are linear, the corresponding sensitivity becomes a constant matrix; see \cref{exm:linear_contraction} below. If the exogenous input $d$ is additive in \eqref{eq:pop_dynamics}, then the explicit formula of \eqref{eq:sens_formula} no longer involves $d$. More broadly, apart from invoking \eqref{eq:sens_formula} based on the related knowledge of dynamics and parameters, we can exploit recursive estimation or identification techniques to construct (approximate) sensitivities; see \citet[Sec.~3.3.1]{hauswirth2021optimization} in the context of widely adopted feedback optimization methods. 
Such a learnable sensitivity is the only model information on \eqref{eq:pop_dynamics} used in our online algorithm. All the characterizations in \cref{assump:stable_system} are for the sake of analysis, and the full, high-fidelity world model (i.e., $f$) of the distribution dynamics is not required.

In the following example, we review an important special case of \eqref{eq:pop_dynamics}, where the dynamics function is linear. We will see how \cref{assump:stable_system} is justified and provide explicit expressions for the distribution $\mu_k$ at time $k$ and the steady-state distribution $\muss(u)$.

\begin{example}[Linear distribution dynamics]\label{exm:linear_contraction}
	Suppose that each sample evolves according to a linear dynamical equation $p_k = f(p_{k-1},u_k,d) = A p_{k-1} + B u_k + Ed$, where $A \in \mathbb{R}^{m \times m}$ is non-zero, $B \in \mathbb{R}^{m\times n}$, $E \in \mathbb{R}^{m\times r}$, $p_0 \sim \mu_0, d \sim \mu_d$. Then, \cref{assump:stable_system} is satisfied if $A$ is Schur stable, i.e., $\rho(A) < 1$. 
	Given a Schur stable matrix $A$ and a positive definite $Q \in \mathbb{R}^{m\times m}$, there exists a unique positive definite matrix $P \in \mathbb{R}^{m\times m}$ satisfying the Lyapunov equation $A^{\top}PA - P + Q = 0$ \citep{khalil2002nonlinear}. Let $\lambda_{\min}(Q) > 0$ and $\lambda_{\max}(P) > 0$ denote the minimum eigenvalue of $Q$ and the maximum eigenvalue of $P$, respectively. With the metric $\|\cdot\|_P$, for any $p, \bar{p} \in \mathbb{R}^m$, $u \in \mathbb{R}^n$, and $d \in \mathbb{R}^r$,
	\begin{align*}
		\|f(p,u,d) - f(\bar{p},u,d)\|_P &= \|A(p - \bar{p})\|_P \stackrel{\text{(a.1)}}{=} \sqrt{(p-\bar{p})^\top (P-Q) (p-\bar{p})} \\
			&\stackrel{\text{(a.2)}}{\leq} \underbrace{\sqrt{1-\frac{\lambda_{\min}(Q)}{\lambda_{\max}(P)}}}_{\in (0,1)} \|p - \bar{p}\|_P,
	\end{align*}
	where (a.1) follows from the aforementioned Lyapunov equation. In (a.2), we use
	\begin{equation*}
		\|x\|_Q \geq \sqrt{\lambda_{\min}(Q)} \|x\| \geq \sqrt{\frac{\lambda_{\min}(Q)}{\lambda_{\max}(P)}} \|x\|_P, \qquad \forall x \in \mathbb{R}^m,
	\end{equation*}
	see also \cref{lem:weighted_norm} in \cref{app:lemmas}. Moreover, since $Q-P = -A^\top PA$ is negative semidefinite, we know $\lambda_{\min}(Q) I \preceq Q \preceq P \preceq \lambda_{\max}(P)I$. If $\lambda_{\min}(Q) = \lambda_{\max}(P)$, then $P=Q$, which implies $A=0$ and contradicts the setting. Therefore, $\lambda_{\min}(Q)/\lambda_{\max}(P) \in (0,1)$. Further, for any $p \in \mathbb{R}^m$, $u, \bar{u} \in \mathbb{R}^n$, and $d \in \mathbb{R}^r$,
	\begin{equation*}
		\|f(p,u,d) - f(p,\bar{u},d)\| = \|B(u - \bar{u})\| \leq \|B\| \, \|u - \bar{u}\|.
	\end{equation*}
	When the decision $u$ is fixed (i.e., $u_k=u,\forall k\in \mathbb{N}_{+}$), the steady-state sample $\pss$ is $\pss = h(u,d) = (I-A)^{-1}(B u + E d)$, and the map $h$ is continuously differentiable. The steady-state sensitivity matrix is $\nabla_u h(u,d) = \bigl[(I-A)^{-1}B\bigr]^\top$, which is constant and independent of $u$ and $d$. In applications, this sensitivity (in engineering lingo called the zero-frequency gain) can be learned from data of decisions and samples \citep[Sec.~3.3.1]{hauswirth2021optimization}. Hence, \cref{assump:stable_system} is satisfied.

    \rev{If $p_0$ and $d$ are independent (i.e., the joint distribution of $p_0$ and $d$ is the product measure $\mu_0 \times \mu_d$),} similar to \citet[Proposition~18]{aolaritei2022uncertainty}, the transient distribution $\mu_k$ of \eqref{eq:pop_dynamics} is
	\begin{equation*}
		\mu_k = \left(A^k x \right)_{\#} \mu_0 * \left({\textstyle \sum_{i=0}^{k-1}A^i B u_{k-i} + \sum_{i=0}^{k-1} A^i E x}\right)_{\#} \mu_d,
	\end{equation*}
	where $(A^k x)_{\#} \mu_0$ denotes the pushforward of $\mu_0$ via the map $\bar{f}(x) = A^k x$, and a similar definition holds for the other term after convolution.
	The steady-state distribution $\muss$ for a fixed decision $u$ is
	\begin{equation*}
		\muss(u) = \left((I-A)^{-1}E x + (I-A)^{-1}B u\right)_{\#} \mu_d.
	\end{equation*}
\end{example}

\subsection{Problem formulation}\label{subsec:prob_formulation}
We aim to find a decision $u$ that optimizes the steady-state behavior of the dynamic distribution \eqref{eq:pop_dynamics}:
\begin{equation}\label{eq:dd_opt_problem}
	\begin{split}
		\min_{u\in \mathbb{R}^n} \quad& \E_{p \sim \muss(u)}[\Phi(u,p)] \\
		\textup{s.t.} \quad& \muss(u) = h(u,\cdot)_{\#} \mu_d,
	\end{split}
\end{equation}
where $\muss(u)$ is the steady-state distribution of \eqref{eq:pop_dynamics} induced by $u$; see also \eqref{eq:steady_state_dist}. Let the reduced objective function of problem~\eqref{eq:dd_opt_problem} be denoted by
\begin{equation}\label{eq:dd_opt_reduced_obj}
	\obj(u) \triangleq \E_{p \sim \muss(u)}[\Phi(u,p)] = \E_{d \sim \mu_d}[\Phi(u,h(u,d))]. 
\end{equation}
The decision-dependent problem~\eqref{eq:dd_opt_problem} formalizes the steady state of the closed loop illustrated by \cref{fig:illustration-cl} in \cref{subsec:motivation}. This problem is relevant in many scenarios with a vast population or intrinsic uncertainties, e.g., voting and recommender systems.

Problem~\eqref{eq:dd_opt_problem} involves several major challenges in terms of the nonconvex objective and the unknown dynamics underlying the decision-dependent distribution. 
\begin{itemize}
	\item First, the distribution dynamics \eqref{eq:pop_dynamics} induce the steady-state distribution $\muss(u)$ and further bring about a composite structure in problem~\eqref{eq:dd_opt_problem}. Even for a convex function $\Phi$, since the steady-state map $h(u,d)$ is nonlinear, the overall objective $\obj(u)$ can be nonconvex.

	\item Another challenge originates from the unknown distribution dynamics, which render the structure of decision dependence elusive and preclude an offline numerical scheme based on the exact (re)formulation of problem~\eqref{eq:dd_opt_problem}.

	\item Finally, we cannot directly sample from the steady-state distribution $\muss(u)$ unless we wait sufficiently long, because the distribution $\mu_k$ is constantly changing with time $k$ and eventually approaches the steady-state distribution when the decision is held fixed.
\end{itemize}
To overcome the above challenges, we will propose an online stochastic algorithm in \cref{sec:design} that samples from the current distribution $\mu_k$ and regulates the distribution shift by anticipating its sensitivity with respect to the decision.

We focus on steady-state performance due to its relevance, generality, and tractability. First, in many problems an optimal steady state matters more than transients, with the latter often not even being modeled; see, for instance, many case studies where feedback optimization or performative prediction is applied. This setup will also allow us to circumvent the need for model knowledge by exploiting distribution sensitivity, which is easier to learn. Second, if we analyze the behavior of a stable and dynamic distribution over a sufficiently long horizon, then the long-run average performance metric (i.e., the cumulative objective values divided by the number of iterations) converges to the steady-state objective under suitable regularity conditions. Finally, in this context we can go beyond the assumptions and policy classes considered in the mean-field literature \citep{lauriere2022learning} and establish provable guarantees for continuous decision vectors given a broad class of nonlinear dynamics and nonconvex objective functions. Specifically, we make the following assumptions on the objective.
\begin{assumption}\label{assump:obj_property}
	The objective $\obj(u)$ is well defined (i.e., $\E_{d \sim \mu_d}[|\Phi(u,h(u,d))|] \!<\! \infty$), $L_{\obj}$-smooth (with $L_{\obj}$-Lipschitz gradients), and bounded below. The function $\Phi(u,p)$ is $L_\Phi^p$-Lipschitz in $p$. The partial gradients $\nabla_u \Phi(u,p)$ and $\nabla_p \Phi(u,p)$ are $M_\Phi^u$-Lipschitz and $M_\Phi^p$-Lipschitz in $p$, respectively.
\end{assumption}

\cref{assump:obj_property} requires a well-defined expectation function, which is common in stochastic optimization \citep[see][Section~9.2.5]{shapiro2021lectures} and can be satisfied, e.g., when for each $u$, $\Phi(u,h(u,d))$ is dominated by an integrable function of $d$. The smoothness condition is also standard \citep[see][]{bottou2018optimization} and holds, e.g., when for every stochastic sample $p_\textup{ss}=h(u,d)$, the objective $\Phi(u,p_\textup{ss}) = \Phi(u,h(u,d))$ is smooth and the second moment of the corresponding smoothness constant is bounded. The requirement that the partial gradients are Lipschitz in the random variable $p$ is related to (though not the same as) the joint smoothness property used in performative prediction \citep{perdomo2020performative,mendler2020stochastic}. Moreover, the condition that $\obj(u)$ is bounded below implies that problem~\eqref{eq:dd_opt_problem} admits a finite infimum $\obj^* \in \mathbb{R}$.

\begin{assumption}\label{assump:obj_rand_Lipschitz}
	There exists a positive random variable $L(d)$ such that $\mathbb{E}_{d \sim \mu_d}[L(d)] < \infty$, and for all $u_1,u_2 \in \mathbb{R}^n$, $\big|\Phi(u_1,h(u_1,d)) - \Phi(u_2,h(u_2,d))\big| \leq L(d) \|u_1 - u_2\|$.
\end{assumption}

\cref{assump:obj_rand_Lipschitz} is similar to the condition in \citet[Eq.~(9.130)]{shapiro2021lectures}. It is a sufficient condition for the interchangeability of the expectation and gradient operators; see \citet[Theorem~9.56]{shapiro2021lectures} and also \eqref{eq:ss_distr_stoch_grad} in \cref{subsec:sg_intuition} below. If the function $\Phi(u,h(u,d))$ is Lipschitz continuous in $u$, uniformly over $d$, then \cref{assump:obj_rand_Lipschitz} holds naturally.

%% file: section/design.tex
\section{Online Stochastic Decision-Making}\label{sec:design}
We present our online stochastic algorithm for solving the decision-dependent problem~\eqref{eq:dd_opt_problem}. The main challenges stem from the non-stationary and largely unknown distribution due to the dynamics \eqref{eq:pop_dynamics} and the decisions $(u_k)_{k\in \mathbb{N}_{+}}$. To disentangle the complexity associated with distributions, we leverage samples drawn from the current transient distribution as informative characterizations and feedback. Motivated by the composite structure of the objective, we further enhance the algorithmic update with a term that proactively anticipates and shapes the dynamic distribution based on its sensitivity. Thus, we establish first-order stationarity guarantees for nonconvex problems.



\subsection{Intuition behind the stochastic gradient}\label{subsec:sg_intuition}
We provide the intuition for constructing appropriate stochastic gradients to solve problem \eqref{eq:dd_opt_problem}.
Ideally, we aim to obtain the gradient of the objective $\obj$ at $u_k$, i.e.,
\begin{align}\label{eq:ss_distr_stoch_grad}
	\nabla \obj(u_k) =& \nabla \E_{p \sim \muss(u_k)} [\Phi(u_k,p)] \notag \\
		=& \nabla \E_{d \sim \mu_d} [\Phi(u_k,h(u_k,d))] \notag \\
		\stackrel{\textup{(a.1)}}{=}& \E_{d \sim \mu_d} [\nabla \Phi(u_k,h(u_k,d))] \notag \\
	  	\stackrel{\textup{(a.2)}}{=}& \E_{d \sim \mu_d} \left[\nabla_u \Phi(u_k,h(u_k,d)) + \nabla_u h(u_k,d) \nabla_p \Phi(u_k,p)|_{p=h(u_k,d)} \right] \notag \\
	  	\stackrel{\textup{(a.3)}}{=}& \E_{(p,d) \sim \gammass(u_k)}[\nabla_u \Phi(u_k,p) + \nabla_u h(u_k,d) \nabla_p \Phi(u_k,p)],
\end{align}
where (a.1) leverages interchangeability of the expectation and gradient operators thanks to \cref{assump:obj_rand_Lipschitz} \citep[see also][Theorem~9.56]{shapiro2021lectures}; (a.2) uses the law of the total derivative; (a.3) involves expectation with respect to the joint distribution $\gammass(u_k) \triangleq (h(u_k,\cdot), \id)_{\#} \mu_d$ of the steady-state variable $p \sim \muss(u_k) = h(u_k,\cdot)_{\#} \mu_d$ and the exogenous input $d \sim \mu_d$. 

However, the exact gradient $\nabla \obj(u_k)$ can be difficult to calculate for two reasons. First, it involves the expectation with respect to the random variable $p$, although in practice we can only access finite samples of the distribution. Second, since each sample needs a few iterations to approach its steady state, the steady-state distribution $\muss(u_k)$ of $p$ corresponding to the decision $u_k$ is unavailable at the current time $k$.

To address these issues, we use samples drawn from the current distribution $\gamma_k$ to construct a mini-batch stochastic gradient, thereby informing decision-making. The intuition is that $\gamma_k$, as reflected by these samples, serves as a reasonable proxy for the steady-state distribution $\gammass(u_k)$, provided that neighboring decisions (i.e., $u_{k-1}$ and $u_k$) are close and the iteration counter $k$ is large. Consequently, the expected gradient involving $\gamma_k$ also becomes a close approximation of the exact gradient $\nabla \obj(u_k)$ involving $\gammass(u_k)$. We will formalize this intuition in \cref{subsec:distr_shift}.


\subsection{Algorithmic design}\label{subsec:alg_design}
Guided by the aforementioned intuition, our online stochastic algorithm for solving problem~\eqref{eq:dd_opt_problem} is
\begin{equation}\label{eq:stochastic_alg}
	u_{k+1} = u_k - \eta \nablabatch{k} \obj(u_k), \qquad k \in \mathbb{N},
\end{equation}
where $\eta > 0$ is a constant step size, $k$ is the iteration counter, and $\nablabatch{k} \obj(u_k)$ is a stochastic gradient based on mini-batches, i.e.,
\begin{equation}\label{eq:batch_stoch_grad}
	  \nablabatch{k} \obj(u_k) \triangleq \frac{1}{\nmb} \sum_{i=1}^{\nmb} \Big( 
	   \underbrace{\nabla_u \Phi(u_k,p_k^i)}_{\numcircled{1}} + \underbrace{\nabla_u h(u_k,d^{i_k}) \nabla_p \Phi(u_k,p_k^i)}_{\numcircled{2}} \Big).
\end{equation}
In \eqref{eq:batch_stoch_grad}, $\nmb \in \mathbb{N}_{+}$ is the size of the mini-batch. \revise{At every time step $k$, we draw $\nmb$ random samples $\{p_k^i\}_{i=1}^{\nmb}$ afresh from the transient distribution $\mu_k$. The fresh sampling procedure is independent of the past randomness.} As such, conditional on the past randomness, the sample pairs $(p_k^i, d^{i_k})$ are independent and identically distributed, where $i_k$ denotes the index of the exogenous input corresponding to the sample $p_k^i$. These samples depend on the decisions owing to the dynamics \eqref{eq:pop_dynamics}, thereby causing a composite structure in the objective $\Phi(u,p)$. 
Further, $\nabla_u h(u_k,d^{i_k})$ is the steady-state sensitivity of $p^i_\textup{ss}$ with respect to the decision $u_k$. Such a sensitivity admits various simplifications and can often be learned from data in practice; see the discussion in \cref{subsec:dynamics}. For instance, the sensitivity is a constant matrix given linear dynamics and does not explicitly involve $d^{i_k}$ if the exogenous input is additive.

In essence, the mini-batch stochastic gradient \eqref{eq:batch_stoch_grad} is a finite-sample approximation of 
\begin{equation}\label{eq:cur_distr_stoch_grad}
	\widenabla \obj(u_k) \triangleq \E_{(p,d) \sim \gamma_k}[\nabla_u \Phi(u_k,p) + \nabla_u h(u_k,d) \nabla_p \Phi(u_k,p)],
\end{equation}
i.e., the approximate expected gradient at $u_k$ when $(p,d)$ follows the joint distribution $\gamma_k$. Specifically, $\gamma_k = (f^{(k)}(x,y),y)_{\#}\alpha(x,y)$ denotes the pushforward of $\alpha$ in \eqref{eq:pop_dynamics}. Further, $f^{(k)}(p_0,d)$ is the value of $p_k$ given a pair of the initial state $p_0$ and the exogenous input $d$ sampled from the joint distribution $\alpha$, a sequence of decisions $(u_t)_{t=1,\ldots,k}$, and the dynamics~\eqref{eq:pop_dynamics}. That is, $f^{(k)}: \mathbb{R}^m \times \mathbb{R}^r \to \mathbb{R}^m$ is a map parameterized by $u_1,\ldots,u_k$. We define the special case $f^{(0)}(p_0,d)$ as $p_0$. We will show in \cref{sec:analysis} that with suitable algorithmic parameters, $\nablabatch{k} \obj(u_k)$ and $\widenabla \obj(u_k)$ are close to $\nabla \obj(u_k)$, enabling the stochastic algorithm \eqref{eq:stochastic_alg} to yield desirable (stationary) solutions.


The mini-batch stochastic gradient $\nablabatch{k} \obj(u_k)$ involves two terms. Term \numcircled{1} in \eqref{eq:batch_stoch_grad} leverages the current sample $p_k^i$ to construct a partial gradient with respect to $u$. Term \numcircled{2} in \eqref{eq:batch_stoch_grad} anticipates how the decision $u_k$ will influence the sample $p$ (and in turn the objective) and uses this link for achieving optimality. As explained in \cref{subsec:sg_intuition}, these two terms result from the aforementioned composite structure of $\Phi(u,p)$ and the law of the total derivative. While the update rule \eqref{eq:stochastic_alg} follows naturally from our presentation, most related online methods do not use an anticipating term such as term \numcircled{2} in \eqref{eq:batch_stoch_grad}, as discussed in \cref{subsec:literature}.

\subsection{Properties of the stochastic gradient}\label{subsec:sg_property}
\rev{Let $(\Omega, \calF, \mathbb{P})$ be the probability space.} Let $\calF_k$ be the $\sigma$-algebra generated by the random variables $\nablabatch{0} \obj(u_0), \ldots, \nablabatch{k}\obj(u_k)$. \rev{We define the trivial $\sigma$-algebra as $\calF_{-1} \triangleq \{\varnothing, \Omega\}$.} Let $u_0 \in \mathbb{R}^n$ be deterministic. Hence, \rev{for any $k \in \mathbb{N}$,} the decision $u_k$ is measurable with respect to $\calF_{k-1}$. We make the following assumption on the variance of the stochastic gradient constructed from an individual sample. For a stochastic vector $\xi \in \mathbb{R}^n$, we define its variance as $\var [\xi] \triangleq \E\!\left[\|\xi - \E[\xi]\|^2\right] = \E[\|\xi\|^2] - \left\|\E[\xi]\right\|^2$.
\begin{assumption}\label{assump:var_grad}
	\rev{For any $k \in \mathbb{N}$ and any $i = 1,\ldots,\nmb$}, the stochastic gradient satisfies
	\begin{equation*}
		\var \left[\nabla_u \Phi(u_k,p_k^i) \!+\! \nabla_u h(u_k,d^{i_k})\nabla_p \Phi(u_k,p_k^i) \big| \calF_{k-1} \right] \leq M \!+\! M_V \|\widenabla \obj(u_k)\|^2,
	\end{equation*}
	where $M, M_V \geq 0$ are constants.
\end{assumption}


\cref{assump:var_grad} is a standard and relatively weak statement that the variance of the stochastic gradient is restricted \citep{bottou2018optimization}. It implies that this variance can be nonzero at the point where $\widenabla \obj(u_k)$ equals zero and grows at most quadratically in the norm of $\widenabla \obj(u_k)$.

\revise{As clarified after \eqref{eq:batch_stoch_grad}, at time $k$ we draw a mini-batch of random samples from the transient distribution. The fresh sampling process is independent of the $\sigma$-algebra $\calF_{k-1}$. Hence, when conditioned on $\calF_{k-1}$, $\{(p_k^i, d^{i_k})\}_{i=1}^{\nmb}$ are independent and identically distributed according to $\gamma_k$. Furthermore, the decision $u_k$ is measurable with respect to $\calF_{k-1}$.}  
Building on this observation, we specify some key properties of the mini-batch gradient $\nablabatch{k} \obj(u_k)$ in the lemma below.

\begin{lemma}\label{lem:mini_batch_grad}
	Given the dynamics~\eqref{eq:pop_dynamics}, $\nablabatch{k} \obj(u_k)$ in \eqref{eq:batch_stoch_grad} is an unbiased estimate of $\widenabla \obj(u_k)$, i.e.,
	\begin{equation}\label{eq:unbiased_batch_grad}
		\EXPT{\nablabatch{k} \obj(u_k) \middle| \calF_{k-1}} = \widenabla \obj(u_k), \quad \rev{\forall k \in \mathbb{N}}.
	\end{equation}
	If \cref{assump:var_grad} holds, the expected second moment of $\nablabatch{k} \obj(u_k)$ satisfies the following bound:
	\begin{equation}\label{eq:second_mom_batch_grad}
		\EXPT{\|\nablabatch{k} \obj(u_k)\|^2 \middle| \calF_{k-1}} \leq \frac{M}{\nmb} + \left(\frac{M_V}{\nmb} + 1\right) \|\widenabla \obj(u_k)\|^2, \quad \rev{\forall k \in \mathbb{N}}.
	\end{equation}
\end{lemma}


%% file: section/analysis.tex
\section{Performance Analysis}\label{sec:analysis}
We analyze the interplay between the distribution dynamics \eqref{eq:pop_dynamics} and the proposed online stochastic algorithm \eqref{eq:stochastic_alg}. In \cref{subsec:distr_shift}, we characterize the distribution shift driven by the decision-maker through the Wasserstein metric. We then establish the optimality guarantees of the algorithm \eqref{eq:stochastic_alg} when applied to the distribution dynamics \eqref{eq:pop_dynamics}. These guarantees hold in expectation and with high probability and are given in \cref{subsec:optimality} and \cref{subsec:optimality_hp}, respectively, thereby covering a broad spectrum of the overall performance. Finally, we consider a finite-sample regime and provide generalization certificates in \cref{subsec:generalization}.

\subsection{Distribution shifts}\label{subsec:distr_shift}
The distribution dynamics \eqref{eq:pop_dynamics} bring about constant shifts of the joint distribution $\gamma_k$. We characterize such distribution shifts through the lens of the Wasserstein distance. Specifically, we focus on the behavior of $W_1(\gamma_k,\gammass(u_k))$, i.e., the Wasserstein distance between the joint distribution $\gamma_k$ at time $k$ and the joint steady-state distribution $\gammass(u_k)$ induced by the decision $u_k$.

Throughout this section, we consider the space $(\mathbb{R}^{m+r},c)$ endowed with the metric
\begin{equation}\label{eq:metric_mix_norm}
	c\left((p,d),(p',d')\right) \triangleq \|p-p'\|_P + \|d-d'\|,
\end{equation}
where $p,p' \in \mathbb{R}^m, d,d' \in \mathbb{R}^r$, and $\|\cdot\|_P$ is the weighted norm. Then, $W_1(\gamma_k,\gammass(u_k))$ is defined by
\begin{equation*}
	W_1(\gamma_k,\gammass(u_k)) = \inf_{\bar{\gamma} \in \Gamma(\gamma_k,\gammass(u_k))} \int_{\mathbb{R}^{m+r} \times \mathbb{R}^{m+r}} c\left((p,d),(p',d')\right) \ud \bar{\gamma}\left((p,d),(p',d')\right),
\end{equation*}
where $(p,d)$ and $(p',d')$ are distributed according to $\gamma_k$ and $\gammass(u_k)$, respectively.


The following lemma gives an upper bound on the Wasserstein distance $W_1(\gamma_k,\gammass(u_k))$ and establishes a related recursive inequality. It is built on the properties of the pushforward operation induced by Borel maps; see \citet[Proposition~3]{aolaritei2022uncertainty}. Recall that $L_f^p$ is the Lipschitz constant of the dynamics function $f$ with respect to $p$ under $\|\cdot\|_P$, and that $L_h^u$ is the Lipschitz constant of the steady-state map $h$ of \eqref{eq:pop_dynamics} with respect to $u$ under $\|\cdot\|$; see \cref{assump:stable_system} and the discussion below. Further, $\lambda_{\max}(P) > 0$ is the maximum eigenvalue of the positive definite matrix $P$, and $f^{(k)}(p_0,d)$ denotes the value of $p_k$ given a pair $(p_0,d) \sim \alpha$ and past decisions; see the paragraph below \eqref{eq:cur_distr_stoch_grad} in \cref{subsec:alg_design}.

\begin{lemma}\label{lem:wass_dist_bd_recur}
	Let \cref{assump:stable_system} hold. With \eqref{eq:pop_dynamics}, the joint distributions $(\gamma_k)_{k\in \mathbb{N}}$ satisfy
	\begin{subequations}
		\begin{align}
			W_1(\gamma_k,\gammass(u_k)) &\leq \int_{\mathbb{R}^m \times \mathbb{R}^r} \left\|f^{(k)}(p_0,d) - h(u_k,d)\right\|_P \ud \alpha(p_0,d) \triangleq V_k, \quad \forall k \in \mathbb{N}, \label{eq:wass_dist_indiv_bd} \\
			V_k &\leq L_f^p V_{k-1} + L_f^p L_h^u \sqrt{\lambda_{\max}(P)} \|u_k - u_{k-1}\|, \quad \forall k \in \mathbb{N}_{+}. \label{eq:wass_dist_bd_contract}
		\end{align}
	\end{subequations}
	Moreover, for any $u_0 \in \mathbb{R}^n$, $V_0 \triangleq \int_{\mathbb{R}^m \times \mathbb{R}^r} \|f^{(0)}(p_0,d) - h(u_0,d)\|_P \ud \alpha(p_0,d)$ is finite.
\end{lemma}

\cref{lem:wass_dist_bd_recur} characterizes the behavior of $W_1(\gamma_k,\gammass(u_k))$ through the evolution of the upper bound $V_k$. Specifically, this upper bound exhibits perturbed contraction \eqref{eq:wass_dist_bd_contract}, \rev{where the contraction coefficient is $L_f^p \in (0,1)$}, and the perturbation term is proportional to the difference of consecutive inputs. If the inputs are kept fixed (i.e., $u_k = u,\forall k \in \mathbb{N}$), then $V_k$ and hence $W_1(\gamma_k,\gammass(u))$ converge to $0$ as $k$ increases. The implication is that $\gamma_k$ converges weakly in $\mathcal{P}_1(\mathbb{R}^{m+r})$ to the steady-state distribution $\gammass(u)$; see \citet[Theorem~6.9]{villani2009optimal}.

We quantify the cumulative sum of squared Wasserstein distances in the theorem below by building on \cref{lem:wass_dist_bd_recur}. As explained in \cref{subsec:sg_intuition}, the difference between $\gamma_k$ and $\gammass(u_k)$ causes a bias in the gradient at time $k$. Such a cumulative sum reflects how those biases accumulate when we deploy our online algorithm \eqref{eq:stochastic_alg} and will be useful for quantifying the convergence measure of the nonconvex problem; see \cref{thm:optimality_nonconvex} in \cref{subsec:optimality} and \cref{thm:optimality_nonconvex_high_prob} in \cref{subsec:optimality_hp}.

\begin{theorem}\label{thm:wasserstein_evolution_general}
	Given \cref{assump:stable_system} and the distribution dynamics \eqref{eq:pop_dynamics}, the joint distributions $(\gamma_k)_{k\in \mathbb{N}}$ satisfy
	\begin{equation}\label{eq:dist_square_sum_general}
		\sum_{k=0}^{T-1} W_1(\gamma_k,\gammass(u_k))^2 \leq \frac{\rev{V_0^2}}{1-\rho_1} + \frac{\rho_2}{1-\rho_1} \sum_{k=1}^{T} \|u_k-u_{k-1}\|^2,
	\end{equation}
	where the coefficients are $\rho_1 = \frac{1+(L_f^p)^2}{2} \in (0,1)$ and $\rho_2 = \frac{1+(L_f^p)^2}{1-(L_f^p)^2} (L_f^p L_h^u)^2 \lambda_{\max}(P)$. Furthermore, the online stochastic algorithm~\eqref{eq:stochastic_alg} applied to \eqref{eq:pop_dynamics} ensures that
	\begin{equation}\label{eq:dist_square_sum_gd}
		\sum_{k=0}^{T-1} W_1(\gamma_k,\gammass(u_k))^2 \leq \frac{V_0^2}{1-\rho_1} + \frac{\eta^2 \rho_2}{1-\rho_1} \sum_{k=0}^{T-1} \|\nablabatch{k} \obj(u_k)\|^2.
	\end{equation}
\end{theorem}

Analogous to \eqref{eq:wass_dist_bd_contract}, the upper bound on the sum of squared Wasserstein distances depends on the cumulative variation in decisions. After invoking the gradient-based update rule \eqref{eq:stochastic_alg}, we know that this cumulative variation is proportional to $\eta^2$ and the squared norm of the mini-batch gradient $\nablabatch{k} \obj(u_k)$, where $k=0,\ldots,T\!-\!1$. In the following subsections, we will exploit \eqref{eq:dist_square_sum_gd} and the connection between $\nablabatch{k} \obj(u_k)$ and $\nabla \obj(u_k)$ to establish convergence guarantees.

An important implication of the distribution shift due to the dynamics \eqref{eq:pop_dynamics} is that the mini-batch stochastic gradient \eqref{eq:batch_stoch_grad}, albeit unbiased with respect to the approximate gradient $\widenabla \obj(u_k)$ (see \eqref{eq:cur_distr_stoch_grad}), is a biased estimate of the true gradient $\nabla \obj(u_k)$. The reason is that we draw samples from the current distribution $\gamma_k$ instead of the steady-state distribution $\gammass(u_k)$ to construct \eqref{eq:batch_stoch_grad}. Let $e_k$ denote the difference between the expected gradients due to the discrepancy between $\gamma_k$ and $\gammass(u_k)$, 
that is,
\begin{align}\label{eq:error_grad}
	e_k \triangleq& \widenabla \obj(u_k) - \nabla \obj(u_k) \notag \\
		=& \E_{(p,d) \sim \gamma_k}[\nabla_u \Phi(u_k,p) + \nabla_u h(u_k,d) \nabla_p \Phi(u_k,p)] \notag \\
		 &- \E_{(p,d) \sim \gammass(u_k)}[\nabla_u \Phi(u_k,p) + \nabla_u h(u_k,d) \nabla_p \Phi(u_k,p)].
\end{align}
The following lemma provides an upper bound on $\|e_k\|$ through $W_1(\gamma_k,\gammass(u_k))$, i.e., the Wasserstein distance between two joint distributions $\gamma_k$ and $\gammass(u_k)$.

\begin{lemma}\label{lem:grad_err_bound}
	If \cref{assump:stable_system,assump:obj_property,assump:obj_rand_Lipschitz} are satisfied, then
	\begin{equation}\label{eq:grad_err_bd_wasserstein}
		\|e_k\| = \|\widenabla \obj(u_k) - \nabla \obj(u_k)\| \leq L W_1(\gamma_k,\gammass(u_k)),
	\end{equation}
	where $L = \max\left(L_{\Phi}^p M_h^d, \frac{M_\Phi^u + L_h^u M_\Phi^p}{\sqrt{\lambda_{\min}(P)}} \right)$.
\end{lemma}

\cref{lem:grad_err_bound} indicates that the difference $e_k$ between the approximate gradient $\widenabla \obj(u_k)$ involving $\gamma_k$ and the true gradient $\nabla \obj(u_k)$ involving $\gammass(u_k)$ admits an upper bound proportional to the Wasserstein distance $W_1(\gamma_k, \gammass(u_k))$. By referring to \cref{lem:wass_dist_bd_recur} on the evolution of $W_1(\gamma_k, \gammass(u_k))$, we know that when neighboring decisions (i.e., $u_k$ and $u_{k-1}$) are close and the iteration counter is large, $\widenabla \obj(u_k)$ is a close approximation of $\nabla \obj(u_k)$. This observation aligns with the intuition stated at the end of \cref{subsec:sg_intuition}.

\subsection{Optimality in expectation}\label{subsec:optimality}
We present the optimality guarantee of our stochastic algorithm~\eqref{eq:stochastic_alg} when applied to the distribution dynamics~\eqref{eq:pop_dynamics}. The major challenge is that the evolution of the distribution and the iterates of the algorithm are coupled. To disentangle this coupling, we leverage the characterizations of the distribution shift in \cref{thm:wasserstein_evolution_general} and the descent-type iterate \eqref{eq:stochastic_alg} and then quantify the overall convergence measure, i.e., the average expected second moment of gradients.

Recall that $\obj(u) = \E_{p \sim \muss(u)}[\Phi(u,p)]$ is the objective, $\obj^*$ is the optimal value of problem~\eqref{eq:dd_opt_problem}, $\nabla \obj(u)$ is given by \eqref{eq:ss_distr_stoch_grad}, $\eta$ is the step size, $T \in \mathbb{N}_+$ is the number of iterations, $\nmb \in \mathbb{N}_+$ is the size of the mini-batch, \rev{$\Lobj$ is the smoothness constant of $\obj(u)$}, $L$ is the constant specified in \cref{lem:grad_err_bound}, $M$ is the constant in the variance bound in \cref{assump:var_grad}, and $\rho_1$ and $\rho_2$ are constants given in \cref{thm:wasserstein_evolution_general}. In the following theorem, we provide the first main convergence result of our stochastic algorithm \eqref{eq:stochastic_alg}.

\begin{theorem}\label{thm:optimality_nonconvex}
	Suppose that \cref{assump:stable_system,assump:obj_property,assump:obj_rand_Lipschitz,assump:var_grad} hold. Let $\eta$ be chosen such that
	\begin{equation}\label{eq:step_size_expt}
		0 < \eta \leq \frac{1}{\sqrt{T}} \cdot \min\left(\frac{1}{4\Lobj \left(\frac{M_V}{\nmb} + 1\right)}, \sqrt{\frac{1-\rho_1}{14\rho_2 L^2 \left(\frac{M_V}{\nmb} + 1\right)}} \right).
	\end{equation}
	The stochastic algorithm \eqref{eq:stochastic_alg} applied to the distribution dynamics \eqref{eq:pop_dynamics} guarantees that
	\begin{align}\label{eq:average_second_moment_bd}
		\frac{1}{T} \sum_{k=0}^{T-1} \EXPT{\|\nabla \obj(u_k)\|^2} \leq \frac{8(\obj(u_0) - \obj^*)}{\eta T} + \frac{4\Lobj M\eta}{\nmb} + \bigO{\frac{1}{T}}.
	\end{align}
	By choosing $\eta$ as the upper bound in \eqref{eq:step_size_expt}, the order of the right-hand side of \eqref{eq:average_second_moment_bd} is
	\begin{equation}\label{eq:average_second_moment_complexity}
		\frac{1}{T} \sum_{k=0}^{T-1} \EXPT{\|\nabla \obj(u_k)\|^2} = \bigO{\frac{1}{\sqrt{T}}}.
	\end{equation}
\end{theorem}

The convergence measure analyzed in \cref{thm:optimality_nonconvex} is the average of the expected second moments of $\nabla \obj(u_k)$. The expectation is taken with respect to the randomness of the proposed algorithm, i.e., the randomness in selecting mini-batch samples. This measure is typical in nonconvex optimization to characterize first-order stationarity of solutions, which is a necessary condition for local optimality \citep{bottou2018optimization}. The upper bound \eqref{eq:average_second_moment_bd} on this measure involves the number of iterations $T$, the step size $\eta$, and other constants related to the problem and the algorithm. With a step size attaining the upper bound in \eqref{eq:step_size_expt}, our stochastic algorithm \eqref{eq:stochastic_alg} yields an $\mathcal{O}(1/\sqrt{T})$ rate of convergence, matching that of stochastic gradient descent for static nonconvex problems. 
It is common in the stochastic optimization literature \citep{ghadimi2013stochastic,bottou2018optimization} to let the bound on the step size be a function of the number of iterations, facilitating convergence rate analysis.

The rate \eqref{eq:average_second_moment_complexity}, nonetheless, is nontrivial given the presence of unknown dynamics and the composite structure of the decision-dependent problem. Online stochastic algorithms reviewed in \cref{subsec:literature} that lack an anticipation term such as the one in \eqref{eq:batch_stoch_grad} can incur persistent biases in gradients, causing sub-optimality (namely, the average expected second moment of gradients does not vanish). 
Although dynamics inhibit us from directly accessing steady-state samples and cause biases in gradients, we demonstrate that the accumulation of these biases does not worsen the overall convergence rate. This is largely due to the contracting distribution dynamics and a relatively slow algorithm (with a bounded step size), \`{a} la time-scale separation \citep{khalil2002nonlinear} as quantified by \eqref{eq:step_size_expt}. Finally, we remark that variance reduction techniques offer a promising means of improving the convergence rate \citep{bottou2018optimization}.

The expected gradient $\nabla \obj(u_k)$, as part of the convergence measure in \eqref{eq:average_second_moment_bd}, involves the steady-state distribution $\gammass(u_k)$. In practice, however, we can only sample from the current distribution $\gamma_k$ to gain an understanding of the quality of solutions. In the following corollary, we analyze the average expected second moment of $\widenabla \obj(u_k)$ involving $\gamma_k$; see \eqref{eq:cur_distr_stoch_grad}.

\begin{corollary}\label{cor:approximate_optimality}
	Let the conditions of \cref{thm:optimality_nonconvex} hold. The stochastic algorithm \eqref{eq:stochastic_alg} acting on the distribution dynamics \eqref{eq:pop_dynamics} ensures that
	\begin{equation}\label{eq:average_approx_sec_mom_complexity}
		\frac{1}{T} \sum_{k=0}^{T-1} \EXPT{\|\widenabla \obj(u_k)\|^2} \leq \frac{7}{3T} \sum_{k=0}^{T-1} \EXPT{\|\nabla \obj(u_k)\|^2} + \bigO{\frac{1}{T}} = \bigO{\frac{1}{\sqrt{T}}}.
	\end{equation}
\end{corollary}

\cref{cor:approximate_optimality} implies that the average expected second moment of $\widenabla \obj(u_k)$ related to the current distribution $\gamma_k$ is still of the order of $\mathcal{O}(1/\sqrt{T})$. The reason is that the convergence measures in \eqref{eq:average_second_moment_complexity} and \eqref{eq:average_approx_sec_mom_complexity} can be connected via the average expected squared Wasserstein distance, i.e., $\big(\sum_{k=0}^{T-1} \EXPT{W_1(\gamma_k,\gammass(u_k))^2}\big)/T$, whose order is the same as both measures; see also \eqref{eq:wasserstein_square_sum_pre} in \cref{app:proof_thm_optimality}.

Apart from the above optimality guarantees, we are also interested in the convergence of the distribution dynamics~\eqref{eq:pop_dynamics} while interacting with the decision-maker as per \eqref{eq:stochastic_alg}. We provide the convergence guarantee in the following theorem.
\begin{theorem}\label{thm:distribution_expt}
	Let the conditions of \cref{thm:optimality_nonconvex} hold. Then,
	\begin{equation}\label{eq:avg_wasserstein_dist}
		\frac{1}{T} \sum_{k=0}^{T-1} \EXPT{W_1(\gamma_k, \gammass(u_k))} \leq \sqrt{\frac{1}{6L^2} \!\left(\!\frac{8(\obj(u_0) \!-\! \obj^*)}{\eta T} \!+\! \frac{4\Lobj M\eta}{\nmb} \!\right)} + \bigO{\frac{1}{\sqrt{T}}} \!=\! \bigO{\frac{1}{T^{\frac{1}{4} }}}.
	\end{equation}
\end{theorem} 

In \cref{thm:distribution_expt}, we characterize the convergence of the distribution dynamics~\eqref{eq:pop_dynamics} via the average expected Wasserstein distance between the current distribution $\gamma_k$ and the corresponding steady-state distribution $\gammass(u_k)$. Similar to \eqref{eq:average_second_moment_bd}, the upper bound in \eqref{eq:avg_wasserstein_dist} also depends on the number of iterations, the step size, and other problem- and algorithm-specific constants. As the number of iterations grows, the average expected Wasserstein distance becomes closer to zero, implying that the dynamic distributions approach their steady-state distributions on average. This convergence plays an important role in optimality certificates that couple the distribution dynamics~\eqref{eq:pop_dynamics} and the algorithm~\eqref{eq:stochastic_alg}; see the discussions below \cref{thm:wasserstein_evolution_general,thm:optimality_nonconvex}.

We remark that the Wasserstein metric for the distribution dynamics \eqref{eq:pop_dynamics} in \cref{thm:distribution_expt} is of the same order as the average expected gradient norm $\sum_{k=0}^{T-1} \mathbb{E}[\|\nabla \obj(u_k)\|]/T$ rather than the average expected second moment of gradients in \cref{thm:optimality_nonconvex}. Given the analogy between optimization and sampling, readers may wonder why the convergence rate \eqref{eq:avg_wasserstein_dist} does not match the $\mathcal{O}(1/\sqrt{T})$ rate of Langevin dynamics for sampling from a target distribution \citep[see, e.g.,][]{chau2021stochastic}. 
Apart from the distinction in problem setups, another major reason for this difference in convergence rates lies in the coupling of the distribution dynamics~\eqref{eq:pop_dynamics} and the algorithm~\eqref{eq:stochastic_alg}. As characterized by \eqref{eq:dist_square_sum_gd}, the incurred cumulative Wasserstein distance is related to the second moments of gradients. Since the nonconvex objective of problem~\eqref{eq:dd_opt_problem} limits the achievable order of the average second moment (namely, $\mathcal{O}(1/\sqrt{T})$, see \cref{thm:optimality_nonconvex}), the order of the average Wasserstein distance between distributions is also restricted.

\subsection{Optimality with high probability}\label{subsec:optimality_hp}
The convergence guarantee in \cref{thm:optimality_nonconvex} holds in expectation. This result reflects the average performance pertaining to the interplay between the stochastic algorithm~\eqref{eq:stochastic_alg} and the distribution dynamics~\eqref{eq:pop_dynamics}. However, some cases (e.g., with constraints on computational resources or time) allow only a single trial or a few trials of the stochastic algorithm, rendering the occurrence of an extreme outcome a key concern, particularly in high-stakes scenarios, such as elections discussed in \cref{exm:political_position}. In this regard, another angle for characterizing algorithmic and distributional behaviors is convergence with high probability. Our goal is to quantify the confidence level that a stochastic algorithm yields satisfactory solutions when applied to distribution dynamics after a single trial involving many iterations.

To establish high-probability guarantees, we need an assumption that differs from \cref{assump:var_grad} and is instead related to the distribution of the gradient noise. Recall from \eqref{eq:unbiased_batch_grad} in \cref{lem:mini_batch_grad} that the mini-batch gradient $\nablabatch{k} \obj(u_k)$ is an unbiased estimate of the approximate gradient $\widenabla \obj(u_k)$. Let
\begin{equation}\label{eq:batch_grad_stoch_err}
	\xi_k = \nablabatch{k} \obj(u_k) - \widenabla \obj(u_k) 
\end{equation}
be the noise in $\nablabatch{k} \obj(u_k)$ relative to $\widenabla \obj(u_k)$. Hence, $\EXPT{\xi_k | \mathcal{F}_{k-1}} = 0$, where $\mathcal{F}_{k-1}$ is the $\sigma$-algebra generated by $\nablabatch{0} \obj(u_0),\ldots,\nablabatch{k-1} \obj(u_{k-1})$. Specifically, we assume that $\xi_k$ is sub-Gaussian, i.e., with a tail dominated by the tail of a Gaussian distribution. This specification is formalized by the following assumption.

\begin{assumption}\label{assump:sub_gaussian_noise}
	When conditioned on $\calF_{k-1}$, the noise $\xi_k$ \eqref{eq:batch_grad_stoch_err} is sub-Gaussian, i.e., there exists $\sigmamb > 0$ such that for any $u_k \in \mathbb{R}^n$, 
	\begin{equation}\label{eq:subgaussian_exp}
		\EXPT{\exp\left(\frac{\|\xi_k\|^2}{\sigmamb^2}\right)\bigg| \calF_{k-1}} \leq \exp(1).
	\end{equation}
\end{assumption}

Compared to \cref{assump:var_grad} on the variance, \cref{assump:sub_gaussian_noise} specifies that the gradient noise $\xi_k$ is light-tailed. Building on Jensen's inequality (i.e., $\forall X \in \mathbb{R}, \exp(\E[X]) \leq \E[\exp(X)]$), \cref{assump:sub_gaussian_noise} implies that the variance of $\nablabatch{k} \obj(u_k)$ admits a constant upper bound $\sigmamb^2$, which is a bit stronger than \cref{assump:var_grad}. In practice, \cref{assump:sub_gaussian_noise} can be satisfied when the noise in the sample gradient corresponding to each individual is sub-Gaussian. 
Such light-tailed sub-Gaussian noises are considered in \cite{harvey2019tight,li2020high} to analyze the behaviors of stochastic gradient descent for static problems. Further, self-normalized concentration inequalities provide a promising avenue for handling heavy-tailed (e.g., sub-Weibull) noises; see \cite{madden2024high,li2022high}.

We establish the high-probability guarantee of our stochastic algorithm~\eqref{eq:stochastic_alg} when applied to the distribution dynamics~\eqref{eq:pop_dynamics}. 
The involved variables, parameters, and constants are the same as those in \cref{thm:optimality_nonconvex}. The proof leverages the characterization of the cumulative squared Wasserstein distances in \cref{thm:wasserstein_evolution_general} and some high-probability bounds on terms involving $\xi_k$. 

\begin{theorem}\label{thm:optimality_nonconvex_high_prob}
	Suppose that \cref{assump:stable_system,assump:obj_property,assump:obj_rand_Lipschitz,assump:sub_gaussian_noise} hold. Let $\eta$ be chosen such that
	\begin{equation}\label{eq:step_size_hp}
		0 < \eta \leq \frac{1}{\sqrt{T}} \cdot \min\left(\frac{1}{6\Lobj}, \sqrt{\frac{1-\rho_1}{21\rho_2 L^2}} \right).
	\end{equation}
	The stochastic algorithm \eqref{eq:stochastic_alg} applied to the distribution dynamics \eqref{eq:pop_dynamics} guarantees that for any fixed $\tau \in (0,1)$, with probability at least $1-\tau$,
	\begin{align}\label{eq:average_second_moment_high_prob}
		\frac{1}{T} \sum_{k=0}^{T-1} \|\nabla \obj(u_k)\|^2 \leq \frac{16(\obj(u_0) - \obj^*)}{\eta T} + 24\eta \Lobj \sigmamb^2 \left(1+\ln \frac{2}{\tau}\right) + \bigO{\frac{1}{T}\left(1+\ln \frac{1}{\tau}\right)}.
	\end{align}
	By selecting $\eta$ as the upper bound in \eqref{eq:step_size_hp}, the order of the right-hand side of \eqref{eq:average_second_moment_high_prob} is
	\begin{equation*}
		\frac{1}{T} \sum_{k=0}^{T-1} \|\nabla \obj(u_k)\|^2 = \bigO{\frac{1}{\sqrt{T}} \left(1 + \ln \frac{1}{\tau}\right)}.
	\end{equation*}	
\end{theorem}

The upper bound on the step size in \eqref{eq:step_size_hp} is of the same order as that of \eqref{eq:step_size_expt} in \cref{thm:optimality_nonconvex}, although constant coefficients differ. In the stochastic optimization literature \citep{ghadimi2013stochastic,bottou2018optimization}, this bound is typically set as a function of the number of iterations to characterize the convergence rate. 
The convergence measure in \cref{thm:optimality_nonconvex_high_prob} is the average squared gradient norm. In contrast to \cref{thm:optimality_nonconvex}, however, the measure here does not involve expectation anymore. The upper bound \eqref{eq:average_second_moment_high_prob} scales inversely with the square root of the number of iterations $T$. Moreover, it features a logarithmic dependence on the inverse of the failure probability $\tau$. This dependence is favorable particularly when the desired confidence level is high, i.e., when $\tau$ is small. For example, $\tau=10^{-3}$ leads to $\ln \frac{1}{\tau} \approx 6.9$, which translates to a moderate increase in magnitude. Hence, under \cref{assump:sub_gaussian_noise} of sub-Gaussian noises, convergence results similar to those in \cref{thm:optimality_nonconvex} also hold with high probability. 

Analogous to \cref{thm:distribution_expt}, we offer a high-probability characterization of the convergence of the distribution dynamics \eqref{eq:pop_dynamics} when driven by a decision-maker as per \eqref{eq:stochastic_alg}.
\begin{theorem}\label{thm:distribution_high_prob}
	Let the conditions of \cref{thm:optimality_nonconvex_high_prob} hold. For any fixed $\tau \in (0,1)$, with probability at least $1-\tau$,
	\begin{equation}\label{eq:wasserstein_dist_avg_hp}
		\frac{1}{T} \sum_{k=0}^{T-1} W_1(\gamma_k, \gammass(u_k)) \leq \sqrt{\textstyle \frac{1}{6L^2} \! \left(\! \frac{16(\obj(u_0) \!-\! \obj^*)}{\eta T} \!+\! 24\eta \Lobj \sigmamb^2 \left(1\!+\!\ln \frac{4}{\tau}\right) \! \right)} + \bigO{\textstyle \frac{1}{T^\frac{1}{2}} \! \sqrt{1\!+\!\ln \frac{1}{\tau}}}.
	\end{equation}
	By choosing $\eta$ as the upper bound in \eqref{eq:step_size_hp}, the order of the right-hand side of \eqref{eq:wasserstein_dist_avg_hp} is
	\begin{equation*}
		\frac{1}{T} \sum_{k=0}^{T-1} W_1(\gamma_k, \gammass(u_k)) = \bigO{\frac{1}{T^\frac{1}{4}} \sqrt{1+\ln \frac{1}{\tau}}}.
	\end{equation*}
\end{theorem}

\cref{thm:distribution_high_prob} ensures that with high probability, the average Wasserstein distance between the current distribution and the associated steady-state distribution approaches zero at a rate of $\mathcal{O}(1/T^{\frac{1}{4}})$ as the number of iterations $T$ increases. Moreover, the upper bound on this average distance exhibits a poly-logarithmic dependence on the inverse of the failure probability $\tau$. Overall, the distributions $\gamma_k$ and $\gammass(u_k)$ are close on average with high probability.

\subsection{Generalization in a finite-sample regime}\label{subsec:generalization}
In this subsection, we examine a finite-sample regime, where access is limited to a specific set of samples rather than allowing unrestricted sampling from the full distribution as in problem~\eqref{eq:dd_opt_problem}. When applied in this context, the stochastic algorithm \eqref{eq:stochastic_alg} yields decisions that essentially optimize an empirical objective associated with these samples. We are interested in the generalization guarantees on how such decisions perform in the problem involving the overall distribution.

This setup is of interest from an application perspective. Various reasons, e.g., the strategy of seeding trials or the restricted sampling due to privacy concerns, may cause us to only work with specific samples from the distribution. Yet, we hope the decisions obtained from analyzing these finite samples can generalize to the distribution-level problem. We will show how our distributional perspective allows seamless integration of the measure concentration argument \citep{fournier2015rate} into performance analysis, thus facilitating generalization guarantees.

The distributions of the initial state and the exogenous input are $\mu_0$ and $\mu_d$, respectively. Moreover, let $\{(p^1_0, d^1), \ldots, (p^N_0, d^N)\}$ be independent and identically distributed samples drawn from the joint distribution $\alpha$, where $N \in \mathbb{N}_{+}$ is the total number of samples, and the marginals of $\alpha$ are $\mu_0$ and $\mu_d$. 
The empirical distributions of the initial state and the exogenous input are $\mu^N_0 = \frac{1}{N} \sum_{i=1}^{N} \delta_{p^i_0}$ and $\mu^N_d = \frac{1}{N} \sum_{i=1}^{N} \delta_{d^i}$, respectively. At every time $k$, the decision affects the overall distribution $\mu_k$ as per \eqref{eq:pop_dynamics} and in turn influences the resulting empirical distribution $\mu^N_k$. In a finite-sample regime, the empirical distribution $\mu^N_k$ rather than the full distribution $\mu_k$ is accessible for sampling. For this setting, the online stochastic algorithm \eqref{eq:stochastic_alg} generates decisions $\{u_0^N,\ldots,u_{T-1}^N\}$ for the following finite-sample empirical problem
\begin{equation}\label{eq:dd_opt_problem_sample}
	\min_{u\in \mathbb{R}^n} \quad \frac{1}{N} \sum_{i=1}^{N} \Phi(u, p^i_{\textup{ss}}) = \frac{1}{N} \sum_{i=1}^{N} \Phi(u, h(u,d^i)) \triangleq \obj^N(u).
\end{equation}
Our focus is on establishing generalization guarantees on the performance of the decisions $\{u_0^N,\ldots,\allowbreak u_{T-1}^N\}$ in the distribution-level problem~\eqref{eq:dd_opt_problem}.





\cref{thm:optimality_nonconvex} allows us to characterize the first-order stationarity of the decisions with respect to the empirical objective $\obj^N$ in \eqref{eq:dd_opt_problem_sample}. In terms of the distribution-level performance, however, we center on stationarity with respect to $\obj$ in \eqref{eq:dd_opt_reduced_obj}, i.e., the reduced objective of problem~\eqref{eq:dd_opt_problem}. To this end, we leverage the following decomposition
\begin{equation}\label{eq:pop_grad_decomposition}
	\|\nabla \obj(u_k^N)\|^2 \leq 2\underbrace{\|\nabla \obj(u_k^N) - \nabla \obj^N(u_k^N)\|^2}_{\numcircled{1}} + 2\underbrace{\|\nabla \obj^N(u_k^N)\|^2}_{\numcircled{2}}.
\end{equation}
The left-hand side of \eqref{eq:pop_grad_decomposition} is the squared norm of the gradient $\nabla \obj(u_k^N)$, which reflects the quality of decisions for the distribution-level problem~\eqref{eq:dd_opt_problem}. Term \numcircled{1} in \eqref{eq:pop_grad_decomposition} is the squared norm of the difference between the empirical gradient $\nabla \obj^N$ and the gradient $\nabla \obj$, and it represents the generalization error \citep{li2022high}. Term \numcircled{2} in \eqref{eq:pop_grad_decomposition} is the squared norm of the empirical gradient, constituting the optimization error for the empirical problem~\eqref{eq:dd_opt_problem_sample}.

The following lemma gives a high-probability bound on the generalization error averaged over iterations. It is built on the measure concentration result derived in Wasserstein distances; see \citet[Theorem~2]{fournier2015rate}. Recall that $r$ is the size of the exogenous input $d$.

\begin{lemma}\label{lem:general_err}
	Let \cref{assump:stable_system,assump:obj_property,assump:obj_rand_Lipschitz} hold and $r \geq 3$. Suppose that there exist $\theta > 1$ and $\kappa > 0$ such that $\mathcal{E}_{\theta,\kappa}(\mu_d) \triangleq \int_{\mathbb{R}^r} e^{\kappa \|x\|^\theta} \ud \mu_d(x)$ is finite. For any fixed $\tau \in (0,1)$, with probability at least $1 - \frac{\tau}{2}$, the decisions $\{u_0^N, \ldots, u_{T-1}^N\}$ satisfy
	\begin{align*}
		\frac{1}{T} \sum_{k=0}^{T-1} &\|\nabla \obj(u_k^N) - \nabla \obj^N(u_k^N)\|^2 \leq \\
		&
		\begin{cases}
			L^2 \left(L_h^d \sqrt{\lambda_{\max}(P)}\!+\!1\right)^2 \left(\frac{1}{c_2 N} \ln\left(\frac{2c_1}{\tau} \right)\right)^{\frac{2}{r}}, & \text{if } N \geq \frac{1}{c_2} \ln\left(\frac{2c_1}{\tau} \right), \\
			L^2 \left(L_h^d \sqrt{\lambda_{\max}(P)}\!+\!1\right)^2 \left(\frac{1}{c_2 N} \ln\left(\frac{2c_1}{\tau} \right)\right)^{\frac{2}{\theta}}, & \text{if } 1 \leq N < \frac{1}{c_2} \ln\left(\frac{2c_1}{\tau} \right),
		\end{cases}
	\end{align*}
	where $c_1$ and $c_2$ are positive constants that only depend on $r,\theta,\kappa$, and $\mathcal{E}_{\theta,\kappa}(\mu_d)$.
\end{lemma}

The finite moment condition (namely, on $\mathcal{E}_{\theta,\kappa}(\mu_d)$) in \cref{lem:general_err} specifies the degree to which the distribution $\mu_d$ is light-tailed. This condition is common in the literature \citep{fournier2015rate,kuhn2024Distributionally}. \cref{lem:general_err} encompasses both the low-data and high-data regimes.


We now characterize the generalization performance of our stochastic algorithm~\eqref{eq:stochastic_alg}. We focus on the decisions $\{u_0^N,\ldots,u_{T-1}^N\}$ obtained by the algorithm~\eqref{eq:stochastic_alg} in solving the finite-sample problem~\eqref{eq:dd_opt_problem_sample}, with the mini-batch size $\nmb$ satisfying $1 \leq \nmb \leq N$. The quality of these decisions is evaluated by the average squared gradient norm of the objective $\obj$ involving the distribution $\mu_d$ and associated with problem~\eqref{eq:dd_opt_problem}. Note that $r$ is the size of $d$, $\theta > 1$ is a constant present in $\mathcal{E}_{\theta,\kappa}(\mu_d)$, and $c_1$ and $c_2$ are constants specified in \cref{lem:general_err}.

\begin{theorem}\label{thm:generalization_pop}
	Let \cref{assump:stable_system,assump:obj_property,assump:obj_rand_Lipschitz,assump:sub_gaussian_noise} and the conditions of \cref{lem:general_err} hold, and assume $\obj^N(u)$ is $\Lobj^N$-smooth. Consider the use of the stochastic algorithm~\eqref{eq:stochastic_alg} to solve the finite-sample problem~\eqref{eq:dd_opt_problem_sample}, with the step size satisfying $0 < \eta \leq \frac{1}{\sqrt{T}} \cdot \min\Big(\frac{1}{6\Lobj^N}, \sqrt{\frac{1-\rho_1}{21\rho_2 L^2}} \Big)$. 
	Let $\tau \in (0,1)$ be fixed. When $N \geq \frac{1}{c_2} \ln\left(\frac{2c_1}{\tau} \right)$, with probability at least $1-\tau$, 
	\begin{align*}
		\frac{1}{T} \sum_{k=0}^{T-1} \|\nabla \obj(u_k^N)\|^2 \leq& 
		2L^2 \left(L_h^d\sqrt{\lambda_{\max}(P)}\!+\!1\right)^2 \left(\frac{1}{c_2 N} \ln\left(\frac{2c_1}{\tau} \right)\right)^{\frac{2}{r}} \\
		& + \frac{32(\obj^N(u_0^N) \!-\! \obj^{N,*})}{\eta T} + 48\eta \Lobj^N \sigmamb^2 \left(1+\ln \frac{4}{\tau}\right) + \bigO{\frac{1}{T}\left(1 + \ln\frac{1}{\tau}\right)}.
	\end{align*}
	When the step size $\eta$ attains its upper bound, the overall order is
	\begin{equation*}
		\frac{1}{T} \sum_{k=0}^{T-1} \|\nabla \obj(u_k^N)\|^2 = \bigO{\frac{1}{N^\frac{2}{r}} \left(\ln \frac{1}{\tau}\right)^\frac{2}{r}} + \bigO{\frac{1}{\sqrt{T}}\left(1+\ln \frac{1}{\tau}\right)}.
	\end{equation*}
	When $1 \leq N < \frac{1}{c_2} \ln\left(\frac{2c_1}{\tau} \right)$, with probability at least $1-\tau$,
	\begin{align*}
		\frac{1}{T} \sum_{k=0}^{T-1} \|\nabla \obj(u_k^N)\|^2 \leq& 
		2L^2 \left(L_h^d\sqrt{\lambda_{\max}(P)}\!+\!1\right)^2 \left(\frac{1}{c_2 N} \ln\left(\frac{2c_1}{\tau} \right)\right)^{\frac{2}{\theta}} \\
		& + \underbrace{\frac{32(\obj^N(u_0^N) \!-\! \obj^{N,*})}{\eta T} + 48\eta \Lobj^N \sigmamb^2 \left(1+\ln \frac{4}{\tau}\right)}_{\sim \bigO{(1+\ln(1/\tau))/\sqrt{T}}} + \bigO{\frac{1}{T}\left(1 + \ln\frac{1}{\tau}\right)}.
	\end{align*}
\end{theorem}

\cref{thm:generalization_pop} quantifies how the performance measure (i.e., the average squared gradient norm) scales polynomially with the number of samples $N$ and the number of iterations $T$, as well as poly-logarithmically with the inverse of the specified failure probability $\tau$. As $N$ and $T$ increase, this performance measure approaches zero. It implies that the decisions $\{u_0^N,\ldots,u_{T-1}^N\}$ arising from the finite-sample problem~\eqref{eq:dd_opt_problem_sample} generalize to the distribution-level problem~\eqref{eq:dd_opt_problem}.


%% file: section/experiment.tex

\section{Case Studies}\label{sec:experiment}
We present two case studies to demonstrate how decision dependence in stochastic optimization can be effectively addressed with our algorithm. The first case study delineates \cref{exm:political_position} in opinion dynamics and involves continuous population distributions. The second example originates from a recommender system context and focuses on discrete distributions over the probability simplex. \rev{All experiments are conducted on an HTCondor cluster using CPU-only nodes, with $6$ GB of memory allocated per job.} 
Our code is publicly available\footnote{\url{https://github.com/zyhe/distribution-dynamics-opt}}.

\subsection{Affinity maximization in a polarized population}\label{subsec:polarized_pop}
We revisit \cref{exm:political_position} and consider the interaction of an opportunistic party and a polarized dynamic population. While the case study is streamlined, a relevant real-world example is the U.S.~presidential election \citep{yang2020us}. The ideological position of a party causes a shift in the position of each individual in the population. This party focuses on picking an ideology to gain the most votes. A tractable proxy for this goal is to maximize the steady-state population-wide affinity.

The shift of the individual position is described by the following polarized dynamics model adapted from \citet{hkazla2024geometric,gaitonde2021polarization}
\begin{equation}\label{eq:polarized_dynamics}
	\widetilde{p}_{k+1} = \lambda p_k + (1-\lambda) p_0 + \sigma \cdot (p_k^{\top} q) q, \qquad p_{k+1} = \frac{\widetilde{p}_{k+1}}{\|\widetilde{p}_{k+1}\|}, \qquad p_0 \sim \mu_0, ~ \|p_0\|=1, ~ k \in \mathbb{N},
\end{equation}
where $p_k \in \mathbb{R}^m$ is the position of an individual at time $k$, $q \in \mathbb{R}^m$ is the ideological position of the party, $\lambda \in (0,1)$ is a coefficient for blending the current and initial positions, $\sigma > 0$ regulates the position evolution, and $\mu_0$ is the distribution of initial positions. In \eqref{eq:polarized_dynamics}, the term $(1-\lambda)p_0$ captures the persistent influence of the initial position $p_0$, which is also present in the classical linear Friedkin--Johnsen model \citep{proskurnikov2017tutorial}. The closer $\lambda$ is to $1$, the less an individual sticks to her initial position. Thanks to the normalization step, the position $p_k$ is of unit norm. The coefficients $\lambda$ and $\sigma$ are the same for each individual, and therefore the overall population is homogeneous. 

The polarized dynamics \eqref{eq:polarized_dynamics} feature biased assimilation \citep{dean2022preference}. If a specific individual prefers an ideology $q$ (i.e., $p_k^{\top} q > 0$), then her position will move closer to $q$. Conversely, if the ideology is disliked (i.e., $p_k^{\top} q < 0$), then her position will instead move away from $q$, or more specifically, closer to $-q$. \revise{As illustrated in \cref{fig:polarize_ss}, the dynamics \eqref{eq:polarized_dynamics} admit a steady-state (i.e., limiting) position $\pss$, and the location of $\pss$ relative to $q$ depends on the angle between $p_0$ and $q$ and the coefficients $\lambda$ and $\sigma$.} 
Hence, the dynamics \eqref{eq:polarized_dynamics} are not necessarily globally contracting. The affinity of an individual towards an ideology $q$ at time $k$ is given by the inner product $p_k^{\top} q$. 

\revise{
Given $p_0$ and a fixed decision $q$, the fixed-point equation corresponding to the dynamics \eqref{eq:polarized_dynamics} is
\begin{equation}\label{eq:fixed_point_polarized}
    \widetilde{p}_\textup{fix} = \lambda \pfix + (1-\lambda) p_0 + \sigma (\pfix^\top q) q,
    \qquad
    \pfix = \frac{\widetilde{p}_\textup{fix}}{\|\widetilde{p}_\textup{fix}\|},
	\qquad \pfix \in \mathbb{R}^m,
\end{equation}
which may admit multiple solutions depending on $\lambda$, $\sigma$, and $p_0^\top q$. While algebraically valid, some of these solutions neither arise as limits of the trajectory nor reflect the behaviors in an opinion dynamics context \citep{dean2022preference}. Our focus is on the \emph{unique limit} $\pss = \hssmap(q,p_0)$ that is approached by the trajectory starting from $p_0$, evolving by \eqref{eq:polarized_dynamics}, and driven by $q$. This limit is the final interaction outcome between an opportunistic party and an individual and affects affinity maximization.} 
We formalize the properties of the polarized dynamics \eqref{eq:polarized_dynamics} in the following theorem.


\begin{theorem}\label{thm:polarized_sens}
	\revise{For a fixed ideology $q$, the dynamics \eqref{eq:polarized_dynamics} starting at $p_0$ converge to a unique limit $\pss$; this limit $\pss$ is a solution of \eqref{eq:fixed_point_polarized}.} As $\lambda \in (0,1)$ and $\sigma > 0$ increase, $\pss$ becomes closer to $q$ (meaning $\pss^\top q \geq p_0^\top q$) when $p_0^\top q > 0$, and to $-q$ (meaning $\pss^\top q \leq p_0^\top q$) when $p_0^\top q < 0$, indicating a stronger steering ability. 
	\revise{For $(q, p_0)$ such that $p_0^\top q \neq 0$ and the relevant Jacobian is non-singular, the steady-state sensitivity $\nabla_q \hssmap(q,p_0)$ in an open neighborhood of $(q,p_0)$ is}
	\begin{equation}\label{eq:sens_polarized_formula}
		\revise{\nabla_q \hssmap(q,p_0)} = - \sigma \left((\pss^{\top}q) I + \pss q^\top\right)(I- \pss \pss^{\top})\left[(\lambda I + \sigma q q^\top)(I-\pss \pss^{\top}) - \|\widetilde{p}_\textup{ss}\| I\right]^{-1},
	\end{equation}
	where $\widetilde{p}_\textup{ss} = \lambda \pss + (1-\lambda) p_0 + \sigma (\pss^\top q) q$.
\end{theorem}

\tikzset{every picture/.style={line width=1pt}} 
\begin{figure}[!tb]%
	\centering
	\subfloat[Not influenced ($p_0^\top q = 0$) \label{fig:polarize_ss_orthogonal}]{
    \resizebox{!}{2.5cm}{
	    \begin{tikzpicture}
	    [x=0.75pt,y=0.75pt,yscale=-1,xscale=1]
	    
		\draw    (90,120.5) -- (89.69,60.05) ;
		\draw [shift={(89.68,58.05)}, rotate = 89.71] [color={rgb, 255:red, 0; green, 0; blue, 0 }  ][line width=0.75]    (10.93,-3.29) .. controls (6.95,-1.4) and (3.31,-0.3) .. (0,0) .. controls (3.31,0.3) and (6.95,1.4) .. (10.93,3.29)   ;
		\draw [color={rgb, 255:red, 76; green, 0; blue, 153 }  ,draw opacity=1 ]   (90,120.5) -- (152.18,120.54) ;
		\draw [shift={(154.18,120.55)}, rotate = 180.04] [color={rgb, 255:red, 76; green, 0; blue, 153 }  ,draw opacity=1 ][line width=0.75]    (10.93,-3.29) .. controls (6.95,-1.4) and (3.31,-0.3) .. (0,0) .. controls (3.31,0.3) and (6.95,1.4) .. (10.93,3.29)   ;

		\draw (64,52.4) node [anchor=north west][inner sep=0.75pt]    {$\textcolor[rgb]{0,0.4,0.8}{p}\textcolor[rgb]{0,0.4,0.8}{_{\text{ss}}}$};
		\draw (95.5,52.4) node [anchor=north west][inner sep=0.75pt]    {$p_{0}$};
		\draw (157,108.9) node [anchor=north west][inner sep=0.75pt]  [color={rgb, 255:red, 76; green, 0; blue, 153 }  ,opacity=1 ]  {$\textcolor[rgb]{0.3,0,0.6}{q}$};
	    \end{tikzpicture}
	    }
    } \hspace{.5cm}
	\subfloat[Acute angle remains \\ ($p_0^{\top} q > 0$) \label{fig:polarize_ss_acute}]{
	\resizebox{!}{2.5cm}{
	    \begin{tikzpicture}
	    [x=0.75pt,y=0.75pt,yscale=-1,xscale=1]
		    
		\draw    (90,120.5) -- (89.69,60.05) ;
		\draw [shift={(89.68,58.05)}, rotate = 89.71] [color={rgb, 255:red, 0; green, 0; blue, 0 }  ][line width=0.75]    (10.93,-3.29) .. controls (6.95,-1.4) and (3.31,-0.3) .. (0,0) .. controls (3.31,0.3) and (6.95,1.4) .. (10.93,3.29)   ;
		\draw [color={rgb, 255:red, 76; green, 0; blue, 153 }  ,draw opacity=1 ]   (90,120.5) -- (145.86,94.88) ;
		\draw [shift={(147.68,94.05)}, rotate = 155.36] [color={rgb, 255:red, 76; green, 0; blue, 153 }  ,draw opacity=1 ][line width=0.75]    (10.93,-3.29) .. controls (6.95,-1.4) and (3.31,-0.3) .. (0,0) .. controls (3.31,0.3) and (6.95,1.4) .. (10.93,3.29)   ;
		\draw [color={rgb, 255:red, 0; green, 102; blue, 204 }  ,draw opacity=1 ]   (90,120.5) -- (122.62,68.24) ;
		\draw [shift={(123.68,66.55)}, rotate = 121.97] [color={rgb, 255:red, 0; green, 102; blue, 204 }  ,draw opacity=1 ][line width=0.75]    (10.93,-3.29) .. controls (6.95,-1.4) and (3.31,-0.3) .. (0,0) .. controls (3.31,0.3) and (6.95,1.4) .. (10.93,3.29)   ;

		\draw (124.67,52.4) node [anchor=north west][inner sep=0.75pt]    {$\textcolor[rgb]{0,0.4,0.8}{p}\textcolor[rgb]{0,0.4,0.8}{_{\text{ss}}}$};
		\draw (66.5,52.4) node [anchor=north west][inner sep=0.75pt]    {$p_{0}$};
		\draw (150,84.9) node [anchor=north west][inner sep=0.75pt]  [color={rgb, 255:red, 76; green, 0; blue, 153 }  ,opacity=1 ]  {$\textcolor[rgb]{0.3,0,0.6}{q}$};
        \end{tikzpicture}
        }
	} \hspace{.5cm}
	\subfloat[Obtuse angle remains \\ ($p_0^{\top} q < 0$) \label{fig:polarize_ss_obtuse}]{
    	\resizebox{!}{2.7cm}{
	    \begin{tikzpicture}
	    [x=0.75pt,y=0.75pt,yscale=-1,xscale=1]
		    
		\draw    (90,120.5) -- (89.69,60.05) ;
		\draw [shift={(89.68,58.05)}, rotate = 89.71] [color={rgb, 255:red, 0; green, 0; blue, 0 }  ][line width=0.75]    (10.93,-3.29) .. controls (6.95,-1.4) and (3.31,-0.3) .. (0,0) .. controls (3.31,0.3) and (6.95,1.4) .. (10.93,3.29)   ;
		\draw [color={rgb, 255:red, 76; green, 0; blue, 153 }  ,draw opacity=1 ]   (90,120.5) -- (146.26,136.98) ;
		\draw [shift={(148.18,137.55)}, rotate = 196.33] [color={rgb, 255:red, 76; green, 0; blue, 153 }  ,draw opacity=1 ][line width=0.75]    (10.93,-3.29) .. controls (6.95,-1.4) and (3.31,-0.3) .. (0,0) .. controls (3.31,0.3) and (6.95,1.4) .. (10.93,3.29)   ;
		\draw [color={rgb, 255:red, 0; green, 102; blue, 204 }  ,draw opacity=1 ]   (90,120.5) -- (47.49,71.56) ;
		\draw [shift={(46.18,70.05)}, rotate = 49.03] [color={rgb, 255:red, 0; green, 102; blue, 204 }  ,draw opacity=1 ][line width=0.75]    (10.93,-3.29) .. controls (6.95,-1.4) and (3.31,-0.3) .. (0,0) .. controls (3.31,0.3) and (6.95,1.4) .. (10.93,3.29)   ;

		\draw (26.5,65.4) node [anchor=north west][inner sep=0.75pt]    {$\textcolor[rgb]{0,0.4,0.8}{p}\textcolor[rgb]{0,0.4,0.8}{_{\text{ss}}}$};
		\draw (70.5,51.9) node [anchor=north west][inner sep=0.75pt]    {$p_{0}$};
		\draw (151.5,125.9) node [anchor=north west][inner sep=0.75pt]  [color={rgb, 255:red, 76; green, 0; blue, 153 }  ,opacity=1 ]  {$\textcolor[rgb]{0.3,0,0.6}{q}$};
        \end{tikzpicture}
        }
	}
	\caption{The polarized dynamics \eqref{eq:polarized_dynamics} admit a unique limiting state depending on the sign of $p_0^{\top} q$.}
	\label{fig:polarize_ss}
\end{figure}


\rev{More broadly, nonlinear opinion dynamics models are capable of capturing bifurcation or multiple polarized attractors through saturation bounds \citep{bizyaeva2023nonlinear}, hyperbolic tangent functions \citep{leonard2021nonlinear}, or weighted medians \citep{mei2024convergence}. In contrast, we adopt the geometric polarized model \eqref{eq:polarized_dynamics} to maintain tractability in performance optimization, while still permitting non-trivial polarized behaviors. When strong nonlinearities invalidate the unique steady-state condition and render the objective set-valued, the current guarantees no longer apply. However, similar \emph{local} guarantees may be established by restricting the analysis to the regions of attraction associated with various equilibria. Another important direction is incorporating principled selection rules across multiple steady states.}

To align with the majority and gain the most votes, the party aims to find an ideology $q$ that optimizes the steady-state population-wide affinity:
\begin{equation}\label{eq:max_aff_pol}
	\begin{split}
		\max_{q \in \mathbb{R}^m} \quad & \mathbb{E}_{p \sim \muss(q)} [p^\top q] \\
		\textup{s.t.} \quad & \revise{\muss(q) = \hssmap(q,\cdot)_{\#} \mu_0,} \\
			& \|q\| \leq 1,
	\end{split}
\end{equation}
where $\muss(q)$ is the steady-state (i.e., limiting) distribution induced by $q$,
and the randomness is due to the persistent influence of the stochastic initial position acting as an exogenous input $d$ in \eqref{eq:pop_dynamics}. In problem~\eqref{eq:max_aff_pol}, the norm constraint on $q$ represents restrictions (e.g., of resources or social norms) and ensures that the optimal solution to \eqref{eq:max_aff_pol} is bounded. \rev{Since the polarized dynamics \eqref{eq:polarized_dynamics} are not globally contracting, the results herein show behaviors beyond the presented theoretical setting.}

As established in \citet[Proposition~1]{dean2022preference}, when $\lambda = 1$ (i.e., the initial position disappears from \eqref{eq:polarized_dynamics}), for a constant ideology $q$, $\pss = q/\|q\|$ if $p_0^{\top} q > 0$, and $\pss = -q/\|q\|$ if $p_0^{\top} q < 0$. In this case, the goal of individual affinity maximization is trivial: the key is to identify the hemisphere of $p_0$ and then stick to some $q$ that ensures $p_0^{\top} q > 0$ \citep{dean2022preference}. Then, the maximum achievable affinity equals $1$. However, for the general case that involves $\lambda \in [0,1)$ and unknown $p_0$ and extends from the individual to the population, affinity maximization is not straightforward anymore, and a systematic means of searching for $q$ is required.

Problem~\eqref{eq:max_aff_pol} entails the underlying distribution dynamics \eqref{eq:polarized_dynamics}. In practice, the difficulty of accessing an accurate dynamics model and the need for real-time decision-making can render the steady-state distribution $\muss(q)$ elusive, thereby inhibiting offline numerical pipelines that rely on the formula or samples of $\muss(q)$. One online strategy to solve problem~\eqref{eq:max_aff_pol} is to apply a stochastic algorithm in the style of performative prediction \citep{hardt2023performative}; see the discussions in \cref{subsec:literature}. This algorithm exploits repeated sampling and retraining but is unaware of the composite structure wherein the steady-state distribution of positions depends on the ideology. The update rule of such an online vanilla algorithm is
\begin{equation}\label{eq:stochastic_vanilla}  
	q_{k+1} = \proj_{\|q\| \leq 1}\left(q_k + \frac{\eta}{\nmb} \sum_{i=1}^{\nmb} p_k^i \right),
\end{equation}
where $\proj_{\|q\| \leq 1}(\cdot)$ denotes the projection onto the norm ball $\{q \in \mathbb{R}^m \,|\, \|q\| \leq 1\}$, and $\{p_k^i\}$ are collected after applying $q_k$ to the dynamics \eqref{eq:polarized_dynamics}. This form of projected gradient ascent is due to maximization in \eqref{eq:max_aff_pol}. 
\rev{
Furthermore, we implement the following online stateful algorithm building on the delayed repeated risk minimization in \citet{brown2022performative}
\begin{equation}\label{eq:stochastic_stateful}
    q_{k+1} =
    \begin{cases}
        \proj_{\|q\| \leq 1}\left(q_k + \frac{\eta}{\nmb} \sum_{i=1}^{\nmb} p_k^i \right), & \text{if } k \equiv 0 \pmod {\bar{k}},  \\
        q_k, & \text{otherwise},
    \end{cases}
\end{equation}
where $\bar{k} \in \mathbb{N}_{+}$ is the epoch length. The intuition behind \eqref{eq:stochastic_stateful} is to let the polarized dynamics \eqref{eq:polarized_dynamics} evolve for $\bar{k}$ steps before periodically updating the decision $q$. The stateful algorithm \eqref{eq:stochastic_stateful} with $\bar{k}=1$ recovers the vanilla algorithm \eqref{eq:stochastic_vanilla}. Further, the above stateful algorithm forms the backbone of the state-dependent stochastic approximation algorithm for controlled Markov chains \citep{li2022state} and motivates the epoch-based first-order algorithm in \citet{ray2022decision} that exploits the linear location-scale family structure.
}
For comparison, we leverage the proposed online stochastic algorithm \eqref{eq:stochastic_alg} interacting with the dynamics \eqref{eq:polarized_dynamics} with the following adjustments. First, we add a projection step within iterations to satisfy the norm constraint $\|q\| \leq 1$. Second, we use real-time samples to construct in an online manner the sensitivity $\widehat{H}(q_k,p_k)$ of the dynamics \eqref{eq:polarized_dynamics}, namely,
\begin{equation}\label{eq:sens_polarized_approx}
	\widehat{H}(q_k,p_k) = - \sigma \left(p_k^{\top}q_k I + p_k q_k^\top\right)(I- p_k p_k^{\top})\left[(\lambda I + \sigma q_k q_k^\top)(I - p_k p_k^{\top}) - \|\widetilde{p}_k\| I\right]^{-1},
\end{equation}
where $\widetilde{p}_k = \lambda p_{k-1} + (1-\lambda) p_0 + \sigma \cdot (p_{k-1}^{\top} q_k) q_k$. The difference between \eqref{eq:sens_polarized_approx} and \eqref{eq:sens_polarized_formula} lies in the replacement of $\pss$ with the current position $p_k$. \rev{In \cref{app:grad_err_extension}, we characterize the effect of this replacement on the gradient difference (similar to \cref{lem:grad_err_bound}) and show that the related theoretical guarantees hold analogously}. Overall, our online stochastic algorithm \eqref{eq:stochastic_alg} tailored to \eqref{eq:max_aff_pol} reads
\begin{equation}\label{eq:stochastic_composite}
	q_{k+1} = \proj_{\|q\| \leq 1}\left(q_k + \frac{\eta}{\nmb} \sum_{i=1}^{\nmb} \left(p_k^i + \widehat{H}(q_k,p_k^i) q_k\right) \right).
\end{equation}
Here the ascent-based update corresponds to maximization in \eqref{eq:max_aff_pol}.

\begin{figure}[!tb]
    \centering
    \begin{minipage}[c]{.47\columnwidth}
		\centering
		\subfloat[\rev{Relative optimality gap}]{\includegraphics[width=\columnwidth]{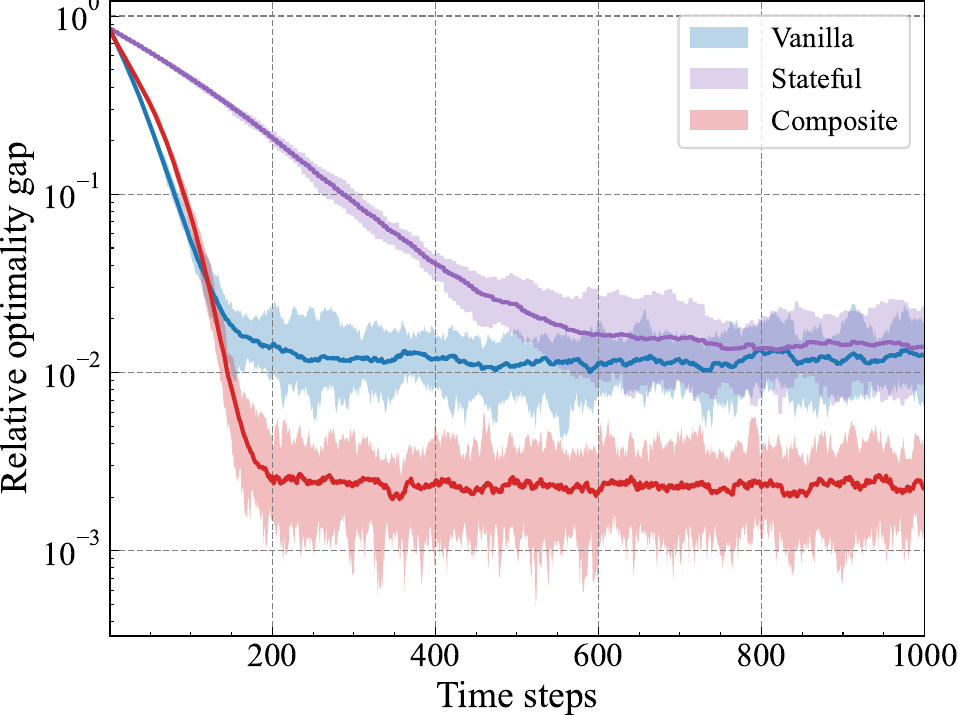}}
	\end{minipage}%
	\hfill
	\begin{minipage}[c]{.47\columnwidth}
		\centering
		\subfloat[\rev{Relative distance to the benchmark solution $q^*$}]{\includegraphics[width=\columnwidth]{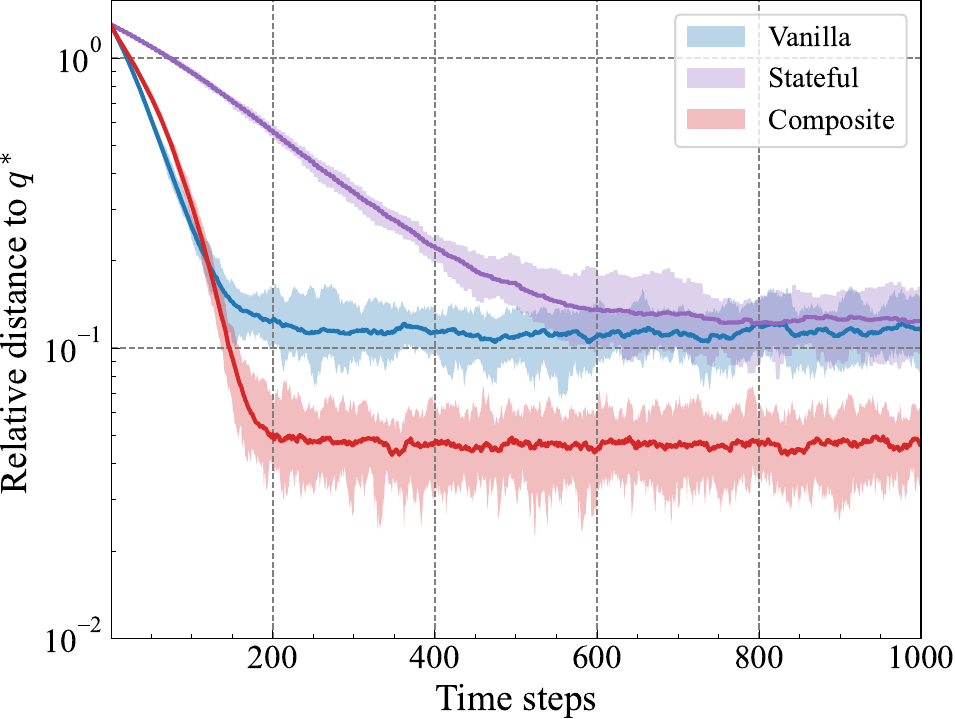}}
	\end{minipage}
    \caption{\rev{This figure illustrates the convergence behaviors of our online stochastic algorithm \eqref{eq:stochastic_composite} (termed ``composite''), the vanilla algorithm \eqref{eq:stochastic_vanilla} (termed ``vanilla''), and the stateful algorithm \eqref{eq:stochastic_stateful} (termed ``stateful'') for the normal-mixing scenario}. The solid lines represent the average values of convergence measures, whereas the shaded areas indicate the ranges of variation across independent trials.}
    \label{fig:affinity_evolution}
\end{figure}

\rev{
\subsubsection{Scenario 1: normal mixing}
We consider a normal-mixing scenario, where the coefficients in the distribution dynamics \eqref{eq:polarized_dynamics} are $\lambda = 0.4$ and $\sigma = 0.5$.} 
We set $p_k$ and $q$ to be $20$-dimensional vectors, i.e., $m=20$. We generate a population of $1000$ individuals and sample their initial positions from a unit hemisphere in the $20$-dimensional space, i.e., the set of unit vectors that form an angle not more than $90^{\circ}$ with a randomly generated reference vector. 
\rev{For the stochastic algorithms \eqref{eq:stochastic_vanilla}, \eqref{eq:stochastic_stateful}, and \eqref{eq:stochastic_composite}, we tune the step sizes over the grid spanning $[5\times 10^{-3}, 0.5]$. The selected step sizes that balance empirical convergence rates and accuracy are $0.125$, $0.175$, and $0.035$, respectively. The mini-batch size adopted in these algorithms is $\nmb = 50$. We set the epoch length $\bar{k}=5$ for the stateful algorithm \eqref{eq:stochastic_stateful}.} 
As a benchmark, we use the optimizer IPOPT \citep{wachter2006implementation} to calculate a (locally) optimal solution $q^*$ and the corresponding optimal value related to problem~\eqref{eq:max_aff_pol}. This optimizer starts from a random initial guess and accesses all the parameters and data across the population. In practice, however, such full access can be prohibitive, and our online algorithm \eqref{eq:stochastic_composite} is more desirable. We run $20$ independent trials of the stochastic algorithms \eqref{eq:stochastic_vanilla}, \eqref{eq:stochastic_stateful}, and \eqref{eq:stochastic_composite}. \rev{Given the hardware setting at the beginning of \cref{sec:experiment}, the total wall-clock runtime for all the algorithms is $6$ minutes and $40$ seconds.}



\cref{fig:affinity_evolution} illustrates the convergence behaviors of different algorithms while interacting with the polarized dynamics \eqref{eq:polarized_dynamics} and solving problem~\eqref{eq:max_aff_pol}. \rev{In \cref{fig:affinity_evolution}, each time step corresponds to one iteration of the dynamics \eqref{eq:polarized_dynamics}. The vanilla algorithm \eqref{eq:stochastic_vanilla} and our composite algorithm \eqref{eq:stochastic_composite} perform updates at every iteration, whereas the stateful algorithm \eqref{eq:stochastic_stateful} adjusts only once every $\bar{k}=5$ time steps. Overall, our algorithm \eqref{eq:stochastic_composite} achieves a lower optimality gap and attains a smaller distance to the benchmark solution compared with the vanilla \eqref{eq:stochastic_vanilla} and the stateful \eqref{eq:stochastic_stateful} algorithms.} The key reason stems from the algorithmic structure. The vanilla algorithm \eqref{eq:stochastic_vanilla} is oblivious to the composite characteristic of problem~\eqref{eq:max_aff_pol} and solely focuses on adaptation. \rev{While the stateful algorithm \eqref{eq:stochastic_stateful} employs periodic updates and incurs lower computational cost, it does not address such a composite characteristic.} In contrast, our algorithm~\eqref{eq:stochastic_composite} takes into account the composite problem structure and actively regulates the shift of distribution dynamics~\eqref{eq:polarized_dynamics}, thereby yielding smaller optimality gaps. We remark that the non-vanishing errors in \cref{fig:affinity_evolution} are due to the variance of the stochastic gradients constructed during iterations. Further experiments confirm that a larger mini-batch size $\nmb$ and a smaller (or diminishing) step size $\eta$ contribute to a lower asymptotic value of the convergence measure. \rev{Extensive experiments where the algorithms and IPOPT start from various random initial points yield consistent comparison results.}

\begin{figure}[!tb]
    \centering
    \begin{minipage}[c]{.45\columnwidth}
		\centering
		\subfloat[\rev{Case of the vanilla algorithm \eqref{eq:stochastic_vanilla}}]{\includegraphics[width=\linewidth]{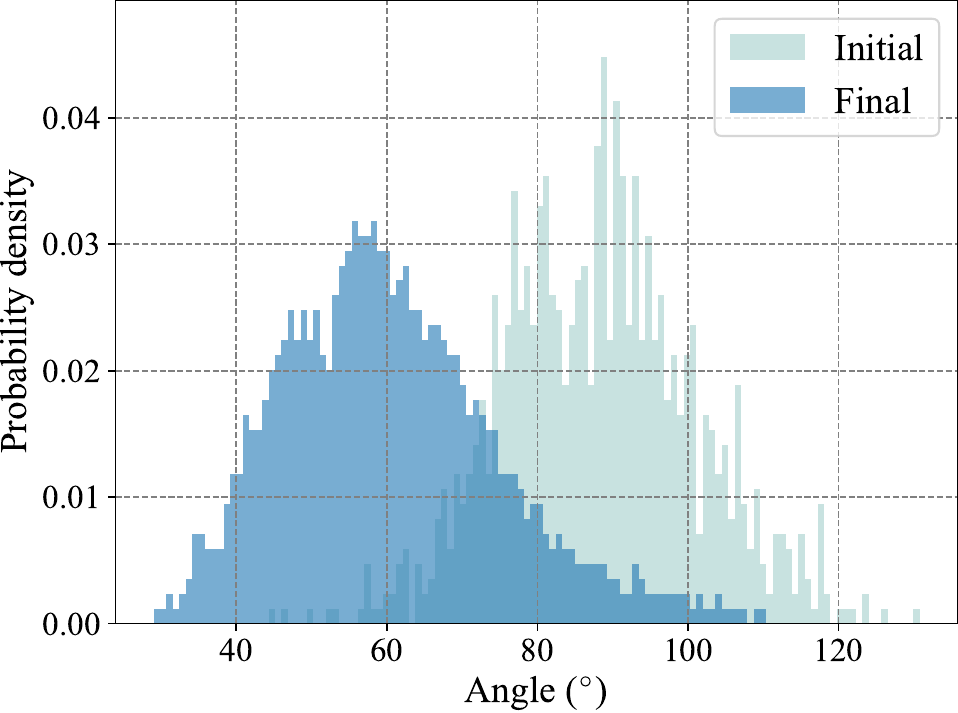} \label{fig:angle_hist_vanilla}} 
	\end{minipage}%
    \hspace{2ex}
    \begin{minipage}[c]{.45\columnwidth}
		\centering
		\subfloat[\rev{Case of the stateful algorithm \eqref{eq:stochastic_stateful}}]{\includegraphics[width=\linewidth]{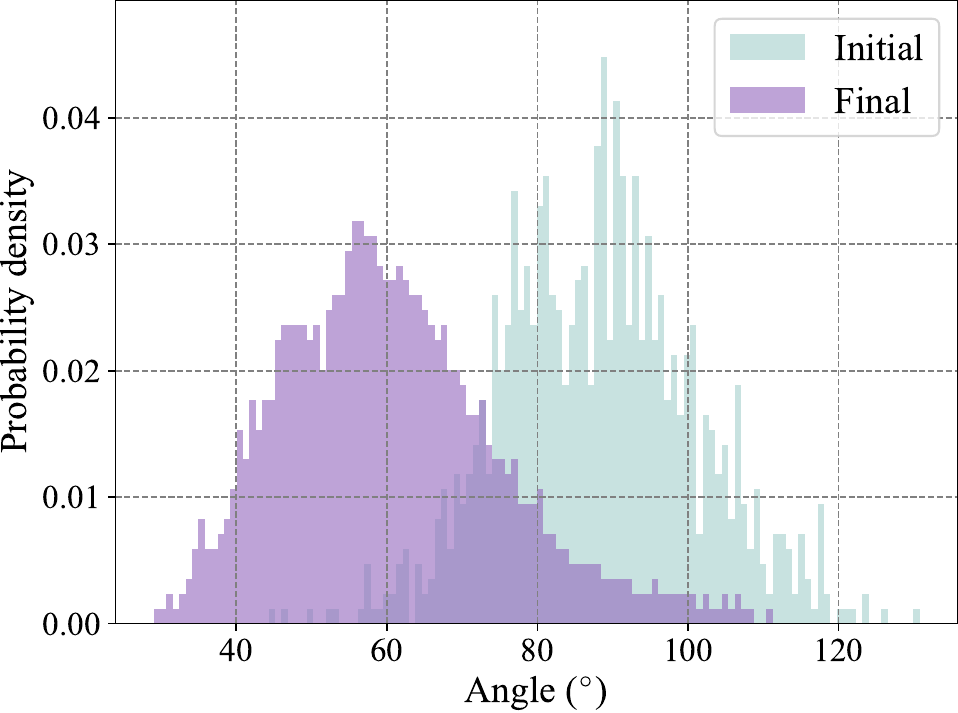} \label{fig:angle_hist_stateful}} 
	\end{minipage}
	\\
	\begin{minipage}[c]{.45\columnwidth}
		\centering
		\subfloat[\rev{Case of our composite algorithm \eqref{eq:stochastic_composite}}]{\includegraphics[width=\columnwidth]{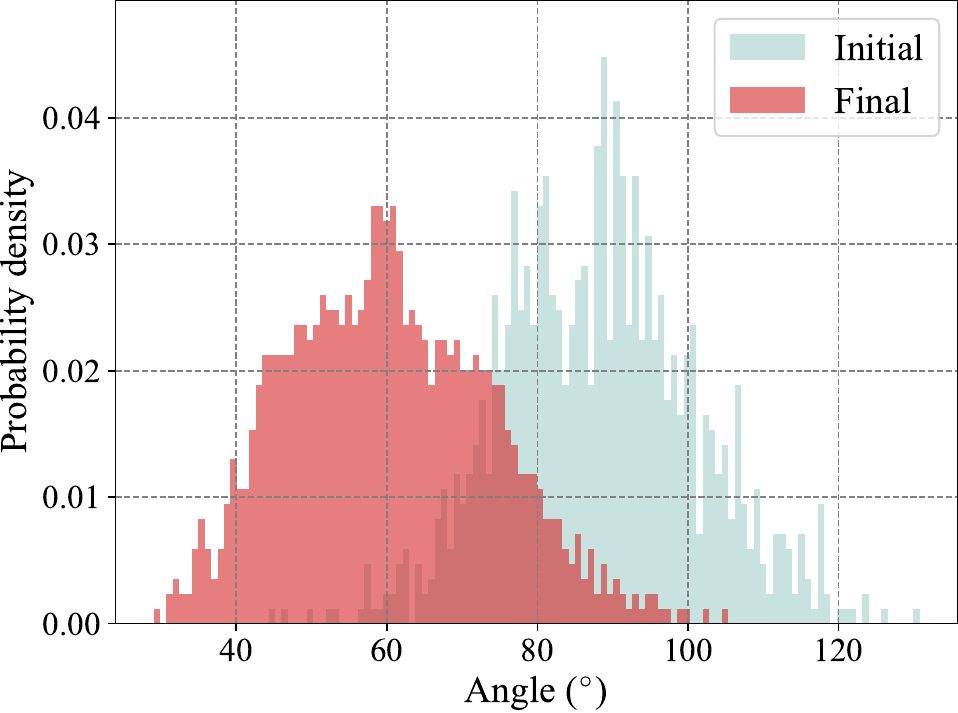}}
	\end{minipage}
    \hspace{2ex}
    \begin{minipage}[c]{.45\columnwidth}
		\centering
		\subfloat[\rev{Case with the benchmark solution $q^*$ via IPOPT}]{\includegraphics[width=\columnwidth]{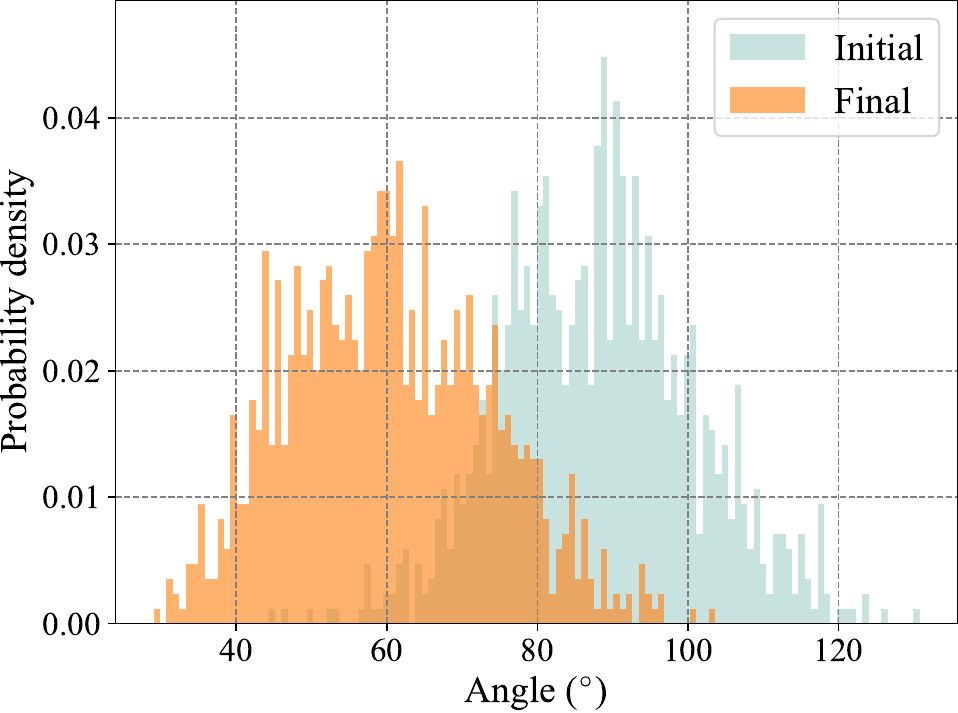}}
	\end{minipage}
    \caption{\rev{This figure illustrates the histograms of the angles between the position and the decision across the population in the normal-mixing scenario.} Initially, the angles are largely concentrated in $[80^{\circ}, 100^{\circ}]$, implying the average population-wide affinity is close to zero. In the end (i.e., when the population reaches the steady state), these angles mostly fall into $[40^\circ, 80^\circ]$, which indicates that the average affinity becomes positive.}
    \label{fig:angle_histogram}
\end{figure}

The above convergence behaviors are also linked with the histograms of the angles between the position and the decision across the population; see \cref{fig:angle_histogram}. Since the objective function of \eqref{eq:max_aff_pol} is an inner product of the position $p$ and the decision $q$, the angle between $p$ and $q$ is a useful indicator of the sign and the value of the inner product $p^\top q$. As shown in \cref{fig:angle_histogram}, the initial angles are largely distributed in $[80^\circ, 100^\circ]$, rendering the average population-wide affinity close to zero. \rev{The benchmark solution returned by IPOPT and the solutions obtained by the algorithms \eqref{eq:stochastic_vanilla}, \eqref{eq:stochastic_stateful}, and \eqref{eq:stochastic_composite} all result in a final configuration wherein most angles fall into $[40^\circ, 80^\circ]$.} Hence, the final average affinity is positive. \rev{Since the vanilla \eqref{eq:stochastic_vanilla} and the stateful \eqref{eq:stochastic_stateful} algorithms focus less on steering the distribution, in \cref{fig:angle_hist_vanilla,fig:angle_hist_stateful}, there are relatively more final angles in $[100^{\circ}, 110^{\circ}]$}, which accounts for a smaller affinity and a larger optimality gap as shown in \cref{fig:affinity_evolution}. In contrast, the solution found by our algorithm \eqref{eq:stochastic_composite} induces population-wide behavior nearly identical to the IPOPT benchmark.

\rev{
\subsubsection{Scenario 2: strong polarization}
We further examine a strong-polarization regime and demonstrate performance comparisons. In this setting, the coefficients in \eqref{eq:polarized_dynamics} are set to $\lambda=0.9$ and $\sigma=0.2$, indicating that the initial position only exerts a weak influence. The problem dimension and the population size are identical to those in the normal-mixing scenario. The angular range of the unit spherical cap used for sampling initial positions is $110^\circ$. We tune the step sizes of the algorithms \eqref{eq:stochastic_vanilla}, \eqref{eq:stochastic_stateful}, and \eqref{eq:stochastic_composite} over the grid spanning $[5\times 10^{-3}, 0.5]$ and settle on $0.375$, $0.425$, and $0.325$, respectively, to balance empirical convergence rates and accuracy. The mini-batch size $\nmb=50$ and the epoch length $\bar{k}=5$ are the same. Under the hardware configuration described at the beginning of \cref{sec:experiment}, the wall-clock runtime covering all the algorithms is $8$ minutes and $10$ seconds.

In the scenario with $\lambda=0.9$, the dynamics \eqref{eq:polarized_dynamics} converge to the steady state more slowly than in the normal-mixing case. Although theoretical intuition suggests that small step sizes may be necessary for convergence, the experimental results indicate the opposite, in that moderate step sizes remain effective. This phenomenon may arise from the polarization property of \eqref{eq:polarized_dynamics}. Since the influence of the initial position is limited, the relative steering ability of the decision-maker becomes stronger, thereby facilitating the co-evolution of fast polarization and algorithmic updates. 

\begin{figure}[!tb]
    \centering
    \begin{minipage}[c]{.47\columnwidth}
		\centering
		\subfloat[\rev{Relative optimality gap}]{\includegraphics[width=\columnwidth]{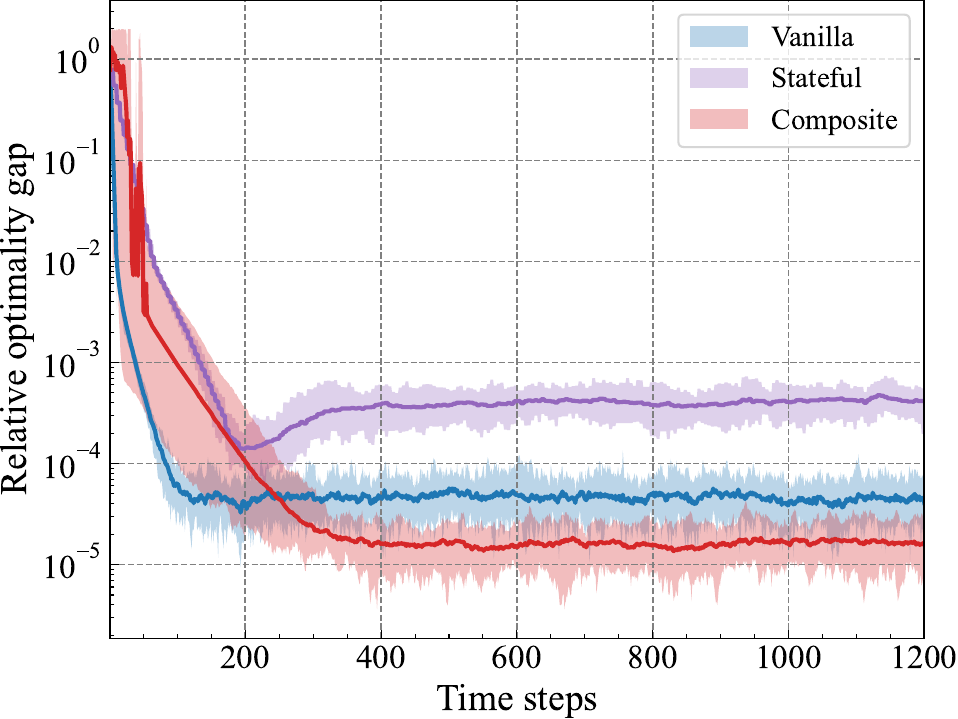}}
	\end{minipage}%
	\hfill
	\begin{minipage}[c]{.47\columnwidth}
		\centering
		\subfloat[\rev{Relative distance to the benchmark solution $q^*$}]{\includegraphics[width=\columnwidth]{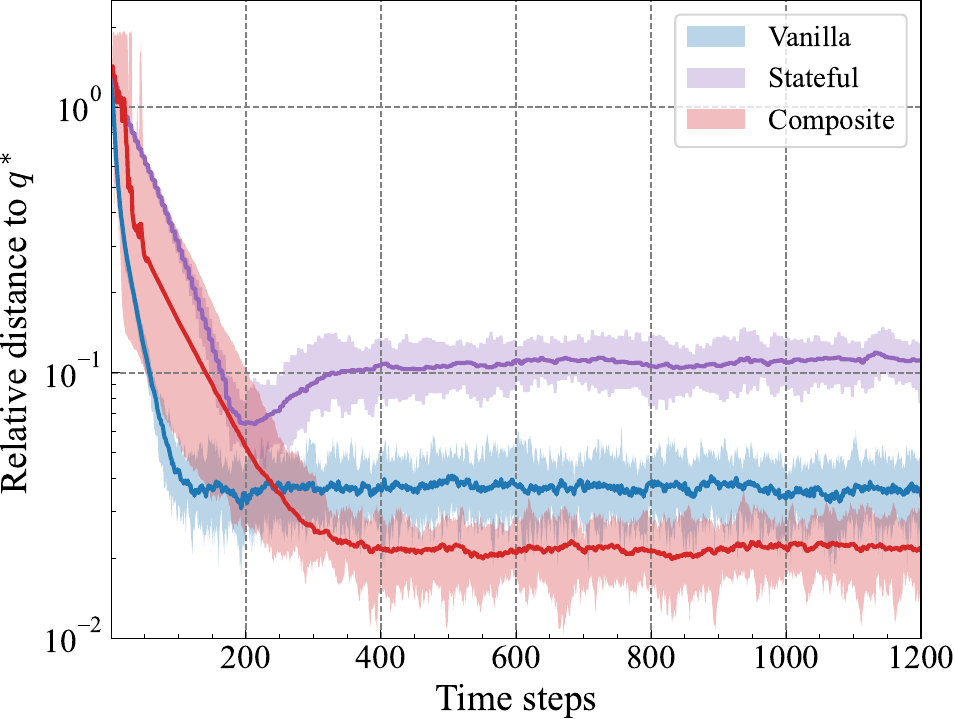}}
	\end{minipage}
    \caption{\rev{This figure illustrates the convergence behaviors of our online stochastic algorithm \eqref{eq:stochastic_composite} (termed ``composite''), the vanilla algorithm \eqref{eq:stochastic_vanilla} (termed ``vanilla''), and the stateful algorithm \eqref{eq:stochastic_stateful} (termed ``stateful'') for the strong-polarization scenario. The solid lines represent the average values of convergence measures, whereas the shaded areas indicate the ranges of variation across independent trials.}}
    \label{fig:affinity_evolution_strong_polarized}
\end{figure}

\cref{fig:affinity_evolution_strong_polarized} demonstrates the convergence behaviors of different algorithms in this strong-polarization scenario. Each unit increment on the horizontal axis of \cref{fig:affinity_evolution_strong_polarized} (i.e., time step) represents one iteration of the polarized dynamics \eqref{eq:polarized_dynamics}. Thanks to exploiting the composite problem structure, our algorithm \eqref{eq:stochastic_composite} achieves a lower optimality gap and a smaller distance to the benchmark solution. Such performance improvements are less pronounced than in the normal-mixing case, as reflected in the angle histograms of different algorithms in \cref{fig:angle_histogram_strong_polarized}. The initial angles largely lie within $[75^\circ, 100^\circ]$. However, in sharp contrast to \cref{fig:angle_histogram}, most final angles fall in $[15^\circ, 25^\circ]$. The stateful algorithm \eqref{eq:stochastic_stateful} causes polarization, as shown by the two peaks in the final distribution in \cref{fig:angle_hist_stateful_sp}. The vanilla algorithm \eqref{eq:stochastic_vanilla} and our algorithm \eqref{eq:stochastic_composite} yield final configurations similar to the optimal benchmark, though \cref{fig:angle_hist_composite_sp} exhibits a relatively higher concentration in the region below $22^\circ$ than \cref{fig:angle_hist_vanilla_sp}.
}


\begin{figure}[!tb]
    \centering
    \begin{minipage}[c]{.45\columnwidth}
		\centering
		\subfloat[\rev{Case of the vanilla algorithm \eqref{eq:stochastic_vanilla}}]{\includegraphics[width=\linewidth]{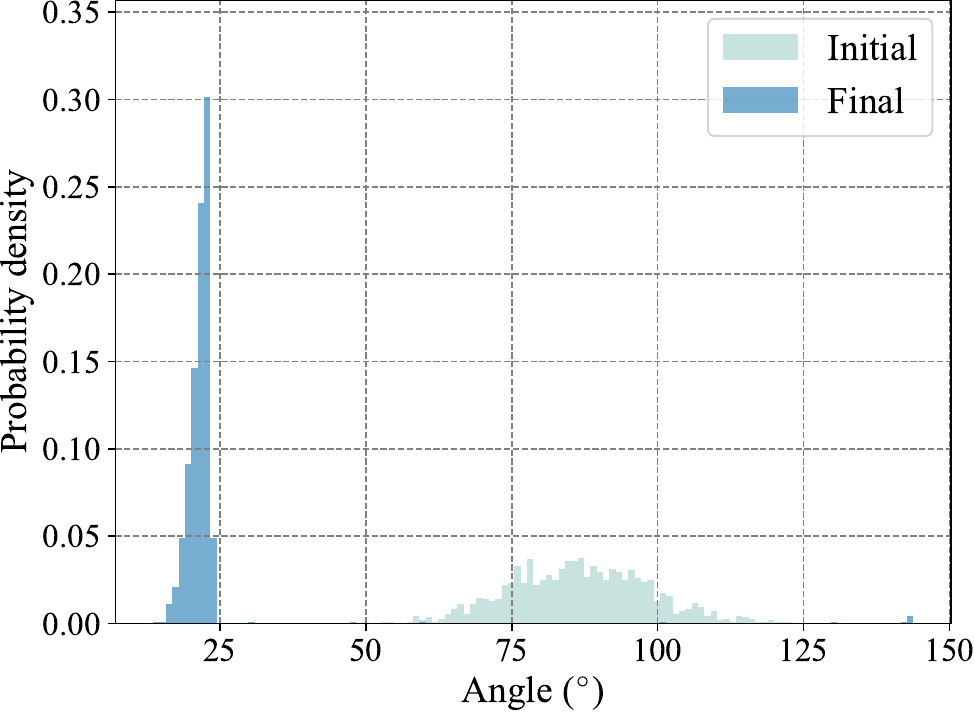} \label{fig:angle_hist_vanilla_sp}} 
	\end{minipage}%
    \hspace{2ex}
    \begin{minipage}[c]{.45\columnwidth}
		\centering
		\subfloat[\rev{Case of the stateful algorithm \eqref{eq:stochastic_stateful}}]{\includegraphics[width=\linewidth]{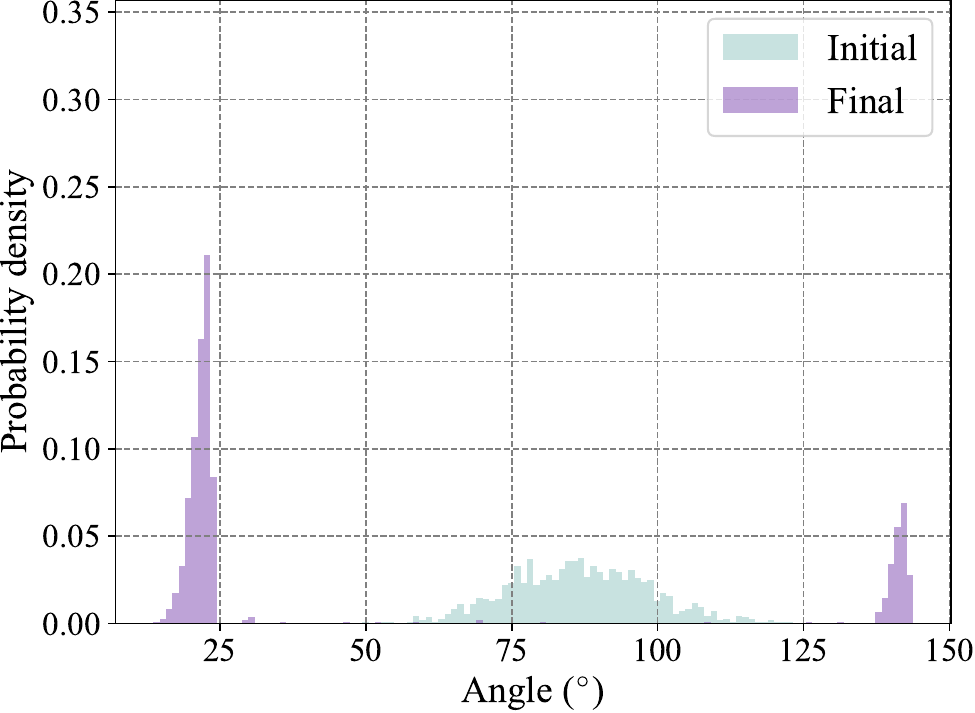} \label{fig:angle_hist_stateful_sp}} 
	\end{minipage}
	\\
	\begin{minipage}[c]{.45\columnwidth}
		\centering
		\subfloat[\rev{Case of our composite algorithm \eqref{eq:stochastic_composite}}]{\includegraphics[width=\columnwidth]{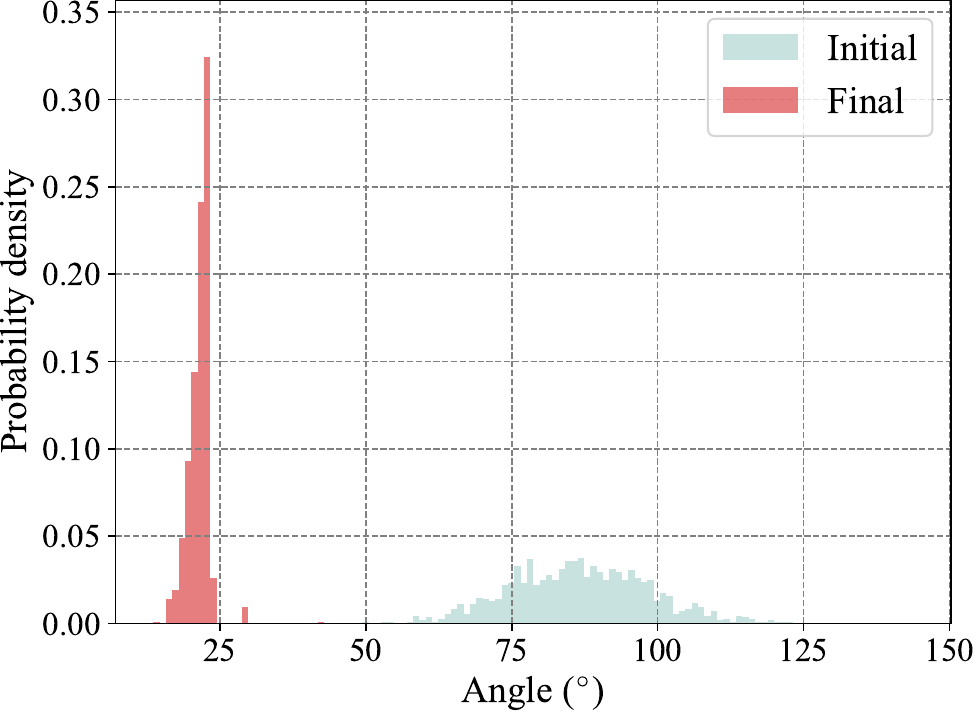}\label{fig:angle_hist_composite_sp}}
	\end{minipage}
    \hspace{2ex}
    \begin{minipage}[c]{.45\columnwidth}
		\centering
		\subfloat[\rev{Case with the benchmark solution $q^*$ via IPOPT}]{\includegraphics[width=\columnwidth]{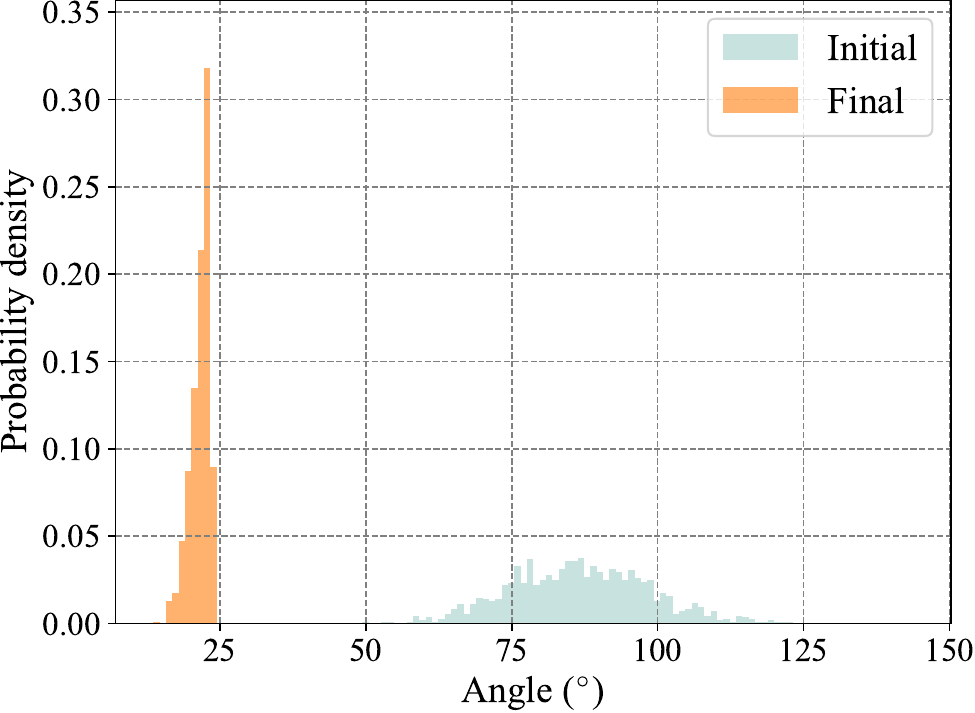}}
	\end{minipage}
    \caption{\rev{This figure illustrates the histograms of the angles between the position and the decision across the population in the strong-polarization scenario. The initial angles mostly fall in $[75^\circ, 100^\circ]$. However, due to the polarization property of \eqref{eq:polarized_dynamics} and the stronger influence of the decision compared to the initial position, the final angles largely lie in $[15^\circ, 25^\circ]$. In \cref{fig:angle_hist_stateful_sp}, the stateful algorithm \eqref{eq:stochastic_stateful} induces polarization, as shown by the two peaks in the distribution of final angles.}
    }
    \label{fig:angle_histogram_strong_polarized}
\end{figure}

\subsection{Performance optimization given discrete choice distributions}\label{subsec:hedge_discrete_distri}
\cref{subsec:polarized_pop} addresses affinity maximization involving continuous population distributions. In contrast, we now consider decision-making given discrete distributions lying in the probability simplex $\Delta^{m-1} = \{p \in \mathbb{R}^m \mid 1^\top p = 1, \rev{0 \leq p(i) \leq 1}, i=1,\ldots,m\}$, where $m \in \mathbb{N}_{+}$, \rev{and $p(i)$ denotes the $i$-th element of $p$}. This setting is motivated by a recommender system scenario, where a user selects a specific item (e.g., product, movie, or plan) from a set of candidates, and the choice model is characterized by a discrete distribution in $\Delta^{m-1}$. The recommender aims to optimize certain performance measures (e.g., gain and diversity) by interacting with this user and adjusting its decision. This setting is also related to the multi-armed bandit problem, in that the user action described by a choice distribution $p \in \Delta^{m-1}$ generates performance feedback determined by the decision-maker.

We consider the following distribution dynamics
\begin{equation}\label{eq:softmax_dynamics}
	p_{k+1} = \lambda_1 p_k + \lambda_2 \softmax(-\epsilon q) + (1-\lambda_1-\lambda_2) p_0, \qquad k \in \mathbb{N},
\end{equation}
where $p_k \in \Delta^{m-1}$ denotes the choice distribution at time $k$, $q \in \mathbb{R}^m$ is a decision vector whose $i$-th element represents, e.g., the price or loss of item $i$ in a set, $\lambda_1, \lambda_2 \in (0,1)$ are combination coefficients that satisfy $\lambda_1+\lambda_2 \leq 1$, and \rev{the $i$-th element of $\softmax(-\epsilon q)$ is given by $\exp(-\epsilon q(i))/\sum_{j=1}^{m} \exp(-\epsilon q(j))$}, where $\epsilon > 0$ is a coefficient, and $i=1,\ldots,m$. The term $\softmax(-\epsilon q)$ gives a distribution consisting of probabilities proportional to the exponential of $-\epsilon q$. One interpretation of this softmax operation is that an item $i$ with a low price $q(i)$ is selected with a high probability. As in \eqref{eq:polarized_dynamics}, the last term of \eqref{eq:softmax_dynamics} captures the persistent influence of the initial choice distribution. The steady-state distribution of \eqref{eq:softmax_dynamics} given a fixed decision $q$ is
\begin{equation}\label{eq:steady_state_softmax}
	\pss = \frac{\lambda_2}{1-\lambda_1} \softmax(-\epsilon q) + \frac{1-\lambda_1-\lambda_2}{1-\lambda_1} p_0 \triangleq h(q, p_0).
\end{equation}

The recommender aims to maximize its expected gain while preserving diversity, subject to certain budget requirements. This is formalized by the following problem
\begin{equation}\label{eq:max_adv_entropy}
\begin{split}
	\max_{q \in \mathbb{R}^m} \quad& p^\top q - \rho \sum_{i=1}^{m} p(i) \log p(i) \\
	\text{s.t.} \quad& p = h(q, p_0), \\
		&1^\top q = b, \\
		&0 \leq q(i) \leq \overline{q}, \qquad i=1,\ldots,m.
\end{split}
\end{equation}
In problem \eqref{eq:max_adv_entropy}, the first term $p^\top q$ of the objective function represents the expected value given the steady-state discrete distribution $p=h(q,p_0)$; see also \eqref{eq:steady_state_softmax}. This term indicates the expected price paid by the user, or conversely, the expected gain of the recommender. \rev{The second term of the objective is an entropy regularization term with the coefficient $\rho > 0$}, and its purpose is to promote diversity by favoring high-entropy distributions $p$. The remaining constraints of problem \eqref{eq:max_adv_entropy} specify budget requirements of the decision $q$, where $b$ and $\overline{q} > 0$ are constants.

Problem~\eqref{eq:max_adv_entropy} involves the underlying distribution dynamics \eqref{eq:softmax_dynamics}. An offline numerical solver hinges on an explicit and accurate model of \eqref{eq:softmax_dynamics}, 
and therefore it can be restrictive in practice. To solve problem~\eqref{eq:max_adv_entropy} in an online fashion, the vanilla algorithm, which is inspired by performative prediction \citep{hardt2023performative} and is oblivious to the composite problem structure, reads as follows
\begin{equation}\label{eq:stochastic_vanilla_simplex}
	q_{k+1} = \proj_{\mathcal{C}}(q_k + \eta p_k),
\end{equation}
where $\proj_{\mathcal{C}}(\cdot)$ denotes the projection to the feasible set $\mathcal{C} = \{q \in \mathbb{R}^m | 1^\top q = b, \rev{0 \leq q(i) \leq \overline{q}}\}$, \rev{$q(i)$ is the $i$-th element of $q$}, $\eta > 0$ is the step size, and $p_k$ is obtained after $q_k$ is applied to the dynamics \eqref{eq:softmax_dynamics} with previous state $p_{k-1}$. 
\rev{Similar to \eqref{eq:stochastic_stateful}, we implement the following stateful algorithm with periodic updates based on \citet{brown2022performative}
\begin{equation}\label{eq:stochastic_stateful_simplex}
    q_{k+1} =
    \begin{cases}
        \proj_{\mathcal{C}}(q_k + \eta p_k), & \text{if } k \equiv 0 \pmod {\bar{k}},  \\
        q_k, & \text{otherwise},
    \end{cases}
\end{equation}
where $\bar{k} \in \mathbb{N}_{+}$ is the epoch length. As discussed in \cref{subsec:polarized_pop}, subsequent stateful algorithms \citep{ray2022decision,li2022state} adopt the same periodic updating structure and apply to various evolving distribution dynamics, including linear mixtures or geometrically ergodic Markov chains. 
Another choice is the derivative-free algorithm in \citet{zhang2022new} that relies on the evaluations of the objective $\Phi(q,p) \triangleq p^\top q - \rho \sum_{i=1}^{m} \rev{p(i) \log p(i)}$:
}
\begin{equation}\label{eq:stochastic_dfo_simplex}
	q_{k+1} = \proj_{\mathcal{C}}\left(q_k + \eta \frac{m}{\delta}(\Phi(q_k + \delta v_k, p_k) - \Phi(q_{k-1}+\delta v_{k-1}, p_{k-1}))v_k\right),
\end{equation}
where $\eta > 0$ is the step size, $m$ is the dimension of $q_k$, $\delta > 0$ is the smoothing parameter, $v_k,v_{k-1}$ are independent random vectors uniformly sampled from the unit sphere in $\mathbb{R}^m$, $q_k + \delta v_k$ is the deployed decision that generates $p_k$, and $\Phi(q_k + \delta v_k, p_k)$ is the objective value evaluated at time $k$. Algorithm~\eqref{eq:stochastic_dfo_simplex} exploits bandit feedback of the objective, and a similar version of \eqref{eq:stochastic_dfo_simplex} is adopted in \citet{miller2021outside,ray2022decision}. In contrast, our algorithm taking into account the composite structure due to decision dependence is
\begin{equation}\label{eq:stochastic_composite_simplex}
	q_{k+1} = \proj_{\mathcal{C}}\Big(q_k + \eta \big(p_k + H_k(q_k - \rho(1 + \log p_k)) \big) \Big).
\end{equation}
In \eqref{eq:stochastic_composite_simplex}, $\eta>0$ is the step size, and the distribution sensitivity $H_k$ (independent of $p_k$ and $p_0$) is
\begin{equation}\label{eq:sens_hedge_approx}
	H_k = \nabla_q h(q_k,p_0) = -\frac{\epsilon \lambda_2}{1-\lambda_1} \cdot \frac{(1^\top z_k)\operatorname{diag}(z_k) - z_k z_k^\top}{(1^\top z_k)^2},
\end{equation}
where the vector $z_k \in \mathbb{R}^m$ is given by $\exp(-\epsilon q_k)$, and $\operatorname{diag}(z_k)$ denotes a diagonal matrix with $z_k$ as its diagonal elements. \rev{The ascent-based updates in \eqref{eq:stochastic_vanilla_simplex}, \eqref{eq:stochastic_stateful_simplex}, \eqref{eq:stochastic_dfo_simplex}, and \eqref{eq:stochastic_composite_simplex} are due to maximization in \eqref{eq:max_adv_entropy}.} Since the recommender directly interacts with a user endowed with the choice distribution $p_k$, the update \eqref{eq:stochastic_composite_simplex} differs slightly from \eqref{eq:stochastic_alg}, although both are grounded in the same principle.

In experiments, the choice distribution $p$ and the decision $q$ are $100$-dimensional vectors, i.e., $m=100$. For the distribution dynamics \eqref{eq:softmax_dynamics}, we set $\lambda_1 = 0.2, \lambda_2 = 0.5$, and $\epsilon = 0.5$. In problem~\eqref{eq:max_adv_entropy}, we set the regularization coefficient $\rho = 0.1$, the budget $b = 250$, and the upper bound $\overline{q} = 5$. 
\rev{For the algorithms \eqref{eq:stochastic_vanilla_simplex}, \eqref{eq:stochastic_stateful_simplex}, and \eqref{eq:stochastic_composite_simplex}, we tune the step sizes over the grid spanning $[0.1, 5]$ and select the values $1.5$, $1.5$, and $1.25$, respectively, to balance empirical convergence rates and accuracy. For the stateful algorithm \eqref{eq:stochastic_stateful_simplex}, we set the epoch length $\bar{k}=5$. For the algorithm \eqref{eq:stochastic_dfo_simplex}, we search over $\eta$ and $\delta$ on the grid spanning $[5\times 10^{-2}, 1] \times [0.5, 5]$, settle on $\eta=0.15, \delta=1.5$, and run $20$ independent trials.} 
We use IPOPT \citep{wachter2006implementation} to calculate a (locally) optimal solution of problem~\eqref{eq:max_adv_entropy} in an offline manner as a benchmark. However, as discussed earlier, numerical solvers can be restrictive in practice when the dynamics model \eqref{eq:softmax_dynamics} is inaccurate and real-time decision-making is necessary.
\rev{Given the hardware setting stated at the beginning of \cref{sec:experiment}, the wall-clock runtime to evaluate all the algorithms is $38$ seconds.}

\begin{figure}[!tb]
    \centering
    \begin{minipage}[c]{.49\columnwidth}
		\centering
		\subfloat[Relative optimality gap]{\includegraphics[width=\columnwidth]{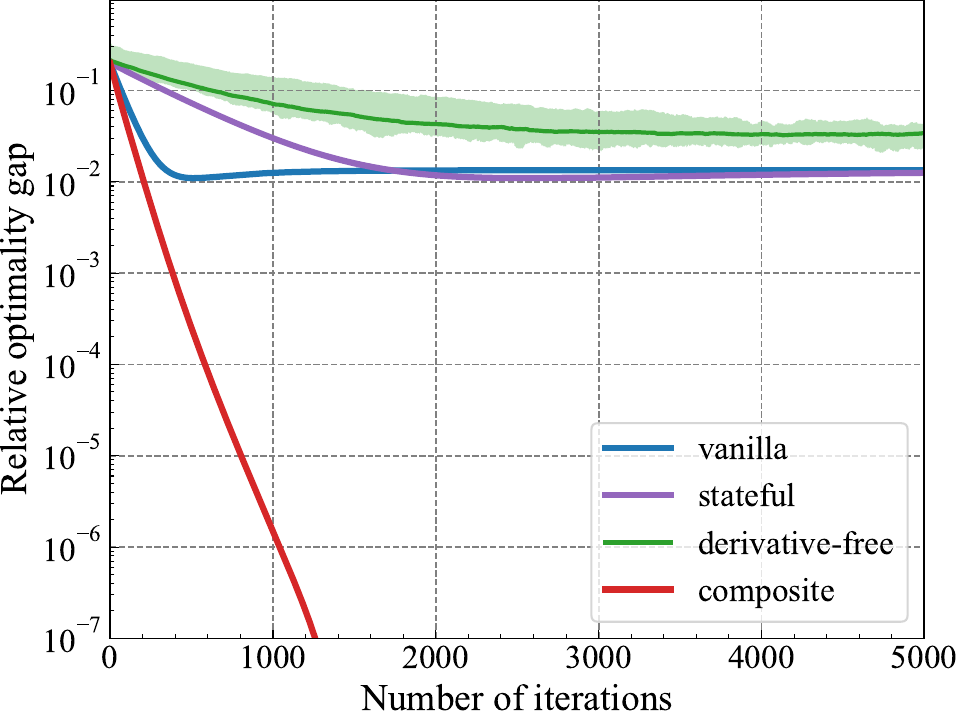}}
	\end{minipage}%
	\hfill
	\begin{minipage}[c]{.49\columnwidth}
		\centering
		\subfloat[Relative distance to the benchmark solution $q^*$]{\includegraphics[width=\columnwidth]{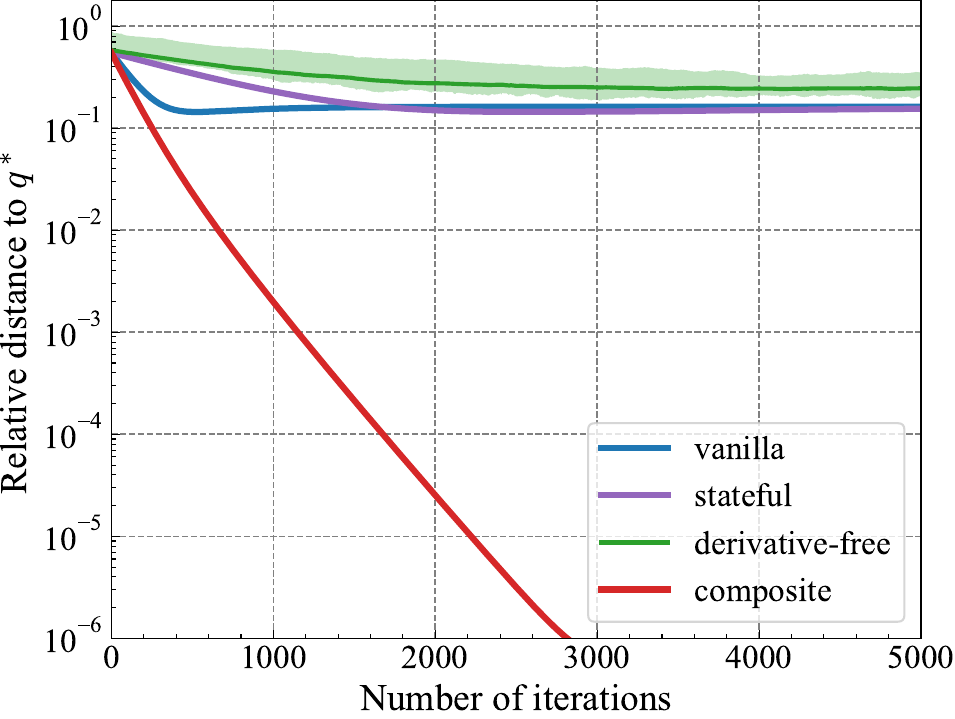}}
	\end{minipage}
    \caption{\rev{This figure compares the convergence behaviors of the vanilla algorithm \eqref{eq:stochastic_vanilla_simplex} (labeled ``vanilla''), the stateful algorithm \eqref{eq:stochastic_stateful_simplex} (labeled ``stateful''), the derivative-free algorithm \eqref{eq:stochastic_dfo_simplex} (labeled ``derivative-free''), and our online stochastic algorithm \eqref{eq:stochastic_composite_simplex} (labeled ``composite'') when interacting with the distribution dynamics \eqref{eq:softmax_dynamics} to solve problem~\eqref{eq:max_adv_entropy}.}}
    \label{fig:loss_evolution}
\end{figure}

\rev{\cref{fig:loss_evolution} shows the convergence results of online algorithms \eqref{eq:stochastic_vanilla_simplex}, \eqref{eq:stochastic_stateful_simplex}, \eqref{eq:stochastic_dfo_simplex}, and \eqref{eq:stochastic_composite_simplex} while interacting with the dynamics \eqref{eq:softmax_dynamics} to solve problem~\eqref{eq:max_adv_entropy}. Here, each time step corresponds to one iteration of \eqref{eq:softmax_dynamics}.} The vanilla algorithm \eqref{eq:stochastic_vanilla_simplex} is online in nature, and therefore it adapts to the changing distribution \eqref{eq:softmax_dynamics}, eventually reaching a fixed-point solution. However, this algorithm suffers from significant sub-optimality, because it neglects the composite problem structure where the steady-state distribution is a function of the decision. \rev{The stateful algorithm \eqref{eq:stochastic_stateful_simplex} reduces computational costs thanks to its periodic updates. Nonetheless, this algorithm requires more interactions with the dynamics \eqref{eq:softmax_dynamics} than the vanilla algorithm \eqref{eq:stochastic_vanilla_simplex} and still incurs sub-optimality.} 
The derivative-free algorithm~\eqref{eq:stochastic_dfo_simplex} converges slowly due to the stochasticity of gradient estimates. 
In contrast, our algorithm \eqref{eq:stochastic_composite_simplex} respects the composite structure by leveraging distribution sensitivities and regulating the shift of the choice distribution. Consequently, it attains significantly higher solution accuracy. We remark that compared to \cref{fig:affinity_evolution} in \cref{subsec:polarized_pop}, the curves of algorithms \eqref{eq:stochastic_vanilla_simplex} and \eqref{eq:stochastic_composite_simplex} in \cref{fig:loss_evolution} do not exhibit stochastic oscillations, because we no longer resort to samples of a population distribution for mini-batch gradients.

The aforementioned difference in solution quality also causes the discrepancy in the evolution of the Wasserstein distances $W_1(p_k, \pss(q^*))$ between the dynamic choice distribution $p_k$ and the steady-state distribution $\pss(q^*)$ induced by the benchmark decision $q^*$ via IPOPT\@. To evaluate $W_1(p_k, \pss(q^*))$, we regard both vectors as distributions on the common support $\{1,2,\ldots,m\}$ endowed with the ground metric $c(i,j) = |i - j|$, where $i,j=1,\ldots,m$. \cref{fig:wass_dist} illustrates such a difference in $W_1(p_k, \pss(q^*))$. Overall, these distances converge to limiting values, thanks to the stable distribution dynamics \eqref{eq:softmax_dynamics} and the convergence of the algorithms to fixed decisions. We note that the vanilla algorithm \eqref{eq:stochastic_vanilla_simplex}, the stateful algorithm \eqref{eq:stochastic_stateful_simplex}, and the derivative-free algorithm \eqref{eq:stochastic_dfo_simplex} bring about biased final choice distributions, whereas our algorithm \eqref{eq:stochastic_composite_simplex} drives the distribution to closely approach $\pss(q^*)$ corresponding to the benchmark decision. As pointed out earlier, this major difference is due to our algorithmic structure, which explicitly addresses the composite problem characteristic resulting from decision dependence.

\begin{figure}[!tb]
    \centering
    \includegraphics[width=0.5\columnwidth]{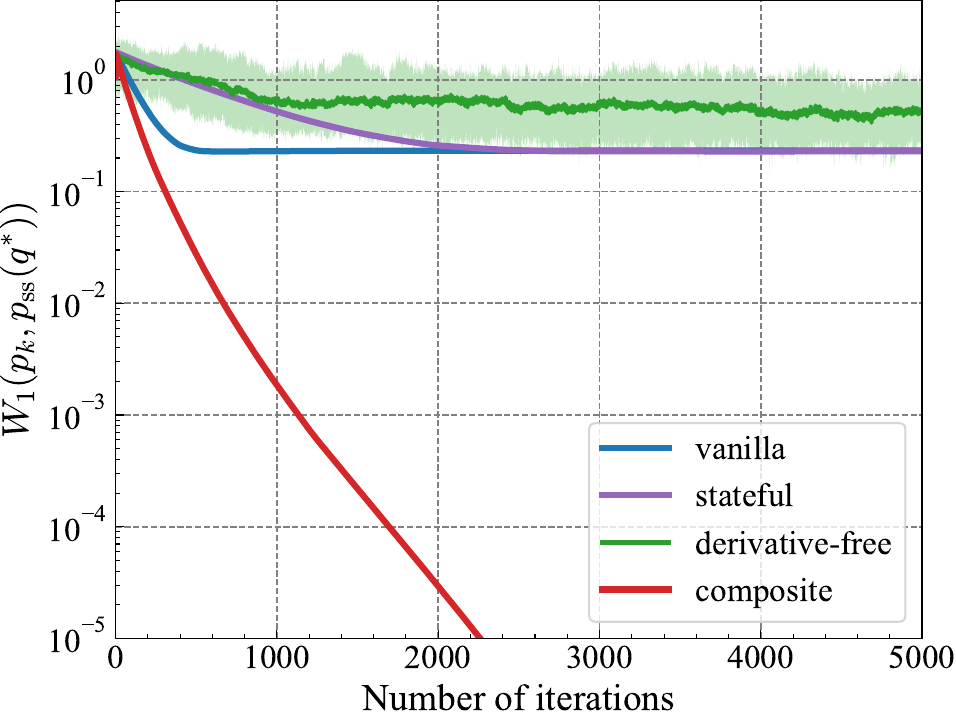}
    \caption{\rev{This figure demonstrates the evolution of the Wasserstein distances $W_1(p_k, \pss(q^*))$ between the choice distribution $p_k$ and the steady-state distribution $\pss(q^*)$ induced by $q^*$. The legends ``vanilla'', ``stateful'', ``derivative-free'', and ``composite'' correspond to algorithms \eqref{eq:stochastic_vanilla_simplex}, \eqref{eq:stochastic_stateful_simplex}, \eqref{eq:stochastic_dfo_simplex}, and \eqref{eq:stochastic_composite_simplex}, respectively.}}
    \label{fig:wass_dist}
\end{figure}

%% file: section/conclusion.tex
\section{Conclusion}\label{sec:conclusion}
We formulated a decision-dependent stochastic optimization problem, where the dependence originates from the closed-loop interaction between a decision-maker and a dynamically evolving distribution. We presented an online stochastic algorithm that leverages samples from the dynamic distribution, takes into account the composite structure of the problem by anticipating the sensitivity of the distribution with respect to the decision, and shapes the overall distribution to inform optimal decision-making. We established the optimality guarantees of the proposed algorithm both in expectation and with high probability. Furthermore, we quantified the generalization performance in a finite-sample regime with an empirical distribution.

Future directions include, but are not limited to, i) addressing constraints related to the decision and the distribution (represented by, e.g., bounds on risk measures), \rev{ii) integrating sensitivity-learning schemes and quantifying the impact of sensitivity modeling or learning errors on performance guarantees}, iii) developing model-free methods (through, e.g., derivative-free optimization) that bypass the need for sensitivity matrices of distribution dynamics, \rev{and iv) extending case studies to address general non-contracting opinion dynamics beyond the combination structure in \cref{subsec:polarized_pop}}. We hope this work sparks further advances at the intersection of machine learning, stochastic optimization, and nonlinear control.


%% file: section/appendix.tex
\crefalias{section}{appendix}
\crefalias{subsection}{appendix}
\section{Useful Lemmas}\label{app:lemmas}
First, we present a lemma that establishes the upper and lower bounds on the weighted norm.
\begin{lemma}\label{lem:weighted_norm}
	Let $x \in \mathbb{R}^m$, and let $P \in \mathbb{R}^{m\times m}$ be positive definite. Then,
	\begin{equation}\label{eq:weight_norm_bd}
		\sqrt{\lambda_{\min}(P)} \|x\| \leq \|x\|_P \leq \sqrt{\lambda_{\max}(P)} \|x\|.
	\end{equation}
\end{lemma}
\begin{proof}
	For a positive definite matrix $P$, both $\lambda_{\max}(P)I - P$ and $P - \lambda_{\min}(P)I$ are positive semidefinite. Hence, for any $x \in \mathbb{R}^m$, $\lambda_{\max}(P) x^\top x - x^\top P x \geq 0$ and $ x^\top P x - \lambda_{\min}(P) x^\top x \geq 0$. 
	By rearranging terms and taking the square root, we know that \eqref{eq:weight_norm_bd} holds.
\end{proof}

Some implications of \cref{lem:weighted_norm} are as follows. If a function $\Psi: \mathbb{R}^m \to \mathbb{R}$ is $L_\Psi$-Lipschitz with respect to $\|\cdot\|_P$, then it is $L_\Psi \sqrt{\lambda_{\max}(P)}$-Lipschitz with respect to $\|\cdot\|$. Conversely, if $\Psi$ is $L'_\Psi$-Lipschitz with respect to $\|\cdot\|$, then it is $L'_\Psi/\sqrt{\lambda_{\min}(P)}$-Lipschitz with respect to $\|\cdot\|_P$.

The following lemma provides upper bounds related to the square of a nonnegative sequence.
\begin{lemma}\label{lem:seq_square_sum}
    Suppose that a nonnegative sequence $(a_k)_{k\in \mathbb{N}}$ satisfies $a_{k+1} \leq c a_k + b_k, \forall k \in \mathbb{N}$, where $c \in (0,1), a_k, b_k \in \mathbb{R}$. Then,
    \begin{subequations}
    \begin{align}
    	a_{k+1}^2 &\leq \frac{1+c^2}{2} a_k^2 + \frac{1+c^2}{1-c^2} b_k^2, \label{eq:square_recursive_ineq} \\
        \sum_{k=0}^{T-1} a_k^2 &\leq \frac{2}{1-c^2} \left(a_0^2 + \frac{1+c^2}{1-c^2} \sum_{k=0}^{T-1} b_k^2\right). \label{eq:decreasing_seq_sum_bd}
    \end{align}
    \end{subequations}
\end{lemma}
\begin{proof}
	We establish a recursive inequality of $(a_k^2)_{k \in \mathbb{N}}$ as follows
	\begin{align*}
		a_{k+1}^2 \leq c^2 a_k^2 + 2c a_k b_k + b_k^2 \leq c^2 a_k^2 + \frac{1-c^2}{2} a_k^2 + \frac{2c^2}{1-c^2}b_k^2 + b_k^2,
	\end{align*}
	where the last inequality holds because \rev{$2ab \leq t a^2 + \frac{b^2}{t}, \forall a,b \in \mathbb{R},t>0$}. Hence, \eqref{eq:square_recursive_ineq} is proved. We sum both sides of \eqref{eq:square_recursive_ineq} over $k=0,\ldots,T-1$ and obtain
	\begin{equation*}
		\sum_{k=1}^{T} a_k^2 \leq \frac{1+c^2}{2} \sum_{k=0}^{T-1} a_k^2 + \frac{1+c^2}{1-c^2} \sum_{k=0}^{T-1} b_k^2.
	\end{equation*} 
	It follows that
	\begin{align*}
		\frac{1-c^2}{2} \sum_{k=0}^{T-1} a_k^2 \leq a_0^2 - a_T^2 + \frac{1+c^2}{1-c^2} \sum_{k=0}^{T-1} b_k^2 \leq a_0^2 + \frac{1+c^2}{1-c^2} \sum_{k=0}^{T-1} b_k^2.
	\end{align*}
	We multiply both sides of the above inequality by $2/(1-c^2)$ and arrive at \eqref{eq:decreasing_seq_sum_bd}.
\end{proof}

We present a useful inequality relating the arithmetic and quadratic means. It can be proved by the Cauchy-Schwarz inequality or Jensen's inequality.
\begin{lemma}\label{lem:am_qm_ineq}
	Consider \rev{$K \in \mathbb{N}_{+}$} real numbers $X_1, \ldots, X_K \in \mathbb{R}$. Then, $\frac{1}{K} \sum_{i=1}^{K} X_i \leq \sqrt{\frac{\sum_{i=1}^K X_i^2}{K}}$.
\end{lemma}

The following lemma uses the Wasserstein distance between distributions to give an upper bound on the difference between the expectations of a Lipschitz function under these distributions. It is a corollary of the Kantorovich--Rubinstein theorem \citep[Particular case 5.16]{villani2009optimal}.

\begin{lemma}[{\citet[Lemma~D.4]{perdomo2020performative}}]\label{lem:function_closeness}
	Consider a metric space $(\mathbb{R}^m,c)$. Let $\mu,\nu \in \mathcal{P}_1(\mathbb{R}^m)$, and let $\Psi:\mathbb{R}^m \to \mathbb{R}^n$ be $L_\Psi$-Lipschitz, i.e., $\forall x,y \in \mathbb{R}^m, \|\Psi(x) - \Psi(y)\| \leq L_\Psi c(x,y)$. With $W_1(\mu,\nu)$ defined by \eqref{eq:W1_dist}, we have
	\begin{equation}\label{eq:func_closeness}
		\left\|\E_{p\sim \mu}[\Psi(p)] - \E_{p\sim \nu}[\Psi(p)] \right\| \leq L_\Psi W_1(\mu,\nu).
	\end{equation}
\end{lemma}

Next, we present H\"older's inequality for probability measures. It enables us to bound the expectation of the product of random variables.
\begin{lemma}[H\"older's inequality]\label{lem:holder_ineq}
	Consider $K \in \mathbb{N}_{+}$ random variables $X_1,\ldots,X_K \in \mathbb{R}$. Let constants $p_1,\ldots,p_K > 0$ satisfy $\sum_{i=1}^{K} 1/p_i = 1$. Then, $\EXPT{\left|\prod_{i=1}^{K}X_i \right|} \leq \prod_{i=1}^{K} \left(\EXPT{|X_i|^{p_i}}\right)^{\frac{1}{p_i}}$.
\end{lemma}

For a set of random variables endowed with tail bounds, we can quantify the probability that their sums do not exceed a certain threshold, as discussed in the following lemma.
\begin{lemma}\label{lem:high_prob_sum_bd}
	Suppose that $K \in \mathbb{N}_{+}$ random variables $X_i \in \mathbb{R}$ satisfy $\PR{X_i \leq a_i} \geq 1 - \tau_i$, where $i \in \{1,\ldots,K\}$, with constants $a_i \in \mathbb{R}$ and $\tau_i \in (0,1)$ such that $\sum_{i=1}^{K} \tau_i \in (0,1)$. Then, we obtain $\PR{\sum_{i=1}^{K} X_i \leq \sum_{i=1}^{K} a_i} \geq 1-\sum_{i=1}^{K} \tau_i$.
\end{lemma}
\begin{proof}
	Let $A_i$ denote the event that $X_i \leq a_i$. Its complementary event $A_i^c$ is $X_i > a_i$. The condition of the lemma implies that $\forall i \in \{1,\ldots,K\}, \PR{A_i^c} \leq \tau_i$. Then,
	\begin{equation*}
		\PR{\bigcup_{i=1}^{K} A_i^c} \stackrel{\text{(a.1)}}{\leq} \sum_{i=1}^{K} \PR{A_i^c} \leq \sum_{i=1}^{K} \tau_i,
	\end{equation*}
	where (a.1) follows from Boole's inequality. Consequently,
	\begin{align*}
		\PR{\sum_{i=1}^{K} X_i \leq \sum_{i=1}^{K} a_i} \stackrel{\text{(a.1)}}{\geq} \PR{\bigcap_{i=1}^{K} A_i} = 1 - \PR{\bigcup_{i=1}^{K} A_i^c} \geq 1 - \sum_{i=1}^{K} \tau_i,
	\end{align*}
	where (a.1) holds because the event $X_i \leq a_i, \forall i \in \{1,\ldots,K\}$ implies that $\sum_{i=1}^{K} X_i \leq \sum_{i=1}^{K} a_i$, but not vice versa.
\end{proof}

\section{Proof of \texorpdfstring{\cref{lem:mini_batch_grad}}{Lemma~\ref*{lem:mini_batch_grad}}}\label{app:proof_lem_batch_grad}
\begin{proof}
	\revise{The sampling procedure to obtain mini-batches is independent of the $\sigma$-algebra $\calF_{k-1}$. Hence, conditional on $\calF_{k-1}$, the sample pairs $\{(p_k^i, d^{i_k})\}_{i=1}^{\nmb}$ are independent and identically distributed according to $\gamma_k$. Moreover, the decision $u_k$ is measurable with respect to $\calF_{k-1}$.}  
	Therefore, the mini-batch stochastic gradient $\nablabatch{k} \obj(u_k)$ given in \eqref{eq:batch_stoch_grad} satisfies 
	\begin{align}\label{eq:unbiased_batch_grad_detail}
		\EXPT{\nablabatch{k} \obj(u_k) \middle| \calF_{k-1}} &= \frac{1}{\nmb} \sum_{i=1}^{\nmb} \E_{(p_k^i,d^{i_k}) \sim \gamma_k} \left[\nabla_u\Phi(u_k,p_k^i) + \nabla_u h(u_k,d^{i_k}) \nabla_p \Phi(u_k,p_k^i) \big| \calF_{k-1}\right] \notag \\
			&= \widenabla \obj(u_k),
	\end{align}
	where \revise{the expectation is taken over the randomness of sampling}, and the last equality is due to \eqref{eq:cur_distr_stoch_grad}. Furthermore, if \cref{assump:var_grad} is satisfied,
	\begin{align}\label{eq:var_batch_grad_bd}
		\var \left[\nablabatch{k} \obj(u_k) \middle| \calF_{k-1} \right] &\stackrel{\textup{(a.1)}}{=} \frac{1}{\nmb^2} \sum_{i=1}^{\nmb} \var \left[\nabla_u \Phi(u_k,p_k^i) + \nabla_u h(u_k,d^{i_k})\nabla_p \Phi(u_k,p_k^i) \big| \calF_{k-1} \right] \notag \\
			&\stackrel{\textup{(a.2)}}{\leq} \frac{M}{\nmb} + \frac{M_V}{\nmb} \|\widenabla \obj(u_k)\|^2,
	\end{align}
	where (a.1) holds because the sample pairs $\{p_k^i, d^{i_k}\}$ are independent when conditioned on $\calF_{k-1}$, and for any random variable $\xi \in \mathbb{R}^n$ and real number $a \in \mathbb{R}$, $\var[a\xi] = \E\!\left[\|a\xi\|^2\right] - \left\|\E[a\xi]\right\|^2 = a^2 \var[\xi]$; (a.2) follows from \cref{assump:var_grad}. Therefore,
	\begin{equation*}
		\begin{split}
			\E\left[\|\nablabatch{k} \obj(u_k)\|^2 \middle| \calF_{k-1} \right] &= \var \left[\nablabatch{k} \obj(u_k) \middle| \calF_{k-1}\right] + \big\|\E[\nablabatch{k} \obj(u_k) \big| \calF_{k-1}]\big\|^2 \\
			&\leq \frac{M}{\nmb} + \left(\frac{M_V}{\nmb} + 1\right) \|\widenabla \obj(u_k)\|^2, 
		\end{split}
	\end{equation*}
	where the last inequality uses \eqref{eq:unbiased_batch_grad_detail} and \eqref{eq:var_batch_grad_bd}. Consequently, \cref{lem:mini_batch_grad} is proved.
\end{proof}

\section{Proofs for \texorpdfstring{\cref{subsec:distr_shift}}{Section~\ref*{subsec:distr_shift}}}
\subsection{Proof of \texorpdfstring{\cref{lem:wass_dist_bd_recur}}{Lemma~\ref*{lem:wass_dist_bd_recur}}}\label{app:proof_lem_dist_evolve}
\begin{proof}
	The joint distributions $\gamma_k$ and $\gammass(u_k)$ lie in the space $\mathcal{P}_1(\mathbb{R}^{m+r})$, endowed with the Wasserstein distance $W_1$ induced by the metric $c$ in \eqref{eq:metric_mix_norm}. Hence,
	\begin{align}
		W_1(\gamma_k,\gammass(u_k)) &\stackrel{\text{(a.1)}}{=} \inf_{\beta \in \Gamma(\alpha(p_0,d), \mu_d(d'))} \int_{\mathbb{R}^m \times \mathbb{R}^r \times \mathbb{R}^r} c\left((f^{(k)}(p_0,d),d), (h(u_k,d'),d')\right) \ud \beta(p_0,d,d') \notag \\
			&\stackrel{\text{(a.2)}}{\leq} \int_{\mathbb{R}^m \times \mathbb{R}^r \times \mathbb{R}^r} c\left((f^{(k)}(p_0,d),d'), (h(u_k,d),d')\right) \ud \left((\pi_1,\pi_2,\pi_2)_{\#} \tilde{\beta}\right) (p_0,d,d') \notag \\
			&\stackrel{\text{(a.3)}}{=} \int_{\mathbb{R}^m \times \mathbb{R}^r \times \mathbb{R}^r} c\left((f^{(k)}(p_0,d),d), (h(u_k,d),d)\right) \ud \tilde{\beta}(p_0,d,d') \notag \\
			&\stackrel{\text{(a.4)}}{=} \int_{\mathbb{R}^m \times \mathbb{R}^r} c\left((f^{(k)}(p_0,d),d), (h(u_k,d),d)\right) \ud \alpha(p_0,d) \notag \\
			&\stackrel{\text{(a.5)}}{=} \int_{\mathbb{R}^m \times \mathbb{R}^r} \left\|f^{(k)}(p_0,d) - h(u_k,d)\right\|_P \ud \alpha(p_0,d). \label{eq:wass_dist_indiv_bd:mid}
	\end{align}
	In \eqref{eq:wass_dist_indiv_bd:mid}, (a.1) holds because $f^{(k)}$ and $h$ are continuous and hence Borel maps, allowing us to use the property of the pushforward operation; see \citet[Proposition~3]{aolaritei2022uncertainty}. In (a.2), $\tilde{\beta}$ is a specific coupling in $\Gamma(\alpha(p_0,d), \mu_d(d'))$, and $\pi_1$ and $\pi_2$ denote the projections on the first and second variables, respectively, i.e., $\pi_1: (p_0,d,d') \mapsto p_0$ and $\pi_2: (p_0,d,d') \mapsto d$. Consequently, the joint distribution of the first two variables of $\big((\pi_1,\pi_2,\pi_2)_{\#} \tilde{\beta}\big)(p_0,d,d')$, namely $(\pi_1,\pi_2)_{\#} \tilde{\beta}$, is $\alpha$, and the marginal distribution of the last variable is $\mu_d$. Since $\big((\pi_1,\pi_2,\pi_2)_{\#} \tilde{\beta}\big)(p_0,d)$ is also a coupling in $\Gamma(\alpha(p_0,d), \mu_d(d'))$, (a.2) is true. Further, (a.3) invokes the change-of-variables formula for pushforward measures. Moreover, (a.4) follows from the fact that the integrand is independent of $d'$, allowing the integration to be performed with respect to the joint distribution $\alpha$ of $p_0$ and $d$. Finally, (a.5) leverages the metric $c$ in \eqref{eq:metric_mix_norm}, where the distance related to the second component is zero for the same $d$. Therefore, \eqref{eq:wass_dist_indiv_bd} is proved.

	We proceed to establish a recursive inequality for $V_k$. For any $k \in \mathbb{N}_{+}$,
	\begin{align}\label{eq:wass_dist_bd_contract_mid}
		V_k &= \int_{\mathbb{R}^m \times \mathbb{R}^r} \left\|f^{(k)}(p_0,d) - h(u_k,d)\right\|_P \ud \alpha(p_0,d) \notag \\
			&\stackrel{\text{(a.1)}}{=} \int_{\mathbb{R}^m \times \mathbb{R}^r} \left\|f\big(f^{(k-1)}(p_0,d),u_k,d\big) - f(h(u_k,d),u_k,d)\right\|_P \ud \alpha(p_0,d) \notag \\
			&\stackrel{\text{(a.2)}}{\leq} L_f^p \int_{\mathbb{R}^m \times \mathbb{R}^r} \left\|f^{(k-1)}(p_0,d)-h(u_k,d)\right\|_P \ud \alpha(p_0,d) \notag \\
			&\stackrel{\text{(a.3)}}{\leq} L_f^p \underbrace{\int_{\mathbb{R}^m \times \mathbb{R}^r} \left\|f^{(k-1)}(p_0,d) \!-\! h(u_{k\!-\!1},d)\right\|_P \ud \alpha(p_0,d)}_{\numcircled{1}=V_{k-1}} \notag \\
			&\qquad + L_f^p \underbrace{\int_{\mathbb{R}^m \times \mathbb{R}^r} \left\|h(u_{k\!-\!1},d) \!-\! h(u_k,d)\right\|_P \ud \alpha(p_0,d)}_{\numcircled{2}},
	\end{align}
	where (a.1) is due to the distribution dynamics \eqref{eq:pop_dynamics} and the steady state $h(u_k,d)$ satisfying the fixed-point equation $f(h(u_k,d),u_k,d) = h(u_k,d)$; (a.2) leverages the property that $f$ is $L_f^p$-Lipschitz in $p$, see \cref{assump:stable_system}; (a.3) uses the triangle inequality. In \eqref{eq:wass_dist_bd_contract_mid}, \rev{term \numcircled{1} corresponds to $V_{k-1}$}, and term \numcircled{2} can be bounded from above by
	\begin{align*}
		\numcircled{2} &\stackrel{\text{(a.1)}}{\leq} \sqrt{\lambda_{\max}(P)} \int_{\mathbb{R}^m \times \mathbb{R}^r} \|h(u_{k-1},d) - h(u_k,d)\| \ud \alpha(p_0,d) \\
			&\stackrel{\text{(a.2)}}{\leq} L_h^u \sqrt{\lambda_{\max}(P)} \int_{\mathbb{R}^m \times \mathbb{R}^r} \|u_{k-1}-u_k\| \ud \alpha(p_0,d) \\
			&\stackrel{\text{(a.3)}}{\leq} L_h^u \sqrt{\lambda_{\max}(P)} \|u_k - u_{k-1}\|,
	\end{align*}
	where (a.1) uses \cref{lem:weighted_norm} in \cref{app:lemmas}; (a.2) follows from the fact that $h$ is $L_h^u$-Lipschitz in $u$; see the discussion below \cref{assump:stable_system}; (a.3) is due to $\int 1 \ud \alpha(p_0,d) = 1$. We incorporate this upper bound into \eqref{eq:wass_dist_bd_contract_mid} and prove \eqref{eq:wass_dist_bd_contract}.

	Finally, we demonstrate that the initial upper bound $V_0$ on $W_1(\gamma_0,\gammass(u_0))$ is finite. Note that
	\begin{align}\label{eq:wass_dist_bd_init}
		V_0 &\stackrel{\text{(a.1)}}{=} \int_{\mathbb{R}^m \times \mathbb{R}^r} \|p_0 - h(u_0,d)\|_P \ud \alpha(p_0,d) \notag \\
			&\stackrel{\text{(a.2)}}{\leq} \sqrt{\lambda_{\max}(P)} \int_{\mathbb{R}^m \times \mathbb{R}^r} \|p_0 - h(u_0,d)\| \ud \alpha(p_0,d) \notag \\
			&\stackrel{\text{(a.3)}}{\leq} \sqrt{\lambda_{\max}(P)} \bigg(\int_{\mathbb{R}^m \times \mathbb{R}^r} \|p_0\|\ud \alpha(p_0,d) + \int_{\mathbb{R}^m \times \mathbb{R}^r} \|h(u_0,d)\|\ud \alpha(p_0,d) \bigg) \notag \\
			&\stackrel{\text{(a.4)}}{\leq} \sqrt{\lambda_{\max}(P)} \bigg(\underbrace{\int_{\mathbb{R}^m} \|p_0\|\ud \mu_0(p_0)}_{\numcircled{1}} + \underbrace{\int_{\mathbb{R}^r} \|h(u_0,d)\|\ud \mu_d(d)}_{\numcircled{2}} \bigg)
	\end{align}
	where (a.1) follows from the definition $f^{(0)}(p_0,d) = p_0$; (a.2) uses \cref{lem:weighted_norm} in \cref{app:lemmas}; (a.3) is due to the triangle inequality; (a.4) holds because the integrands only depend on the corresponding random variables and not on their joint distribution. In \eqref{eq:wass_dist_bd_init}, term \numcircled{1} is finite thanks to $\mu_0 \in \mathcal{P}_1(\mathbb{R}^m)$, implying that $\mu_0$ admits a finite first moment. We now show that term \numcircled{2} in \eqref{eq:wass_dist_bd_init} is finite. In fact,
	\begin{align*}
		\int_{\mathbb{R}^r} \|h(u_0,d)\| \ud \mu_d(d) &\stackrel{\text{(a.1)}}{\leq} \int_{\mathbb{R}^r} \|h(u_0,d) - h(u_0,0)\| \ud \mu_d(d) + \int_{\mathbb{R}^r} \|h(u_0,0)\| \ud \mu_d(d) \\
			&\stackrel{\text{(a.2)}}{\leq} L_h^d \int_{\mathbb{R}^r} \|d\| \ud \mu_d(d) + \|h(u_0,0)\| < \infty,
	\end{align*}
	where (a.1) invokes the triangle inequality, and (a.2) holds because $h(u_0,d)$ is $L_h^d$-Lipschitz in $d$ and $\int_{\mathbb{R}^r} 1 \ud \mu_d(d) = 1$. We know that $\int_{\mathbb{R}^r} \|d\| \ud \mu_d(d)$ is finite due to $\mu_d \in \mathcal{P}_1(\mathbb{R}^r)$. Since both terms \numcircled{1} and \numcircled{2} in \eqref{eq:wass_dist_bd_init} are finite, the initial bound $V_0$ is also finite.
\end{proof}

\subsection{Proof of \texorpdfstring{\cref{thm:wasserstein_evolution_general}}{Theorem~\ref*{thm:wasserstein_evolution_general}}}\label{app:proof_thm_evolve}
\begin{proof}
	For any $k \in \mathbb{N}$, the Wasserstein distance $W_1(\gamma_k,\gammass(u_k))$ is always nonnegative, because the metric $c$ given in \eqref{eq:metric_mix_norm} will not be negative. Therefore, the corresponding upper bound $V_k$ is also nonnegative. Hence, we start from \eqref{eq:wass_dist_bd_contract} in \cref{lem:wass_dist_bd_recur}, apply \eqref{eq:decreasing_seq_sum_bd} in \cref{lem:seq_square_sum}, and obtain
	\begin{equation*}
		\sum_{k=0}^{T-1} V_k^2 \leq \frac{V_0^2}{1-\rho_1} + \frac{\rho_2}{1-\rho_1} \sum_{k=1}^{T} \|u_k-u_{k-1}\|^2,
	\end{equation*}
	where $\rho_1$ and $\rho_2$ are given in \cref{thm:wasserstein_evolution_general}. The upper bound \eqref{eq:wass_dist_indiv_bd} and the non-negativity of the Wasserstein distance imply that $0 \leq W_1(\gamma_k,\gammass(u_k)) \leq V_k$, which leads to
	\begin{equation}\label{eq:sum_square_wasser_bd}
		\sum_{k=0}^{T-1} W_1(\gamma_k,\gammass(u_k))^2 \leq \sum_{k=0}^{T-1} V_k^2. 
	\end{equation}
	Therefore, \eqref{eq:dist_square_sum_general} holds. The update rule~\eqref{eq:stochastic_alg} indicates that $u_k - u_{k-1} = -\eta\nablabatch{k-1}\obj(u_{k-1})$. We incorporate this expression into \eqref{eq:dist_square_sum_general} and arrive at \eqref{eq:dist_square_sum_gd}.
\end{proof}

\subsection{Proof of \texorpdfstring{\cref{lem:grad_err_bound}}{Lemma~\ref*{lem:grad_err_bound}}}\label{app:proof_lem_grad_err}
\begin{proof}
	Let $u_k \in \mathbb{R}^n$ be given. Consider the following function $\Psi: \mathbb{R}^{m} \times \mathbb{R}^r \to \mathbb{R}^n$ representing the full gradient inside the expectations in \eqref{eq:ss_distr_stoch_grad} and \eqref{eq:cur_distr_stoch_grad}
	\begin{equation*}
		\Psi(p,d) \triangleq \nabla_u \Phi(u_k,p) + \nabla_u h(u_k,d)\nabla_p \Phi(u_k,p).
	\end{equation*}
	We now demonstrate that $\Psi(p,d)$ is Lipschitz continuous in $(p,d)$ with respect to the metric $c$ defined in \eqref{eq:metric_mix_norm}. For any $p_1,p_2 \in \mathbb{R}^m$ and $d_1,d_2 \in \mathbb{R}^r$,
	\begin{align*}
		\|\Psi(p_1,d_1) -& \Psi(p_2,d_2)\| \stackrel{\text{(a.1)}}{\leq} \|\Psi(p_1,d_1) - \Psi(p_1,d_2)\| + \|\Psi(p_1,d_2) - \Psi(p_2,d_2)\| \notag \\
			\stackrel{\text{(a.2)}}{\leq}& \|\nabla_u h(u_k,d_1) - \nabla_u h(u_k,d_2)\|\,\|\nabla_p \Phi(u_k,p_1)\| \notag \\
			&+ \|\nabla_u \Phi(u_k,p_1) - \nabla_u \Phi(u_k,p_2)\| + \|\nabla_u h(u_k,d_2)\|\, \|\nabla_p \Phi(u_k,p_1) - \nabla_p \Phi(u_k,p_2)\| \notag \\
			\stackrel{\text{(a.3)}}{\leq}& L_\Phi^p M_h^d \|d_1 - d_2\| + \left(M_\Phi^u + L_h^u M_\Phi^p\right) \|p_1 - p_2\| \notag \\
			\stackrel{\text{(a.4)}}{\leq}& L(\|p_1-p_2\|_P + \|d_1-d_2\|),
	\end{align*}
	where (a.1) follows from the triangle inequality; (a.2) uses the fact that norms are sub-multiplicative and the triangle inequality; (a.3) invokes the Lipschitz continuity of $\nabla_u h(u_k,d)$, $\nabla_u \Phi(u_k,p)$, $\nabla_p \Phi(u_k,p)$, $\Phi(u_k,p)$, and $h(u_k,d)$ thanks to \cref{assump:stable_system,assump:obj_property}, indicating $\|\nabla_p \Phi(u_k,p_1)\| \leq L_\Phi^p$ and $\|\nabla_u h(u_k,d_2)\| \leq L_h^u$; (a.4) uses \eqref{eq:weight_norm_bd} in \cref{lem:weighted_norm} to relate the Euclidean norm to the weighted norm $\|\cdot\|_P$. Afterward, we leverage \cref{lem:function_closeness} in \cref{app:lemmas} to obtain \eqref{eq:grad_err_bd_wasserstein}.
\end{proof}

\section{Proofs for \texorpdfstring{\cref{subsec:optimality}}{Section~\ref*{subsec:optimality}}}
\subsection{Proof of \texorpdfstring{\cref{thm:optimality_nonconvex}}{Theorem~\ref*{thm:optimality_nonconvex}}}\label{app:proof_thm_optimality}
The overarching idea is to analyze the coupled evolution of the distribution dynamics~\eqref{eq:pop_dynamics} and the algorithm~\eqref{eq:stochastic_alg}. To this end, we first quantify the cumulative Wasserstein metric related to the distribution~\eqref{eq:pop_dynamics}. Afterward, we focus on the iterative update of the algorithm~\eqref{eq:stochastic_alg} and synthesize the overall convergence measure, i.e., the average expected second moment of gradients.

We start with the following lemma that characterizes the cumulative sum of the squared Wasserstein distances between the distribution at each time and the corresponding steady-state distribution. It is built on \cref{lem:wass_dist_bd_recur} and further uses the condition \eqref{eq:step_size_expt} on the step size in \cref{thm:optimality_nonconvex}, thereby exposing the true gradient $\nabla \obj(u_k)$ in the upper bound on this cumulative sum.

\begin{lemma}\label{lem:wasserstein_dist_sum}
	Under the conditions of \cref{thm:optimality_nonconvex}, we have
	\begin{align}\label{eq:wasserstein_dist_sum}
		\sum_{k=0}^{T-1} \EXPT{W_1(\gamma_k,\gammass(u_k))^2} &\leq \frac{1}{6L^2} \sum_{k=0}^{T-1} \EXPT{\|\nabla \obj(u_k)\|^2} + \frac{7 V_0^2}{6(1-\rho_1)} + \frac{M}{12L^2(M_V \!+\! \nmb)}.
	\end{align}
\end{lemma}
\begin{proof}
	We first prove the following intermediate inequality
	\begin{equation}\label{eq:wasserstein_square_sum_pre}
    	\sum_{k=0}^{T-1} \EXPT{W_1(\gamma_k,\gammass(u_k))^2} \leq \frac{V_0^2}{1-\rho_1} + \frac{\eta^2\rho_2 M T}{(1-\rho_1)\nmb} + \frac{\eta^2\rho_2}{1-\rho_1} \cdot \left(\frac{M_V}{\nmb} \!+\! 1\right) \sum_{k=0}^{T-1} \EXPT{\|\widenabla \obj(u_k)\|^2},
	\end{equation}
	where the coefficients $\rho_1$ and $\rho_2$ are specified in \cref{thm:wasserstein_evolution_general}. 
	We know from the update rule~\eqref{eq:stochastic_alg} that $u_k - u_{k-1} = -\eta \nablabatch{k-1} \obj(u_{k-1})$. We plug this expression into \eqref{eq:wass_dist_bd_contract}, invoke \eqref{eq:square_recursive_ineq} in \cref{lem:seq_square_sum}, \cref{app:lemmas} (note that $(V_k)_{k \in \mathbb{N}}$ is a nonnegative sequence; see \cref{app:proof_thm_evolve}), and obtain 
	\begin{equation}\label{eq:squared_wasserstein_recursive}
		V_k^2 \leq \rho_1 V_{k-1}^2 + \eta^2 \rho_2 \|\nablabatch{k-1} \obj(u_{k-1})\|^2.
	\end{equation}
	Recall that $\mathcal{F}_k$ is the $\sigma$-algebra generated by $\nablabatch{0} \obj(u_0), \ldots, \nablabatch{k} \obj(u_k)$. It follows that
	\begin{align*}
		\EXPT{V_k^2 \big| \calF_{k\!-\!2}} &\leq \rho_1 \EXPT{V_{k-1}^2 \big| \calF_{k\!-\!2}} + \eta^2\rho_2 \rev{\EXPT{\|\nablabatch{k-1} \obj(u_{k\!-\!1})\|^2 \big| \calF_{k\!-\!2}}} \\
			&\stackrel{\text{(a.1)}}{\leq} \rho_1 \EXPT{V_{k-1}^2 \big| \calF_{k-2}} + \frac{\eta^2 \rho_2M}{\nmb} + \eta^2 \rho_2 \left(\frac{M_V}{\nmb} + 1\right) \|\widehat{\nabla}^{k-1} \obj(u_{k-1})\|^2,
	\end{align*}
	where (a.1) is due to \eqref{eq:second_mom_batch_grad} in \cref{lem:mini_batch_grad}. We further take the total expectation of both sides of the above inequality, use the tower rule, and arrive at
	\begin{equation}\label{eq:squared_wasserstein_expt_recursive}
			\EXPT{V_k^2} \leq \rho_1 \EXPT{V_{k-1}^2} + \frac{\eta^2 \rho_2 M}{\nmb} + \eta^2 \rho_2 \left(\frac{M_V}{\nmb} + 1\right) \EXPT{\|\widehat{\nabla}^{k-1} \obj(u_{k-1})\|^2}.
	\end{equation}
    \rev{Similar to the proof of \eqref{eq:decreasing_seq_sum_bd} in \cref{lem:seq_square_sum}, \cref{app:lemmas}, we sum both sides of \eqref{eq:squared_wasserstein_expt_recursive} over $k=1,\ldots,T$, use the upper bound $V_0^2 - \EXPT{V_T^2} \leq V_0^2$, and invoke \eqref{eq:sum_square_wasser_bd} to obtain the inequality \eqref{eq:wasserstein_square_sum_pre}.}

	We now focus on the last term of \eqref{eq:wasserstein_square_sum_pre}. Recall from \eqref{eq:ss_distr_stoch_grad} and \eqref{eq:cur_distr_stoch_grad} in \cref{subsec:sg_intuition} that $\nabla \obj(u_k)$ and $\widenabla \obj(u_k)$ are the true gradient and the approximate gradient at $u_k$, respectively. Moreover, $e_k$ is the difference of $\widenabla \obj(u_k)$ and $\nabla \obj(u_k)$; see \eqref{eq:error_grad}. Therefore,
	\begin{align}\label{eq:bd_sec_mom_approx_grad}
		\EXPT{\|\widenabla \obj(u_k)\|^2} &\stackrel{\text{(a.1)}}{\leq} 2\EXPT{\|\nabla \obj(u_k)\|^2} + 2\EXPT{\|e_k\|^2} \notag \\
			&\stackrel{\text{(a.2)}}{\leq} 2\EXPT{\|\nabla \obj(u_k)\|^2} + 2L^2 \EXPT{W_1(\gamma_k,\gammass(u_k))^2},
	\end{align}
	where (a.1) uses the inequality $\|a+b\|^2 \leq 2\|a\|^2 + 2\|b\|^2, \forall a,b \in \mathbb{R}^n$, and (a.2) applies \eqref{eq:grad_err_bd_wasserstein} in \cref{lem:grad_err_bound}. We plug this upper bound into \eqref{eq:wasserstein_square_sum_pre}, rearrange terms, and obtain
	\begin{align}\label{eq:wasserstein_dist_sum_mid}
		\bigg(1 &- \frac{2\eta^2\rho_2L^2}{1-\rho_1} \left(\frac{M_V}{\nmb} + 1\right)\bigg) \sum_{k=0}^{T-1} \EXPT{W_1(\gamma_k,\gammass(u_k))^2} \notag \\
		 	&\leq \frac{2\eta^2\rho_2}{1-\rho_1} \left(\frac{M_V}{\nmb} + 1\right) \sum_{k=0}^{T-1} \EXPT{\|\nabla \obj(u_{k})\|^2} + \frac{V_0^2}{1-\rho_1} + \frac{\eta^2\rho_2 M T}{(1-\rho_1)\nmb}.
	\end{align}
	The parametric condition \eqref{eq:step_size_expt} of \cref{thm:optimality_nonconvex} leads to $\frac{2\eta^2\rho_2L^2}{1-\rho_1} \left(\frac{M_V}{\nmb} + 1\right) \leq \frac{1}{7 T} \leq \frac{1}{7}$. It follows that
	\begin{align*}
		&\frac{1}{1 \!-\! \frac{2\eta^2\rho_2L^2}{1-\rho_1} \left(\frac{M_V}{\nmb} \!+\! 1\right)} \cdot \frac{2\eta^2\rho_2}{1\!-\!\rho_1} \left(\frac{M_V}{\nmb} + 1\right) \leq \rev{\frac{1}{1\!-\!\frac{1}{7}}\cdot\frac{1}{7L^2} = \frac{1}{6L^2}}, \\
		& \frac{1}{1 \!-\! \frac{2\eta^2\rho_2L^2}{1\!-\!\rho_1} \left(\frac{M_V}{\nmb} \!+\! 1\right)} \cdot \frac{\eta^2\rho_2 M T}{(1-\rho_1)\nmb} \leq \rev{\frac{1}{1\!-\!\frac{1}{7}}\cdot\frac{M}{14L^2(M_V \!+\! \nmb)} = \frac{M}{12L^2(M_V \!+\! \nmb)}}.
	\end{align*}
	We divide both sides of \eqref{eq:wasserstein_dist_sum_mid} by $1 - \frac{2\eta^2\rho_2L^2}{1-\rho_1} \left(\frac{M_V}{\nmb} \!+\! 1\right)$, use the above upper bounds on coefficients, and obtain \eqref{eq:wasserstein_dist_sum}.
\end{proof}

With \cref{lem:wasserstein_dist_sum} in hand, we are ready to prove \cref{thm:optimality_nonconvex}.

\begin{proof}
	Since $\obj(u)$ is \rev{$\Lobj$-smooth}, we have
	\begin{align}\label{eq:descent_lemma_intermediate}
		\obj(u_{k+1}) \leq& \obj(u_k) + \nabla\obj(u_k)^{\top}(u_{k+1}-u_k) + \frac{\Lobj}{2} \|u_{k+1} - u_k\|^2 \notag \\
			\stackrel{\text{(a.1)}}{=}& \obj(u_k) - \eta \nabla\obj(u_k)^{\top} \nablabatch{k} \obj(u_k) + \frac{\Lobj\eta^2}{2} \|\nablabatch{k} \obj(u_k)\|^2,
	\end{align}
	where (a.1) follows from the update $u_{k+1} - u_k = -\eta \nablabatch{k} \obj(u_k)$; see \eqref{eq:stochastic_alg}.
	Hence,
	\begin{align}\label{eq:descent_lemma_bd_pre}
		\EXPT{\obj(u_{k+1}) \big| \mathcal{F}_{k-1}} \stackrel{\text{(a.1)}}{\leq}& \obj(u_k) - \eta\nabla\obj(u_k)^\top \widenabla\obj(u_k) + \frac{\Lobj\eta^2}{2} \EXPT{\|\nablabatch{k} \obj(u_k)\|^2 \big| \mathcal{F}_{k-1}} \notag \\
			\stackrel{\text{(a.2)}}{\leq}& \obj(u_k) - \eta \|\nabla \obj(u_k)\|^2 - \eta \nabla \obj(u_k)^{\top}e_k \notag \\
			&+ \frac{\Lobj\eta^2}{2} \left(\frac{M}{\nmb} + \left(\frac{M_V}{\nmb} + 1\right) \|\widenabla \obj(u_k)\|^2\right) \notag \\
			\stackrel{\text{(a.3)}}{\leq}& \obj(u_k) - \left(\frac{\eta}{2} - \Lobj\eta^2\left(\frac{M_V}{\nmb} + 1\right) \right) \|\nabla \obj(u_k)\|^2 \notag \\
			&+ \left(\frac{\eta}{2} + \Lobj\eta^2\left(\frac{M_V}{\nmb} + 1\right) \right)\|e_k\|^2 + \frac{\Lobj M\eta^2}{2\nmb},
	\end{align}
	where (a.1) uses $\E[\nablabatch{k} \obj(u_k) \big| \mathcal{F}_{k-1}] = \widenabla \obj(u_k)$, see \eqref{eq:unbiased_batch_grad}; (a.2) uses the equality $\widenabla \obj(u_k) = \nabla\obj(u_k) + e_k$ (see \eqref{eq:error_grad}) and the upper bound \eqref{eq:second_mom_batch_grad}; (a.3) follows from the inequalities
	\begin{equation*}
		-\nabla\obj(u_k)^\top e_k \leq \frac{1}{2}\|\nabla\obj(u_k)\|^2 + \frac{1}{2}\|e_k\|^2, \qquad \|\widenabla \obj(u_k)\|^2 \leq 2\|\nabla \obj(u_k)\|^2 + 2\|e_k\|^2.
	\end{equation*}
	We take expectation of both sides of \eqref{eq:descent_lemma_bd_pre}, use the tower rule, and obtain
	\begin{align*}
		\EXPT{\obj(u_{k+1})} \leq& \EXPT{\obj(u_k)} - \left(\frac{\eta}{2} - \Lobj\eta^2\left(\frac{M_V}{\nmb} + 1\right) \right) \EXPT{\|\nabla \obj(u_k)\|^2} \\
			&+ \left(\frac{\eta}{2} + \Lobj\eta^2\left(\frac{M_V}{\nmb} + 1\right) \right) \EXPT{\|e_k\|^2} + \frac{\Lobj M\eta^2}{2\nmb}.
	\end{align*}
	We sum up both sides of the above inequality for $k=0,\ldots,T-1$, reorganize terms, and obtain
	\begin{equation*}
		\begin{split}
			\bigg(\frac{\eta}{2} &- \Lobj\eta^2\left(\frac{M_V}{\nmb} + 1\right) \bigg) \sum_{k=0}^{T-1} \EXPT{\|\nabla \obj(u_k)\|^2} \\
			&\leq \obj(u_0) - \EXPT{\obj(u_T)} + \left(\frac{\eta}{2} + \Lobj\eta^2\left(\frac{M_V}{\nmb} + 1\right) \right) \sum_{k=0}^{T-1} \EXPT{\|e_k\|^2} + \frac{\Lobj M\eta^2 T}{2\nmb}.
		\end{split}
	\end{equation*}
	The parametric condition \eqref{eq:step_size_expt} of \cref{thm:optimality_nonconvex} ensures that $\frac{\eta}{2} - \Lobj\eta^2\left(\frac{M_V}{\nmb} + 1\right) \geq \frac{\eta}{4}$ and $\frac{\eta}{2} + \Lobj\eta^2\left(\frac{M_V}{\nmb} + 1\right) \leq \frac{3}{4}\eta$.
	We leverage \eqref{eq:grad_err_bd_wasserstein} and obtain
	\begin{equation*}
		\sum_{k=0}^{T-1} \EXPT{\|e_k\|^2} \leq L^2 \sum_{k=0}^{T-1} \EXPT{W_1(\gamma_k,\gammass(u_k))^2}.
	\end{equation*}
	Furthermore, we invoke \eqref{eq:wasserstein_dist_sum} in \cref{lem:wasserstein_dist_sum}. Hence, we establish the following upper bound
	\begin{align*}
		\frac{\eta}{8} \sum_{k=0}^{T-1} \EXPT{\|\nabla \obj(u_k)\|^2} \leq \obj(u_0) - \EXPT{\obj(u_T)} + \frac{7\eta L^2 V_0^2}{8(1-\rho_1)} + \frac{\eta M}{16(M_V+\nmb)} + \frac{\Lobj M\eta^2 T}{2\nmb}.
	\end{align*}
	We multiply both sides of the above inequality by $8/(\eta T)$ and arrive at
	\begin{align*}
		\frac{1}{T} \sum_{k=0}^{T-1} \EXPT{\|\nabla \obj(u_k)\|^2} \leq& \underbrace{\frac{8(\obj(u_0) \!-\! \obj^*)}{\eta T}}_{\sim \bigO{1/\sqrt{T}}} + \underbrace{\frac{4\Lobj M\eta}{\nmb}}_{\sim \bigO{1/\sqrt{T}}} + \underbrace{\frac{7L^2 V_0^2}{(1-\rho_1)T}}_{\sim \bigO{1/T}} + \underbrace{\frac{M}{2T (M_V \!+\! \nmb)}}_{\sim \bigO{1/T}},
	\end{align*}
	where we additionally use the fact that $\mathbb{E}[\obj(u_T)] \geq \obj^*$, because $\obj^*$ is the optimal value of problem~\eqref{eq:dd_opt_problem}. Therefore, \eqref{eq:average_second_moment_bd} and \eqref{eq:average_second_moment_complexity} are proved.
\end{proof}

\subsection{Proof of \texorpdfstring{\cref{cor:approximate_optimality}}{Corollary~\ref*{cor:approximate_optimality}}}\label{app:proof_cor_approx_opt}
\begin{proof}
	We know from \eqref{eq:bd_sec_mom_approx_grad} in \cref{app:proof_thm_optimality} that
	\begin{equation*}
		\frac{1}{T} \sum_{k=0}^{T-1} \EXPT{\|\widenabla \obj(u_k)\|^2} \leq \frac{2}{T} \sum_{k=0}^{T-1} \EXPT{\|\nabla \obj(u_k)\|^2} + \frac{2L^2}{T} \sum_{k=0}^{T-1} \EXPT{W_1(\gamma_k,\gammass(u_k))^2}.
	\end{equation*}
	Furthermore, we leverage \cref{lem:wasserstein_dist_sum} in \cref{app:proof_thm_optimality} to obtain
	\begin{align*}
		\frac{1}{T} \sum_{k=0}^{T-1} \EXPT{\|\widenabla \obj(u_k)\|^2} \leq \frac{7}{3T} \sum_{k=0}^{T-1} \EXPT{\|\nabla \obj(u_k)\|^2} \!+\! \underbrace{\frac{1}{T} \left(\frac{7L^2 V_0^2}{3(1\!-\!\rho_1)} \!+\! \frac{M}{6(M_V \!+\! \nmb)} \right)}_{\sim \bigO{1/T}}.
	\end{align*}
	By invoking \eqref{eq:average_second_moment_bd} and \eqref{eq:average_second_moment_complexity}, we prove \eqref{eq:average_approx_sec_mom_complexity}.
\end{proof}

\subsection{Proof of \texorpdfstring{\cref{thm:distribution_expt}}{Theorem~\ref*{thm:distribution_expt}}}\label{app:distribution_expt}
\begin{proof}
	We first derive the following upper bound
	\begin{align*}
		\frac{1}{T} \sum_{k=0}^{T-1} \EXPT{W_1(\gamma_k,\gammass(u_k))^2} &\stackrel{\textup{(a.1)}}{\leq} \frac{1}{6L^2 T} \sum_{k=0}^{T-1} \EXPT{\|\nabla \obj(u_k)\|^2} \!+\! \underbrace{\frac{1}{T} \left(\frac{7 V_0^2}{6(1\!-\!\rho_1)} \!+\! \frac{M}{12L^2(M_V \!+\! \nmb)}\right)}_{\sim \bigO{1/T}} \\
		&\stackrel{\textup{(a.2)}}{\leq} \underbrace{\frac{1}{6L^2} \left(\frac{8(\obj(u_0) - \obj^*)}{\eta T} + \frac{4\Lobj M\eta}{\nmb}\right)}_{\sim \bigO{1/\sqrt{T}}} + \bigO{\frac{1}{T}},
	\end{align*}
	where (a.1) follows from \cref{lem:wasserstein_dist_sum} in \cref{app:proof_thm_optimality}, and (a.2) invokes the upper bound \eqref{eq:average_second_moment_bd} in \cref{thm:optimality_nonconvex}. Furthermore, \cref{lem:am_qm_ineq} in \cref{app:lemmas} implies that
	\begin{equation*}
		\frac{1}{T} \sum_{k=0}^{T-1} \EXPT{W_1(\gamma_k, \gammass(u_k))} \leq \sqrt{\frac{1}{T} \sum_{k=0}^{T-1} \EXPT{W_1(\gamma_k,\gammass(u_k))^2}}.
	\end{equation*}
	 We combine the above inequalities and prove \cref{thm:distribution_expt}.
\end{proof}

\section{Proofs for \texorpdfstring{\cref{subsec:optimality_hp}}{Section~\ref*{subsec:optimality_hp}}}

\subsection{Proof of \texorpdfstring{\cref{thm:optimality_nonconvex_high_prob}}{Theorem~\ref*{thm:optimality_nonconvex_high_prob}}}\label{app:proof_thm_optimality_high_prob}
The proof resembles what we have presented in \cref{app:proof_thm_optimality}. The key idea is to quantify the coupled evolution of optimization iterations and dynamic distributions. However, to establish convergence with high probability rather than in expectation, we will derive and exploit some tail bounds related to stochastic errors of gradients. 

Throughout this subsection, we rely on the following decomposition
\begin{equation}\label{eq:batch_grad_err_sum}
	\nablabatch{k} \obj(u_k) = \widenabla \obj(u_k) + \xi_k = \nabla \obj(u_k) + e_k + \xi_k,
\end{equation}
where $e_k$ and $\xi_k$ are given in \eqref{eq:error_grad} and \eqref{eq:batch_grad_stoch_err}, respectively. Specifically, $e_k$ denotes the difference of expected gradients (i.e., $\widenabla \obj(u_k)$ and $\nabla \obj(u_k)$) arising from the discrepancy between $\gamma_k$ and $\gammass(u_k)$, and $\xi_k$ indicates the stochastic noise in the mini-batch gradient $\nablabatch{k} \obj(u_k)$ relative to its expectation $\widenabla \obj(u_k)$.

Along with the iteration~\eqref{eq:stochastic_alg}, the stochastic noise $\xi_k$ \eqref{eq:batch_grad_stoch_err} exerts a cumulative influence on the accuracy of solutions. In the following lemma, we provide high-probability bounds on some cumulative quantities that involve $\xi_k$ and are relevant to the overall convergence measure.
\begin{lemma}\label{lem:err_inner_prod_square}
	Let \cref{assump:sub_gaussian_noise} hold. For any $\tau_1,\tau_2 \in (0,1)$ and $\lambda > 0$, the iteration \eqref{eq:stochastic_alg} ensures
	\begin{subequations}
		\begin{align}
			\sum_{k=0}^{T-1} -\nabla \obj(u_k)^\top \xi_k &\leq \frac{3}{4} \lambda \sigmamb^2 \sum_{k=0}^{T-1} \|\nabla \obj(u_k)\|^2 + \frac{1}{\lambda} \ln\frac{1}{\tau_1}, \label{eq:cross_term_hp} \\
			\sum_{k=0}^{T-1} \|\xi_k\|^2 &\leq T \sigmamb^2\left(1 + \ln\frac{1}{\tau_2}\right), \label{eq:err_sec_mom_hp}
		\end{align}
	with probabilities at least $1-\tau_1$ and $1-\tau_2$, respectively.
	\end{subequations}
\end{lemma}
\begin{proof}
	Recall that $\calF_k$ is the $\sigma$-algebra generated by $\nablabatch{0} \obj(u_0),\ldots,\nablabatch{k}\obj(u_k)$. Hence, as a function of $\nablabatch{0} \obj(u_0),\ldots,\nablabatch{k-1} \obj(u_{k-1})$ determined by the update \eqref{eq:stochastic_alg}, $u_k$ is measurable with respect to $\calF_{k-1}$. Moreover, when conditioned on $\calF_{k-1}$, $\nablabatch{k} \obj(u_k)$ is an unbiased estimate of $\widenabla \obj(u_k)$; see \eqref{eq:unbiased_batch_grad}. Therefore,
	\begin{equation*}
		\EXPT{-\nabla \obj(u_k)^\top \xi_k \big|\calF_{k-1}}  = -\nabla\obj(u_k)^\top \EXPT{\nablabatch{k} \obj(u_k) - \widenabla \obj(u_k) \big|\calF_{k-1}} = 0.
	\end{equation*}
	Hence, $\big(-\nabla \obj(u_k)^\top \xi_k\big)_{k\in \mathbb{N}}$ is a martingale difference sequence. Whenever $\nabla\obj(u_k)=0$, we define $\frac{(-\nabla\obj(u_k)^\top\xi_k)^2}{(\sigmamb\|\nabla\obj(u_k)\|)^2} \triangleq 0$. Furthermore, $\sigmamb\|\nabla\obj(u_k)\|$ is measurable with respect to $\calF_{k-1}$, and 
	\begin{align*}
		&\EXPT{\exp\left(\frac{(-\nabla\obj(u_k)^\top \xi_k)^2}{(\sigmamb\|\nabla\obj(u_k)\|)^2} \right) \bigg| \calF_{k-1}} \stackrel{\text{(a.1)}}{\leq} \EXPT{\exp\left(\frac{\|\nabla\obj(u_k)\|^2 \|\xi_k\|^2}{(\sigmamb\|\nabla\obj(u_k)\|)^2} \right) \bigg| \calF_{k-1}} \\
			&\quad= \EXPT{\exp\left(\frac{\|\xi_k\|^2}{\sigmamb^2}\right) \bigg| \calF_{k-1}} \stackrel{\text{(a.2)}}{\leq} \exp(1),
	\end{align*}
	where (a.1) uses the Cauchy-Schwarz inequality, and (a.2) is due to \eqref{eq:subgaussian_exp}. 
	Then, we apply the martingale concentration inequality \citep[see][Lemma~1]{li2020high} and obtain \eqref{eq:cross_term_hp}.

	We proceed to prove \eqref{eq:err_sec_mom_hp}. We know from \eqref{eq:subgaussian_exp} and the tower rule that
	\begin{equation}\label{eq:subgaussian_exp_unconditional}
		\EXPT{\exp\left(\frac{\|\xi_k\|^2}{\sigmamb^2}\right)} = \EXPT{\EXPT{\exp\left(\frac{\|\xi_k\|^2}{\sigmamb^2}\right) \bigg| \calF_{k-1}}} \leq \exp(1), \quad \forall k \in \mathbb{N}.
	\end{equation}
	Let $\sigmambbar \triangleq \sqrt{T}\sigmamb > 0$. We note that
	\begin{align}\label{eq:exp_sum_err_square}
		&\EXPT{\exp\left(\frac{\sum_{k=0}^{T-1} \|\xi_k\|^2}{\sigmambbar^2}\right)} = \EXPT{\prod_{k=0}^{T-1} \exp\left(\frac{\|\xi_k\|^2}{\sigmambbar^2}\right)} \notag \\
			&\quad\stackrel{\text{(a.1)}}{\leq} \prod_{k=0}^{T-1} \EXPT{\exp\left(\frac{\|\xi_k\|^2}{\sigmambbar^2}\right)^T}^\frac{1}{T} = \prod_{k=0}^{T-1} \EXPT{\exp\left(\frac{T\|\xi_k\|^2}{T\sigmamb^2}\right)}^\frac{1}{T} \notag \\ 
			&\quad \stackrel{\text{(a.2)}}{\leq} \prod_{k=0}^{T-1} [\exp(1)]^\frac{1}{T} = \exp(1),
	\end{align}
	where (a.1) follows from H\"older's inequality, see \cref{lem:holder_ineq} in \cref{app:lemmas}, and (a.2) uses \eqref{eq:subgaussian_exp_unconditional}. 
	Therefore, for any constant $a>0$,
	\begin{align*}
		\PR{\sum_{k=0}^{T-1} \|\xi_k\|^2 > a} &= \PR{\exp\left(\frac{\sum_{k=0}^{T-1} \|\xi_k\|^2}{\sigmambbar^2}\right) > \exp\left(\frac{a}{\sigmambbar^2}\right)} \\
		&\stackrel{\text{(a.1)}}{\leq} \frac{\EXPT{\exp\left(\sum_{k=0}^{T-1} \|\xi_k\|^2/\sigmambbar^2\right)}}{\exp(a/\sigmambbar^2)} \stackrel{\text{(a.2)}}{\leq} \frac{\exp(1)}{\exp(a/\sigmambbar^2)},
	\end{align*}
	where (a.1) holds because of Markov's inequality, and (a.2) uses \eqref{eq:exp_sum_err_square}. For any fixed $\tau_2 \in (0,1)$, we set $a = T \sigmamb^2 \big(1+\ln \frac{1}{\tau_2}\big)$. It follows that \eqref{eq:err_sec_mom_hp} holds.
\end{proof}

Similar to \cref{lem:wasserstein_dist_sum}, the following lemma provides an upper bound on the cumulative squared Wasserstein distances. This cumulative sum quantifies the evolution of the distribution driven by the decision-maker. Due to the coupling between the distribution dynamics and the decision, the upper bound involves quantities related to gradients, which originate from optimization iterations.
\begin{lemma}\label{lem:wasserstein_dist_sum_hp}
	Under the conditions of \cref{thm:optimality_nonconvex_high_prob}, we have
	\begin{equation}\label{eq:wasserstein_dist_sum_hp}
		\sum_{k=0}^{T-1} W_1(\gamma_k,\gammass(u_k))^2 \leq \frac{1}{6L^2} \sum_{k=0}^{T-1} \|\nabla \obj(u_k)\|^2 + \frac{7 V_0^2}{6(1-\rho_1)} + \frac{1}{6L^2 T} \sum_{k=0}^{T-1} \|\xi_k\|^2.
	\end{equation}
\end{lemma}
\begin{proof}
	We rely on \eqref{eq:dist_square_sum_gd} in \cref{thm:wasserstein_evolution_general}, \cref{subsec:distr_shift} and proceed by deriving an upper bound on the cumulative squared norm of $\nablabatch{k} \obj(u_k)$. 
	Based on \eqref{eq:batch_grad_err_sum}, we know
	\begin{align}\label{eq:decompose_batch_grad_squared}
		\|\nablabatch{k} \obj(u_k)\|^2 &= \|\nabla\obj(u_k) + e_k + \xi_k\|^2 \notag \\
			&\stackrel{\text{(a.1)}}{\leq} 3\|\nabla\obj(u_k)\|^2 + 3\|e_k\|^2 + 3\|\xi_k\|^2 \notag \\
			&\stackrel{\text{(a.2)}}{\leq} 3\|\nabla\obj(u_k)\|^2 + 3\|\xi_k\|^2 + 3L^2 W_1(\gamma_k,\gammass(u_k))^2, 
	\end{align}
	where (a.1) uses the inequality $\forall a,b,c \in \mathbb{R}^n, \|a+b+c\|^2 \leq 3(\|a\|^2+\|b\|^2+\|c\|^2)$, which is a vector-form extension of \cref{lem:am_qm_ineq} in \cref{app:lemmas}, and (a.2) follows from \eqref{eq:grad_err_bd_wasserstein}. We plug \eqref{eq:decompose_batch_grad_squared} into the right-hand side of \eqref{eq:dist_square_sum_gd}, rearrange terms, and obtain
	\begin{align}\label{eq:wasserstein_dist_sum_mid_hp}
		\bigg(1 \!-\!\frac{3\eta^2\rho_2L^2}{1-\rho_1} \bigg)
		 \sum_{k=0}^{T-1} W_1(\gamma_k,\gammass(u_k))^2 \leq \frac{V_0^2}{1-\rho_1} + \frac{3\eta^2\rho_2}{1-\rho_1}\sum_{k=0}^{T-1} \|\nabla \obj(u_k)\|^2 + \frac{3\eta^2\rho_2}{1-\rho_1}\sum_{k=0}^{T-1}\|\xi_k\|^2.
	\end{align}
	The parametric condition \eqref{eq:step_size_hp} of \cref{thm:optimality_nonconvex_high_prob} ensures that $\frac{3\eta^2\rho_2L^2}{1-\rho_1} \leq \frac{1}{7T} \leq \frac{1}{7}$, which leads to
	\begin{align*}
		\frac{1}{1 - \frac{3\eta^2\rho_2L^2}{1-\rho_1}} \cdot \frac{3\eta^2\rho_2}{1-\rho_1} \leq \rev{\frac{1}{1-\frac{1}{7}} \cdot \frac{1}{7L^2 T} = \frac{1}{6L^2 T}} \leq \frac{1}{6L^2}.
	\end{align*}
	We divide both sides of \eqref{eq:wasserstein_dist_sum_mid_hp} by $1 - \frac{3\eta^2\rho_2L^2}{1-\rho_1}$, use the above bounds on coefficients, and obtain \eqref{eq:wasserstein_dist_sum_hp}.
\end{proof}

We give the complete proof of \cref{thm:optimality_nonconvex_high_prob}. The key idea is to synthesize the convergence measure (i.e., the average squared gradient norm) from optimization iterations and the characterization of distribution dynamics in \cref{lem:wasserstein_dist_sum_hp}. We also leverage the high-probability bounds in \cref{lem:err_inner_prod_square}.
\begin{proof}
	We incorporate \eqref{eq:batch_grad_err_sum} into \eqref{eq:descent_lemma_intermediate} and have
	\begin{align*}
		\obj(u_{k+1}) \leq& \obj(u_k) - \eta \nabla\obj(u_k)^\top (\nabla \obj(u_k) + e_k + \xi_k) + \frac{\Lobj\eta^2}{2} \|\nabla\obj(u_k) + e_k + \xi_k\|^2 \\
		\stackrel{\text{(a.1)}}{\leq}& \obj(u_k) - \left(\frac{\eta}{2} - \frac{3}{2} \Lobj\eta^2\right) \|\nabla \obj(u_k)\|^2 - \eta \nabla\obj(u_k)^\top \xi_k + \frac{3}{2} \Lobj\eta^2 \|\xi_k\|^2 \\
		&+ \left(\frac{\eta}{2} + \frac{3}{2} \Lobj\eta^2\right) L^2 W_1(\gamma_k,\gammass(u_k))^2,
	\end{align*}
	where (a.1) follows from the inequality $-\nabla\obj(u_k)^\top e_k \leq \frac{1}{2}\|\nabla\obj(u_k)\|^2 + \frac{1}{2} \|e_k\|^2$, \eqref{eq:decompose_batch_grad_squared}, and \eqref{eq:grad_err_bd_wasserstein}. We sum up both sides of the above inequality for $k=0,\ldots,T-1$, reorganize terms, and obtain
	\begin{align*}
		\left(\frac{\eta}{2} - \frac{3}{2} \Lobj\eta^2\right)\sum_{k=0}^{T-1} \|\nabla\obj(u_k)\|^2 \leq& \obj(u_0) - \obj(u_T) -\eta\sum_{k=0}^{T-1} \nabla\obj(u_k)^\top\xi_k + \frac{3}{2} \Lobj\eta^2 \sum_{k=0}^{T-1} \|\xi_k\|^2 \\
			&+ \left(\frac{\eta}{2} + \frac{3}{2} \Lobj\eta^2\right)L^2 \sum_{k=0}^{T-1} W_1(\gamma_k,\gammass(u_k))^2.
	\end{align*}
	The parametric condition \eqref{eq:step_size_hp} of \cref{thm:optimality_nonconvex_high_prob} guarantees that $\frac{\eta}{2} - \frac{3}{2} \Lobj\eta^2 \geq \frac{\eta}{4}$ and $\frac{\eta}{2} + \frac{3}{2} \Lobj\eta^2 \leq \frac{3}{4}\eta$. 
	Furthermore, we use the fact that $\obj(u_T)$ is not less than the optimal value $\obj^*$ of problem~\eqref{eq:dd_opt_problem}, incorporate the upper bound \eqref{eq:wasserstein_dist_sum_hp}, rearrange terms, and arrive at
	\begin{align}\label{eq:second_moment_sum_hp_pre}
		\frac{\eta}{8} \sum_{k=0}^{T-1} \|\nabla\obj(u_k)\|^2 \leq \obj(u_0) \!-\! \obj^* \underbrace{-\eta\sum_{k=0}^{T-1} \nabla\obj(u_k)^\top\xi_k}_{\numcircled{1}} + \underbrace{\left(\frac{3}{2} \Lobj\eta^2 \!+\! \frac{\eta}{8T}\right) \sum_{k=0}^{T-1} \|\xi_k\|^2}_{\numcircled{2}} + \frac{7}{8} \eta L^2 \frac{V_0^2}{1-\rho_1}.
	\end{align}
	For terms \numcircled{1} and \numcircled{2} in \eqref{eq:second_moment_sum_hp_pre}, we exploit the high-probability bounds in \cref{lem:err_inner_prod_square}. Specifically, we select $\tau_1 = \tau_2 = \tau/2$ and apply both \cref{lem:high_prob_sum_bd,lem:err_inner_prod_square}. Then, with probability at least $1-\tau$,
	\begin{align}\label{eq:sum_cross_squared_hp}
		\numcircled{1} + \numcircled{2} \leq \frac{3}{4}\eta\lambda\sigmamb^2 \sum_{k=0}^{T-1} \|\nabla\obj(u_k)\|^2 + \frac{\eta}{\lambda}\ln \frac{2}{\tau} + \left(\frac{3}{2}\Lobj \eta^2 + \frac{\eta}{8T}\right) T\sigmamb^2\left(1 + \ln\frac{2}{\tau}\right).
	\end{align}
	Since $\lambda$ can be any positive constant, we select $\lambda = \frac{1}{12\sigmamb^2}$. We then plug \eqref{eq:sum_cross_squared_hp} into \eqref{eq:second_moment_sum_hp_pre} and know that with probability at least $1-\tau$,
	\begin{align*}
		\frac{\eta}{16} \sum_{k=0}^{T-1} \|\nabla \obj(u_k)\|^2 \leq& \obj(u_0) \!-\! \obj^* \!+\! \frac{7}{8} \eta L^2 \frac{V_0^2}{1\!-\!\rho_1} \!+\! 12\eta \sigmamb^2 \ln \frac{2}{\tau} \!+\! \left(\frac{3}{2}\Lobj\eta^2 \!+\! \frac{\eta}{8T}\right) T\sigmamb^2\left(1 \!+\! \ln\frac{2}{\tau}\right).
	\end{align*}
	We multiply both sides of the above inequality by $16/(\eta T)$ and arrive at
	\begin{align*}
		\frac{1}{T} \sum_{k=0}^{T-1} \|\nabla \obj(u_k)\|^2 \leq& \underbrace{\frac{16(\obj(u_0) - \obj^*)}{\eta T}}_{\sim \bigO{1/\sqrt{T}}} + \underbrace{\left(24\eta \Lobj + \frac{2}{T}\right)\sigmamb^2}_{\sim \bigO{1/\sqrt{T}}} + \underbrace{24\eta \Lobj \sigmamb^2 \ln \frac{2}{\tau}}_{\sim \bigO{\ln(1/\tau)/\sqrt{T}}} \\
		 &+ \underbrace{\frac{194}{T} \sigmamb^2 \ln \frac{2}{\tau}}_{\sim \bigO{\ln(1/\tau)/T}} + \underbrace{\frac{14L^2 V_0^2}{(1-\rho_1)T}}_{\sim \bigO{1/T}},
	\end{align*}
	which holds with probability at least $1-\tau$. Hence, \cref{thm:optimality_nonconvex_high_prob} is proved.
\end{proof}

\subsection{Proof of \texorpdfstring{\cref{thm:distribution_high_prob}}{Theorem~\ref*{thm:distribution_high_prob}}}\label{app:distribution_high_prob}
\begin{proof}
	Starting from \cref{lem:wasserstein_dist_sum_hp} in \cref{app:proof_thm_optimality_high_prob}, we know that
	\begin{align}\label{eq:wasserstein_dist_avg_hp_mid}
		\frac{1}{T} \sum_{k=0}^{T-1} W_1(\gamma_k,\gammass(u_k))^2 \leq \frac{1}{6L^2} \cdot \frac{1}{T} \sum_{k=0}^{T-1} \|\nabla \obj(u_k)\|^2 + \frac{1}{6L^2 T^2} \sum_{k=0}^{T-1} \|\xi_k\|^2 + \bigO{\frac{1}{T}}.
	\end{align}
	\cref{thm:optimality_nonconvex_high_prob} in \cref{subsec:optimality_hp} indicates that with probability at least $1 - \frac{\tau}{2}$,
	\begin{equation}\label{eq:average_second_moment_high_prob_mdf}
		\frac{1}{T} \sum_{k=0}^{T-1} \|\nabla \obj(u_k)\|^2 \leq \frac{16(\obj(u_0) - \obj^*)}{\eta T} + 24\eta \Lobj\sigmamb^2 \left(1+\ln \frac{4}{\tau}\right) + \bigO{\frac{1}{T}\left(1+\ln \frac{1}{\tau}\right)}.
	\end{equation}
	Moreover, \cref{lem:err_inner_prod_square} in \cref{app:proof_thm_optimality_high_prob} implies that with probability at least $1 - \frac{\tau}{2}$,
	\begin{equation}\label{eq:err_sec_mom_hp_mdf}
		\sum_{k=0}^{T-1} \|\xi_k\|^2 \leq T \sigmamb^2\left(1 + \ln\frac{2}{\tau}\right).
	\end{equation}
	By incorporating the high-probability bounds \eqref{eq:average_second_moment_high_prob_mdf} and \eqref{eq:err_sec_mom_hp_mdf} into \eqref{eq:wasserstein_dist_avg_hp_mid} and using \cref{lem:high_prob_sum_bd} in \cref{app:lemmas}, we know that with probability at least $1-\tau$,
	\begin{align*}
		\frac{1}{T} \sum_{k=0}^{T-1} W_1(\gamma_k,\gammass(u_k))^2 \leq \underbrace{\frac{1}{6L^2} \left(\frac{16(\obj(u_0) \!-\! \obj^*)}{\eta T} \!+\! 24\eta \Lobj \sigmamb^2 \left(1\!+\!\ln \frac{4}{\tau}\right) \right)}_{\sim \bigO{(1+\ln(1/\tau))/\sqrt{T}}} \!+\! \bigO{\frac{1}{T}\left(1\!+\!\ln \frac{1}{\tau}\right)}.
	\end{align*}
	To link with the non-squared Wasserstein distance, we apply \cref{lem:am_qm_ineq} in \cref{app:lemmas} and obtain
	\begin{equation*}
		\frac{1}{T} \sum_{k=0}^{T-1} W_1(\gamma_k, \gammass(u_k)) \leq \sqrt{\frac{1}{T} \sum_{k=0}^{T-1} W_1(\gamma_k,\gammass(u_k))^2}.
	\end{equation*}
	Consequently, \cref{thm:distribution_high_prob} is proved. 
\end{proof}

\section{Proofs for \texorpdfstring{\cref{subsec:generalization}}{Section~\ref*{subsec:generalization}}}
\subsection{Proof of \texorpdfstring{\cref{lem:general_err}}{Lemma~\ref*{lem:general_err}}}\label{app:proof_lem_general}
We rely on the concentration of the empirical measure evaluated by Wasserstein distances \citep{fournier2015rate}. Such a concentration is concerned with the closeness between $\mu_d$ and $\mu_d^N$ with high probability. Since $\mu_d,\mu_d^N \in \mathcal{P}(\mathbb{R}^r)$, we define $W_1(\mu_d,\mu_d^N)$ as
\begin{equation}\label{eq:W1_dist_disturbance}
	W_1(\mu_d,\mu_d^N) = \inf_{\bar{\beta} \in \Gamma(\mu_d,\mu_d^N)} \int_{\mathbb{R}^r \times \mathbb{R}^r} \|d - d'\| \ud \bar{\beta}(d,d').
\end{equation}
The measure concentration argument is delineated in the following lemma.


\begin{lemma}[from {\citet[Theorem~2]{fournier2015rate}}]\label{lem:concentr_measure}
	Let the conditions of \cref{lem:general_err} hold.
	Then, for any $N \geq 1$,
	\begin{equation*}
		\PR{W_1(\mu_d, \mu_d^N) \leq \epsilon} \geq
		\begin{cases}
			1 - c_1 e^{-c_2 N \epsilon^r} & \text{if } \epsilon \in (0, 1], \\
			1 - c_1 e^{-c_2 N \epsilon^\theta} & \text{if } \epsilon \in (1,\infty),
		\end{cases}
	\end{equation*}
	where $c_1$ and $c_2$ are positive constants that only depend on $r,\theta,\kappa$, and $\mathcal{E}_{\theta,\kappa}(\mu_d)$.
\end{lemma}

The proof of \cref{lem:general_err} is as follows.
\begin{proof}
	The empirical objective $\obj^N(u)$ can be equivalently written as $\obj^N(u) = \E_{p \sim \muss^N(u)}[\Phi(u,p)]$. Accordingly, the gradient $\nabla \obj^N(u)$ involves the empirical distribution $\gammass^N(u)$. For any $u \in \mathbb{R}^n$, in terms of the difference between the gradients given the distribution $\gammass(u)$ and the empirical distribution $\gammass^N(u) \triangleq (h(u,\cdot),\id)_{\#} \mu_d^N$, we have
	\begin{equation}\label{eq:diff_grad_pop_empr}
		\left\|\nabla \obj(u) - \nabla \obj^N(u)\right\| \stackrel{\text{(a.1)}}{\leq} L W_1(\gammass(u),\gammass^N(u)) \stackrel{\text{(a.2)}}{\leq} L \left(L_h^d\sqrt{\lambda_{\max}(P)}+1\right) W_1(\mu_d, \mu^N_d),
	\end{equation}
	where (a.1) is similar to \eqref{eq:grad_err_bd_wasserstein} in \cref{lem:grad_err_bound}, \cref{subsec:distr_shift}.  
	We derive (a.2) in \eqref{eq:diff_grad_pop_empr} as follows. Let $\bar{\beta}^*$ be an optimal coupling of $(\mu_d,\mu_d^N)$ in the sense of \eqref{eq:W1_dist_disturbance}, i.e., $\int \|d - d'\| \ud \bar{\beta}^*(d,d') = W_1(\mu_d,\mu_d^N)$. Such an optimal coupling exists; see \citet[Theorem~4.1]{villani2009optimal} for a formal proof. Then,
	\begin{align*}
		W_1(\gammass(u),\gammass^N(u)) &= W_1\big((h(u,\cdot),\id)_{\#}\mu_d, (h(u,\cdot),\id)_{\#}\mu_d^N\big) \\
			&\stackrel{\text{(a.1)}}{=} \inf_{\bar{\beta}(d,d') \in \Gamma(\mu_d,\mu_d^N)} \int_{\mathbb{R}^r \times \mathbb{R}^r} c\big((h(u,d),d), (h(u,d'),d')\big) \ud \bar{\beta}(d,d') \\
			&\stackrel{\text{(a.2)}}{\leq} \int_{\mathbb{R}^r \times \mathbb{R}^r} c\big((h(u,d),d), (h(u,d'),d')\big) \ud \bar{\beta}^*(d,d') \\
			&\stackrel{\text{(a.3)}}{\leq} \int_{\mathbb{R}^r \times \mathbb{R}^r} \left(\|h(u,d) - h(u,d')\|_P + \|d - d'\|\right) \ud \bar{\beta}^*(d,d') \\
			&\stackrel{\text{(a.4)}}{\leq} \left(L_h^d\sqrt{\lambda_{\max}(P)}+1\right) \int_{\mathbb{R}^r \times \mathbb{R}^r} \|d - d'\| \ud \bar{\beta}^*(d,d') \\
			&\stackrel{\text{(a.5)}}{=} \left(L_h^d\sqrt{\lambda_{\max}(P)}+1\right) W_1(\mu_d,\mu_d^N),
	\end{align*}
	where (a.1) follows from the property of the pushforward operation \citep[Proposition~3]{aolaritei2022uncertainty}, and the metric $c$ is given by \eqref{eq:metric_mix_norm}; (a.2) holds because $\bar{\beta}^*$ is a specific coupling of $(\mu_d,\mu_d^N)$, i.e., $\bar{\beta}^* \in \Gamma(\mu_d,\mu_d^N)$; (a.3) uses the expression \eqref{eq:metric_mix_norm}; (a.4) leverages the property that $h(u,d)$ is $L_h^d$-Lipschitz in $d$ and \eqref{eq:weight_norm_bd} in \cref{lem:weighted_norm}, \cref{app:lemmas}; (a.5) is true because $\bar{\beta}^*$ is an optimal coupling of $(\mu_d,\mu_d^N)$. Since \eqref{eq:diff_grad_pop_empr} holds for any $u \in \mathbb{R}^n$, we plug in $\{u_0^N, \ldots, u_{T-1}^N\}$ and know that
	\begin{align}\label{eq:general_err_wasser_dist}
		\frac{1}{T} \sum_{k=0}^{T-1} \|\nabla\obj(u_k^N) - \nabla\obj^N(u_k^N)\|^2 &\leq \frac{1}{T} \sum_{k=0}^{T-1} L^2 \left(L_h^d\sqrt{\lambda_{\max}(P)}+1\right)^2 W_1(\mu_d, \mu^N_d)^2 \notag \\
			&= L^2 \left(L_h^d\sqrt{\lambda_{\max}(P)}+1\right)^2 W_1(\mu_d, \mu^N_d)^2.
	\end{align}
	Furthermore, \cref{lem:concentr_measure} ensures that with probability at least $1 - \frac{\tau}{2}$,
	\begin{equation}\label{eq:measure_concen_wasser_dist}
		W_1(\mu_d,\mu_d^N) \leq
		\begin{cases}
			\left(\frac{1}{c_2 N} \ln\left(\frac{2c_1}{\tau} \right)\right)^{\frac{1}{r}}, &\text{if } N \geq \frac{1}{c_2} \ln\left(\frac{2c_1}{\tau} \right), \\
			\left(\frac{1}{c_2 N} \ln\left(\frac{2c_1}{\tau} \right)\right)^{\frac{1}{\theta}}, &\text{if } 1 \leq N < \frac{1}{c_2} \ln\left(\frac{2c_1}{\tau} \right).
		\end{cases}
	\end{equation}
	We combine \eqref{eq:measure_concen_wasser_dist} with \eqref{eq:general_err_wasser_dist} and prove \cref{lem:general_err}.
\end{proof}

\subsection{Proof of \texorpdfstring{\cref{thm:generalization_pop}}{Theorem~\ref*{thm:generalization_pop}}}\label{app:proof_thm_generalization}
\begin{proof}
	The key idea is to start from the decomposition \eqref{eq:pop_grad_decomposition} and then leverage the high-probability bounds on the generalization error (see \cref{lem:general_err}) and the optimization error (see \cref{thm:optimality_nonconvex_high_prob}).

	We exploit \eqref{eq:pop_grad_decomposition} in \cref{subsec:generalization} to obtain
	\begin{align}\label{eq:pop_grad_avg_decomposition}
		\frac{1}{T} \sum_{k=0}^{T-1} \|\nabla\obj(u_k^N)\|^2 \leq \underbrace{\frac{2}{T} \sum_{k=0}^{T-1} \|\nabla\obj(u_k^N) - \nabla\obj^N(u_k^N)\|^2}_{\numcircled{1}} + \underbrace{\frac{2}{T} \sum_{k=0}^{T-1} \|\nabla\obj^N(u_k^N)\|^2}_{\numcircled{2}}.
	\end{align}
	\cref{lem:general_err} provides the following upper bounds on term \numcircled{1} in \eqref{eq:pop_grad_avg_decomposition}:
	\begin{equation}\label{eq:bound_gen_err}
		\numcircled{1} \leq
		\begin{cases}
			2L^2 \left(L_h^d\sqrt{\lambda_{\max}(P)}+1\right)^2 \left(\frac{1}{c_2 N} \ln\left(\frac{2c_1}{\tau} \right)\right)^{\frac{2}{r}}, & \text{if } N \geq \frac{1}{c_2} \ln\left(\frac{2c_1}{\tau} \right), \\
			2L^2 \left(L_h^d\sqrt{\lambda_{\max}(P)}+1\right)^2 \left(\frac{1}{c_2 N} \ln\left(\frac{2c_1}{\tau} \right)\right)^{\frac{2}{\theta}}, & \text{if } 1 \leq N < \frac{1}{c_2} \ln\left(\frac{2c_1}{\tau} \right),
		\end{cases}
	\end{equation}
	which hold with probability at least $1-\tau/2$. For term \numcircled{2} in \eqref{eq:pop_grad_avg_decomposition}, we invoke \cref{thm:optimality_nonconvex_high_prob} and know that with probability at least $1-\tau/2$,
	\begin{align}\label{eq:bound_opt_emp_err}
		\numcircled{2} \leq \underbrace{\frac{32(\obj^N(u_0^N) - \obj^{N,*})}{\eta T} + 48\eta \Lobj^N \sigmamb^2 \left(1+\ln \frac{4}{\tau}\right)}_{\sim \bigO{(1+\ln(1/\tau))/\sqrt{T}}} + \bigO{\frac{1}{T}\left(1 + \ln\frac{1}{\tau}\right)}.
	\end{align}
	Supported by \cref{lem:high_prob_sum_bd}, we combine the upper bounds \eqref{eq:bound_gen_err} and \eqref{eq:bound_opt_emp_err} and prove \cref{thm:generalization_pop}.
\end{proof}

\section{Proof and Additional Results for \texorpdfstring{\cref{sec:experiment}}{Section~\ref*{sec:experiment}}}
\subsection{Proof of \texorpdfstring{\cref{thm:polarized_sens}}{Theorem~\ref*{thm:polarized_sens}}}\label{app:proof_thm_polarized_sens}

We first present a lemma concerning the angles between vectors involving strict conic combinations.

\begin{lemma}\label{lem:conic_comb_angle}
	Consider any nonzero vectors $v_1, v_2 \in \mathbb{R}^m$ that form an acute angle $\theta_{12} \in (0, \frac{\pi}{2})$. Let $v_3 = \lambda_1 v_1 + \lambda_2 v_2$ be a strict conic combination of $v_1$ and $v_2$, where $\lambda_1, \lambda_2 > 0$. Then, the angle $\theta_{23}$ between $v_3$ and $v_2$ is smaller than $\theta_{12}$.
\end{lemma}
\begin{proof}
	Without loss of generality, we consider $v_1$ and $v_2$ of unit norm. If not, we can normalize these vectors and perform a similar analysis. Let $\bar{v}_3 = \barlambda_1 v_1 + \barlambda_2 v_2$ be the normalized vector of $v_3$, where $\barlambda_1 = \lambda_1/\|v_3\|$, and $\barlambda_2 = \lambda_2/\|v_3\|$. It follows that
	\begin{equation}\label{eq:normalized_v3}
		\|\bar{v}_3\|^2 = \barlambda_1^2 + \barlambda_2^2 + 2\barlambda_1 \barlambda_2 v_1^\top v_2 = 1.
	\end{equation}
	Since unit vectors $v_1$ and $v_2$ form an acute angle, we know $0< v_1^\top v_2 < 1$. Therefore,
	\begin{equation*}
		\barlambda_1^2 + \barlambda_2^2 < \barlambda_1^2 + \barlambda_2^2 + 2\barlambda_1 \barlambda_2 v_1^{\top}v_2 < \left(\barlambda_1 + \barlambda_2\right)^2,
	\end{equation*}
	which leads to $\barlambda_1 + \barlambda_2 > 1$ and $0 < \barlambda_1, \barlambda_2 < 1$. Then,
	\begin{align*}
		\cos(\theta_{23}) &- \cos(\theta_{12}) \stackrel{\text{(a.1)}}{=} \left(\barlambda_1 v_1 + \barlambda_2 v_2\right)^\top v_2 - v_1^\top v_2 \stackrel{\text{(a.2)}}{=} \barlambda_2 + (\barlambda_1 - 1) v_1^\top v_2 \\
			&\stackrel{\text{(a.3)}}{=} \barlambda_2 + (\barlambda_1 - 1) \frac{1-\barlambda_1^2-\barlambda_2^2}{2\barlambda_1 \barlambda_2} = \frac{(\barlambda_1 + 1)(\barlambda_1 + \barlambda_2 - 1)(\barlambda_2 + 1 - \barlambda_1)}{2\barlambda_1 \barlambda_2} \stackrel{\text{(a.4)}}{>} 0,
	\end{align*}
	where (a.1) and (a.2) hold because $v_1$, $v_2$, and $\bar{v}_3$ are unit vectors; (a.3) leverages \eqref{eq:normalized_v3} to rewrite $v_1^{\top} v_2$; (a.4) uses the fact that $\barlambda_1$, $\barlambda_1+\barlambda_2-1$, and $\barlambda_2+1-\barlambda_1$ are positive, because $\barlambda_2 + 1 - \barlambda_1 > \barlambda_2 > 0$. It follows from the monotonicity of the cosine function on $[0,\pi]$ that $\theta_{23} < \theta_{12}$. 
\end{proof}

\cref{lem:conic_comb_angle} formalizes the geometric intuition that if a vector lies in a conic region bounded by two vectors forming an acute angle, then this vector is closer to one of the boundary vectors than the two boundary vectors are to each other. We provide the proof of \cref{thm:polarized_sens} as follows.

\begin{proof}
	\revise{We divide the proof into three parts. First, we prove the existence and uniqueness of the trajectory limit $\pss = \hssmap(q, p_0)$ of the dynamics \eqref{eq:polarized_dynamics}. We then derive the steady-state sensitivity associated with $\hssmap(q,p_0)$.} Finally, we quantify how the coefficients $\lambda$ and $\sigma$ in \eqref{eq:polarized_dynamics} influence the closeness between $\pss$ and $q$ (or $-q$), an indicator of the steering ability of the decision-maker.

	\emph{Existence and uniqueness of the limit $\pss$:} Note that the initial state $p_0$ is nonzero, because $\|p_0\|=1$. We first address the special case in \cref{fig:polarize_ss_orthogonal}, i.e., $p_0^\top q = 0$. The dynamics~\eqref{eq:polarized_dynamics} indicate that $p_k = p_0 \neq 0, \forall k \in \mathbb{N}$. Therefore, \revise{the limiting state} exists and is unique: it equals $p_0$. In another special case where $p_0$ and $q$ are parallel, \revise{the unique limit is either $q/\|q\|$ or $-q/\|q\|$, depending on whether $p_0$ and $q$ point in the same or opposite directions.} 

	We proceed to analyze the case in \cref{fig:polarize_ss_acute}, i.e., $p_0^\top q > 0$ and $p_0$ and $q$ are not parallel. \revise{By \eqref{eq:polarized_dynamics}, for any $k\in \mathbb{N}$, $p_k$ lies in the linear span of $p_0$ and $q$.} Let $\sgn(\cdot)$ denote the sign function, i.e., $\sgn(x)$ equals $-1$, $0$, or $1$ if $x<0$, $x=0$, or $x>0$, respectively. The initial condition satisfies $\sgn(p_0^\top q) = 1$. Suppose that for a particular $k \in \mathbb{N}$, $p_k^\top q > 0$ holds. It follows from the dynamics~\eqref{eq:polarized_dynamics} that
	\begin{equation}\label{eq:sgn_inner_prod_dyn}
		\sgn(p_{k+1}^\top q) = \sgn(\tilde{p}_{k+1}^\top q) = \sgn\left((\lambda + \sigma\|q\|^2)p_k^\top q + (1-\lambda)p_0^\top q\right) \stackrel{\text{(a.1)}}{=} 1,
	\end{equation}
	where (a.1) uses the base case $p_0^\top q > 0$ and the induction hypothesis $p_k^\top q > 0$. The above derivation is similar to \citet[Proof of Proposition~1]{dean2022preference}. Hence, $p_{k+1}^\top q > 0$ also holds. Therefore, for any $k \in \mathbb{N}$, $p_k^\top q > 0$, implying $p_k$ is nonzero and that $p_k$ and $q$ always form an acute angle. Let $p_k^+ \triangleq (1-\lambda)p_0 + \sigma \cdot (p_k^{\top} q) q$ be an intermediate vector. We now prove the following proposition on the angles related to $p_k, p_{k+1}$, and $p_k^+$ illustrated by \cref{fig:polarize_intermediate_proof}.

	\begin{figure}[!tb]%
	    \centering
	    \tikzset{every picture/.style={line width=1pt}} 
	    
	    \resizebox{!}{4.5cm}{
	    \begin{tikzpicture}[x=0.75pt,y=0.75pt,yscale=-1,xscale=1]
	    
	    \draw    (110.5,260) -- (153.83,134.86) ;
	    \draw [shift={(154.48,132.97)}, rotate = 109.1] [color={rgb, 255:red, 0; green, 0; blue, 0 }  ][line width=0.75]    (10.93,-3.29) .. controls (6.95,-1.4) and (3.31,-0.3) .. (0,0) .. controls (3.31,0.3) and (6.95,1.4) .. (10.93,3.29)   ;
	    \draw    (110.5,260) -- (266.68,260.04) ;
	    \draw [shift={(268.68,260.05)}, rotate = 180.02] [color={rgb, 255:red, 0; green, 0; blue, 0 }  ][line width=0.75]    (10.93,-3.29) .. controls (6.95,-1.4) and (3.31,-0.3) .. (0,0) .. controls (3.31,0.3) and (6.95,1.4) .. (10.93,3.29)   ;
	    \draw    (110.5,260) -- (122.42,127.39) ;
	    \draw [shift={(122.6,125.4)}, rotate = 95.14] [color={rgb, 255:red, 0; green, 0; blue, 0 }  ][line width=0.75]    (10.93,-3.29) .. controls (6.95,-1.4) and (3.31,-0.3) .. (0,0) .. controls (3.31,0.3) and (6.95,1.4) .. (10.93,3.29)   ;
	    \draw [color={rgb, 255:red, 76; green, 0; blue, 153 }  ,draw opacity=1 ][fill={rgb, 255:red, 208; green, 2; blue, 27 }  ,fill opacity=1 ]   (110.5,260) -- (184.7,150.62) ;
	    \draw [shift={(185.82,148.97)}, rotate = 124.15] [color={rgb, 255:red, 76; green, 0; blue, 153 }  ,draw opacity=1 ][line width=0.75]    (10.93,-3.29) .. controls (6.95,-1.4) and (3.31,-0.3) .. (0,0) .. controls (3.31,0.3) and (6.95,1.4) .. (10.93,3.29)   ;
	    \draw  [dash pattern={on 0.84pt off 2.51pt}]  (110.5,260) -- (185.67,192.31) ;
	    \draw [shift={(187.15,190.97)}, rotate = 137.99] [color={rgb, 255:red, 0; green, 0; blue, 0 }  ][line width=0.75]    (10.93,-3.29) .. controls (6.95,-1.4) and (3.31,-0.3) .. (0,0) .. controls (3.31,0.3) and (6.95,1.4) .. (10.93,3.29)   ;
	    \draw [color={rgb, 255:red, 76; green, 0; blue, 153 }  ,draw opacity=1 ] [dash pattern={on 0.84pt off 2.51pt}]  (110.5,260) -- (262.65,193.11) ;
	    \draw [shift={(264.48,192.3)}, rotate = 156.27] [color={rgb, 255:red, 76; green, 0; blue, 153 }  ,draw opacity=1 ][line width=0.75]    (10.93,-3.29) .. controls (6.95,-1.4) and (3.31,-0.3) .. (0,0) .. controls (3.31,0.3) and (6.95,1.4) .. (10.93,3.29)   ;
	    \draw [color={rgb, 255:red, 0; green, 102; blue, 204 }  ,draw opacity=1 ][fill={rgb, 255:red, 208; green, 2; blue, 27 }  ,fill opacity=1 ]   (110.5,260) -- (224.13,187.38) ;
	    \draw [shift={(225.82,186.3)}, rotate = 147.42] [color={rgb, 255:red, 0; green, 102; blue, 204 }  ,draw opacity=1 ][line width=0.75]    (10.93,-3.29) .. controls (6.95,-1.4) and (3.31,-0.3) .. (0,0) .. controls (3.31,0.3) and (6.95,1.4) .. (10.93,3.29)   ;
	    
	    \draw (258,262.9) node [anchor=north west][inner sep=0.75pt]    {$q$};
	    \draw (100.5,111.9) node [anchor=north west][inner sep=0.75pt]    {$p_{0}$};
	    \draw (143.83,112.23) node [anchor=north west][inner sep=0.75pt]    {$p_{k}$};
	    \draw (170.83,126.9) node [anchor=north west][inner sep=0.75pt]  [color={rgb, 255:red, 76; green, 0; blue, 153 }  ,opacity=1 ]  {$\textcolor[rgb]{0.3,0,0.6}{p_{k+1}}$};
	    \draw (180.5,168.57) node [anchor=north west][inner sep=0.75pt]    {$p_{k}^{+}$};
	    \draw (259.83,168.9) node [anchor=north west][inner sep=0.75pt]  [color={rgb, 255:red, 76; green, 0; blue, 153 }  ,opacity=1 ]  {$p_{k+1}^{+}$};
	    \draw (215.5,164.9) node [anchor=north west][inner sep=0.75pt]  [color={rgb, 255:red, 0; green, 102; blue, 204 }  ,opacity=1 ]  {$\textcolor[rgb]{0,0.4,0.8}{p_{k+2}}$};
	    
	    \end{tikzpicture}
	    }

	    \caption{This figure illustrates the relative positions of $p_k,p_{k+1}$, and $p_k^+$ and the trend that $p_k$ becomes closer to $q$ as $k$ increases, given $p_0^\top q > 0$.}
	    \label{fig:polarize_intermediate_proof}
    \end{figure}

	\begin{proposition}\label{prop:relative_angle}
		If $p_0^\top q > 0$ and $p_0$ and $q$ are not parallel, then the dynamics~\eqref{eq:polarized_dynamics} ensure that for $k \in \mathbb{N}$, the angle between $p_k^+$ and $q$ is smaller than that between $p_{k+1}$ and $q$, which in turn is smaller than that between $p_k$ and $q$.
	\end{proposition}
	
	In the base case when $k=0$, $p_0^+$ is a strict conic combination of $p_0$ and $q$, and $p_1$ is a strict conic combination of $p_0$ and $p_0^+$. Therefore, \cref{lem:conic_comb_angle} ensures that \cref{prop:relative_angle} holds. Suppose that \cref{prop:relative_angle} is true for a particular $k \in \mathbb{N}$. Because $p_k$ and $p_{k+1}$ are of unit norm and the angle between $p_{k+1}$ and $q$ is smaller than that between $p_k$ and $q$, we know that $0 < p_{k}^\top q < p_{k+1}^\top q$. Hence,
	\begin{equation*}
		p_{k+1}^+ = (1-\lambda)p_0 + \sigma\cdot (p_{k+1}^\top q)q \stackrel{\text{(a.1)}}{=} p_k^{+} + \underbrace{\sigma\cdot(p_{k+1}^\top q - p_k^\top q)}_{> 0} q,
	\end{equation*}
	where (a.1) uses the expression of $p_k^+$. Since $p_{k+1}^+$ is a strict conic combination of $p_k^+$ and $q$, 
	we know from \cref{lem:conic_comb_angle} that the angle between $p_{k+1}^+$ and $q$ is smaller than that between $p_k^+$ and $q$, which is in turn smaller than that between $p_{k+1}$ and $q$ (cf.~the induction hypothesis). Finally, $p_{k+2}$ lies in the convex cone formed by $p_{k+1}$ and $p_{k+1}^+$. Hence, \cref{prop:relative_angle} also holds true for $k+1$. Therefore, \cref{prop:relative_angle} is proved.

	A consequence of \cref{prop:relative_angle} is that the angle between $p_k$ and $q$ is monotonically decreasing as $k$ increases. Additionally, this angle is bounded from below by zero, because $p_k$ always lies in the convex cone formed by $p_0$ and $q$. The monotone convergence theorem, together with the fact that $p_k$ stays on the unit sphere inside the cone spanned by $p_0$ and $q$, ensures that there exists \revise{a unique limit $\pss$ given a fixed decision $q$ and an initial position $p_0$}.

	The analysis for the case in \cref{fig:polarize_ss_obtuse} (i.e., $p_0^\top q < 0$) is similar to the above analysis for \cref{fig:polarize_ss_acute} (i.e., $p_0^\top q > 0$) by considering the convex cone formed by $p_0$ and $-q$, or put differently, the angle between $p_k$ and $-q$ for different $k$.

	\emph{Derivation of the steady-state sensitivity:} 
	The trajectory limit $\pss$ of the dynamics \eqref{eq:polarized_dynamics} satisfies
	\begin{subequations}\label{eq:polarized_fixed_point}
	\begin{align}
		\widetilde{p}_\textup{ss} &= \lambda \pss + (1-\lambda) p_0 + \sigma \cdot (\pss^{\top} q) q \triangleq \tilde{f}(\pss,q,p_0), \\
		\pss &= \frac{\widetilde{p}_\textup{ss}}{\|\widetilde{p}_\textup{ss}\|} = \frac{\tilde{f}(\pss,q,p_0)}{\|\tilde{f}(\pss,q,p_0)\|} \triangleq f(\pss,q,p_0). \label{eq:polarized_fixed_point_b}
	\end{align}
	\end{subequations}
	\revise{
	We focus on the general case and consider a pair $(q, p_0)$ with $p_0^\top q \neq 0$. We have proved that $\pss$ and $\widetilde{p}_\textup{ss}$ are nonzero; see the reasoning about \eqref{eq:sgn_inner_prod_dyn}. Further, $f$ in \eqref{eq:polarized_fixed_point_b} is well-defined and differentiable in an open neighborhood of $(q, p_0)$. Hence,}
	\begin{align*}
		\nabla_q f(\pss,q,p_0) &= \sigma \left[(\pss^\top q) I \!+\! \pss q^\top \right] \left(\frac{1}{\|\widetilde{p}_\textup{ss}\|} I \!-\! \frac{\widetilde{p}_\textup{ss} \tilde{p}^\top_\textup{ss}}{\|\widetilde{p}_\textup{ss}\|^3} \right) = \frac{\sigma}{\|\widetilde{p}_\textup{ss}\|} \left[(\pss^\top q) I \!+\! \pss q^\top \right] \left(I \!-\! \pss \pss^{\top}\right), \\
		\nabla_p f(\pss,q,p_0) &= (\lambda I + \sigma q q^\top) \frac{I - \pss \pss^\top}{\|\widetilde{p}_\textup{ss}\|}.
	\end{align*}
	We further apply \eqref{eq:sens_formula} and obtain the steady-state sensitivity as follows
	\begin{align*}
		\revise{\nabla_q \hssmap(q,p_0)} &= -\nabla_q f(\pss,q,p_0) \left[\nabla_p f(\pss,q,p_0) - I\right]^{-1} \\
			&= - \sigma \left(\pss^\top q I + \pss q^\top\right)(I- \pss \pss^\top)\left[(\lambda I + \sigma q q^\top)(I-\pss \pss^\top) - \|\widetilde{p}_\textup{ss}\| I\right]^{-1}.
	\end{align*}
	
	\begin{figure}[!tb]%
	    \centering
		\tikzset{every picture/.style={line width=1pt}} 

		\resizebox{!}{4.5cm}{
		\begin{tikzpicture}[x=0.75pt,y=0.75pt,yscale=-1,xscale=1]

		\draw    (110.5,260) -- (252.05,259.55) ;
		\draw [shift={(254.05,259.55)}, rotate = 179.82] [color={rgb, 255:red, 0; green, 0; blue, 0 }  ][line width=0.75]    (10.93,-3.29) .. controls (6.95,-1.4) and (3.31,-0.3) .. (0,0) .. controls (3.31,0.3) and (6.95,1.4) .. (10.93,3.29)   ;
		\draw    (110.5,260) -- (122.42,127.39) ;
		\draw [shift={(122.6,125.4)}, rotate = 95.14] [color={rgb, 255:red, 0; green, 0; blue, 0 }  ][line width=0.75]    (10.93,-3.29) .. controls (6.95,-1.4) and (3.31,-0.3) .. (0,0) .. controls (3.31,0.3) and (6.95,1.4) .. (10.93,3.29)   ;
		\draw [color={rgb, 255:red, 0; green, 102; blue, 204 }  ,draw opacity=1 ][fill={rgb, 255:red, 208; green, 2; blue, 27 }  ,fill opacity=1 ]   (110.5,260) -- (223.95,174.25) ;
		\draw [shift={(225.55,173.05)}, rotate = 142.92] [color={rgb, 255:red, 0; green, 102; blue, 204 }  ,draw opacity=1 ][line width=0.75]    (10.93,-3.29) .. controls (6.95,-1.4) and (3.31,-0.3) .. (0,0) .. controls (3.31,0.3) and (6.95,1.4) .. (10.93,3.29)   ;
		\draw  [dash pattern={on 4.5pt off 4.5pt}]  (117.55,185.7) -- (208.55,185.55) ;
		\draw  [dash pattern={on 0.84pt off 2.51pt}]  (128.05,247.55) .. controls (132.89,242.92) and (135.68,257.55) .. (131.18,259.55) ;
		\draw  [dash pattern={on 0.84pt off 2.51pt}]  (114.33,219.67) .. controls (133.18,211.55) and (160.68,254.55) .. (155.18,260.55) ;
		\draw  [dash pattern={on 0.84pt off 2.51pt}]  (190.05,185.05) .. controls (182.26,195.92) and (197.55,196.05) .. (193.05,198.05) ;

		\draw (241.67,267.07) node [anchor=north west][inner sep=0.75pt]    {$q$};
		\draw (100.5,111.9) node [anchor=north west][inner sep=0.75pt]    {$p_{0}$};
		\draw (223.17,146.23) node [anchor=north west][inner sep=0.75pt]  [color={rgb, 255:red, 0; green, 102; blue, 204 }  ,opacity=1 ]  {$\pss$};
		\draw (125.33,163.07) node [anchor=north west][inner sep=0.75pt]  [font=\normalsize]  {$\sigma \| q\|^2 \cos( \bar{\theta} )$};
		\draw (71.33,204.73) node [anchor=north west][inner sep=0.75pt]  [font=\normalsize]  {$1-\lambda $};
		\draw (135.33,241.4) node [anchor=north west][inner sep=0.75pt]  [font=\normalsize]  {$\bar{\theta} $};
		\draw (154.17,229.23) node [anchor=north west][inner sep=0.75pt]  [font=\normalsize]  {$\phi $};
		\draw (172.17,189.23) node [anchor=north west][inner sep=0.75pt]  [font=\normalsize]  {$\bar{\theta} $};
		\end{tikzpicture}
		}

		\caption{This figure demonstrates the relative positions of $p_0$, $\pss$, and $q$ and shows that $(1-\lambda) p_0 + \sigma\cdot(\pss^\top q) q$ is collinear with $\pss$.}
	    \label{fig:polarize_steady_state_proof}
	\end{figure}

	\emph{The influence of coefficients:} We quantify how the coefficients $\lambda$ and $\sigma$ in \eqref{eq:polarized_dynamics} determine the steering ability of the decision-maker, which is reflected by the angle $\bar{\theta}$ between $\pss$ and $q$ (or $-q$). We consider the case where $p_0^\top q > 0$ and $p_0$ and $q$ are not parallel, and  similar reasoning applies to the case of $p_0^\top q < 0$. Let $\phi \in (0,90^\circ)$ denote the angle between $p_0$ and $q$. Note that $p_0$ and $\pss$ are of unit norm. Since $\pss$ satisfies the fixed-point equation \eqref{eq:polarized_fixed_point}, the vector $(1-\lambda) p_0 + \sigma\cdot(\pss^\top q) q$ (i.e., $(1-\lambda) p_0 + \sigma\|q\| \cos(\bar{\theta})q$) is collinear with $\pss$; see \cref{fig:polarize_steady_state_proof}. While this illustration and the following analysis are for two-dimensional vectors, extensions can be derived for higher-dimensional settings. We analyze the triangle in \cref{fig:polarize_steady_state_proof} and use the law of sines to obtain
	\begin{equation*}
		\frac{\sigma \|q\|^2 \cos(\bar{\theta})}{\sin(\phi-\bar{\theta})} = \frac{1-\lambda}{\sin(\bar{\theta})} \quad \stackrel{\text{(a.1)}}{\Longrightarrow} \quad \frac{\sin(2\bar{\theta})}{\sin(\phi-\bar{\theta})} = \frac{2(1-\lambda)}{\sigma \|q\|^2},
	\end{equation*}
	where (a.1) invokes the double angle formula $\sin(2\bar{\theta}) = 2\sin(\bar{\theta})\cos(\bar{\theta})$. Note that $\sin(2\bar{\theta})/\sin(\phi-\bar{\theta})$ is an increasing function of $\bar{\theta}$. Increases in $\lambda \in (0,1)$ and $\sigma > 0$ result in a decrease of $\bar{\theta}$, implying a stronger steering ability of the decision-maker.
\end{proof}

\subsection{Effect of approximate sensitivity matrices}\label{app:grad_err_extension}
In \cref{subsec:polarized_pop}, we construct an approximate sensitivity matrix as per \eqref{eq:sens_polarized_approx} based on  samples from the current transient distribution (instead of the limiting steady-state distribution) and the analytical formula \eqref{eq:sens_formula} derived from the implicit function theorem \citep[Theorem~1B.1]{dontchev2009implicit}. When incorporated into the online stochastic update (e.g., \eqref{eq:stochastic_alg}), this approximate sensitivity can introduce an additional gradient error. In this subsection, we characterize this effect and show that the optimality and generalization analysis in the main text still applies.

\revise{We denote the generic steady-state map by $h$ and write $\pss = h(u,d)$; in \cref{subsec:polarized_pop}, $\hssmap(q,p_0)$ plays a similar role.}  
Similar to \eqref{eq:sens_formula}, we consider an open neighborhood of $(u,d)$ and define the sensitivity matrix based on the random variable $p$ as $\nabla_u h(u,d; p) \triangleq -\nabla_u f(p,u,d)\left[\nabla_p f(p,u,d) - I\right]^{-1}$. When evaluated at the steady state $\pss = h(u,d)$, the matrix $\nabla_u h(u,d; \pss)$ is the exact sensitivity at $(u,d)$. We assume that $\nabla_u h(u,d; p)$ is $M_h^p$-Lipschitz in $p$ with respect to the weighted norm $\|\cdot\|_P$ and $M_h^d$-Lipschitz in $d$ with respect to $\|\cdot\|$. These conditions hold when the relevant first and second derivatives of the dynamics function $f$ in \eqref{eq:pop_dynamics}, together with the inverse $[\nabla_p f(p,u,d) - I]^{-1}$, are uniformly bounded in the aforementioned open neighborhood. 

We analyze the setting where we use the approximate sensitivity $\nabla_u h(u_k,d; p_k)$ with $p_k$ at time $k$ in place of the exact sensitivity $\nabla_u h(u_k,d)$ in iterative updates. 
For a given decision $u_k \in \mathbb{R}^n$, let
\begin{equation}\label{eq:cur_distr_sens_stoch_grad}
    \bar \nabla \obj(u_k) \triangleq \E_{(p,d) \sim \gamma_k}[\nabla_u \Phi(u_k,p) + \nabla_u h(u_k,d; p) \nabla_p \Phi(u_k,p)]
\end{equation}
be another approximate expected gradient at $u_k$. In \eqref{eq:cur_distr_sens_stoch_grad}, $p$ and $d$ follow the joint distribution $\gamma_k$, and the sensitivity is constructed from the random variable at time $k$. Recall from \eqref{eq:ss_distr_stoch_grad} that $\nabla \obj(u_k)$ is the exact gradient of the objective $\obj$ at $u_k$. We extend \cref{lem:grad_err_bound} and show that the resulting gradient difference $\bar{e}_k \triangleq \bar \nabla \obj(u_k) - \nabla \obj(u_k)$ admits an upper bound proportional to the Wasserstein distance $W_1(\gamma_k,\gammass(u_k))$ between the joint distributions $\gamma_k$ and $\gammass(u_k)$, albeit with a different multiplicative constant. The following extension provides \emph{local} characterizations, since the sensitivity formula \eqref{eq:sens_formula} involving $\pss$ holds in an open neighborhood of $(u, d)$.

\begin{lemma}\label{lem:grad_err_bound_extension}
    Consider an open neighborhood of $(u, d)$. Suppose that \cref{assump:stable_system,assump:obj_property,assump:obj_rand_Lipschitz} are satisfied. We further assume that $\nabla_u h(u,d; p)$ is $M_h^p$-Lipschitz in $p$ with respect to the weighted norm $\|\cdot\|_P$ and $M_h^d$-Lipschitz in $d$ with respect to $\|\cdot\|$ in this neighborhood. Then,
    \begin{equation}\label{eq:grad_err_bd_wasserstein_extension}
		\|\bar{e}_k\| = \|\bar\nabla \obj(u_k) - \nabla \obj(u_k)\| \leq \bar{L} W_1(\gamma_k,\gammass(u_k)),
	\end{equation}
	where $\bar{L} = \max\left(L_{\Phi}^p M_h^d, \frac{M_\Phi^u \!+\! L_h^u M_\Phi^p}{\sqrt{\lambda_{\min}(P)}} \!+\! L_\Phi^p M_h^p \right)$.
\end{lemma}
\begin{proof}
    The proof is similar to that of \cref{lem:grad_err_bound} in \cref{app:proof_lem_grad_err}. However, we additionally address the error of $\nabla_u h(u,d; p)$ relative to the exact sensitivity. Suppose that $\bar{\Psi}(p,d) \triangleq \nabla_u \Phi(u_k,p) + \nabla_u h(u_k,d; p)\nabla_p \Phi(u_k,p)$. Then, $\E_{(p,d)\sim \gamma_k}[\bar{\Psi}(p,d)] = \bar \nabla \obj(u_k)$, and $\E_{(p,d)\sim \gammass}[\bar{\Psi}(p,d)] = \nabla \obj(u_k)$. 
    For any $p_1,p_2 \in \mathbb{R}^m$ and $d_1,d_2 \in \mathbb{R}^r$,
	\begin{align}\label{eq:Psi_Lipschitz}
		\|\bar{\Psi}(p_1,d_1) -& \bar{\Psi}(p_2,d_2)\| \leq \|\bar{\Psi}(p_1,d_1) - \bar{\Psi}(p_1,d_2)\| + \|\bar{\Psi}(p_1,d_2) - \bar{\Psi}(p_2,d_2)\| \notag \\
			\leq& \|\nabla_u h(u_k,d_1; p_1) - \nabla_u h(u_k,d_2; p_1)\|\,\|\nabla_p \Phi(u_k,p_1)\| + \|\nabla_u \Phi(u_k,p_1) - \nabla_u \Phi(u_k,p_2)\| \notag \\
			&+ 
            \|\nabla_u h(u_k,d_2;p_1) \nabla_p \Phi(u_k,p_1) - \nabla_u h(u_k,d_2;p_2) \nabla_p \Phi(u_k,p_2)\|
            \notag \\
			\stackrel{\text{(a.1)}}{\leq}& L_\Phi^p M_h^d \|d_1 \!-\! d_2\| + M_\Phi^u\|p_1 \!-\! p_2\| + \|\nabla_u h(u_k,d_2;p_1)\| \| \nabla_p \Phi(u_k,p_1) \!-\! \nabla_p \Phi(u_k,p_2)\| \notag \\
            &+ \|\nabla_u h(u_k,d_2;p_1) - \nabla_u h(u_k,d_2;p_2)\| \|\nabla_p \Phi(u_k,p_2) \| \notag \\
            \leq& L_\Phi^p M_h^d \|d_1 - d_2\| + M_\Phi^u\|p_1 - p_2\| + L_h^u M_\Phi^p \|p_1 - p_2\| + L_\Phi^p M_h^p \|p_1 - p_2\|_P \notag \\
			\leq& \bar{L}(\|p_1-p_2\|_P + \|d_1-d_2\|),
	\end{align}
    where $\bar{L}$ is specified in \cref{lem:grad_err_bound_extension}.
    In (a.1) of \eqref{eq:Psi_Lipschitz}, we leverage the Lipschitz continuity of $\nabla_u h(u_k,d;p)$, $\Phi(u_k,p)$, and $\nabla_u \Phi(u_k,p)$ and apply the triangle inequality to decompose the last term in the preceding inequality. Furthermore, we exploit \cref{lem:function_closeness} in \cref{app:lemmas} to obtain \eqref{eq:grad_err_bd_wasserstein_extension}. Hence, \cref{lem:grad_err_bound_extension} is proved.
\end{proof}

\cref{lem:grad_err_bound_extension} characterizes the difference between the approximate gradient $\bar \nabla \obj(u_k)$ and the exact gradient $\nabla \obj(u_k)$ arising from the use of real-time samples to evaluate the sensitivity matrix based on \eqref{eq:sens_formula}. This characterization enables the subsequent theoretical analysis in \cref{subsec:optimality,subsec:optimality_hp,subsec:generalization} to be applied in a similar manner. 


